\documentclass[a4paper,twoside,10pt]{article}
\usepackage[english]{babel}
\usepackage[latin1]{inputenc}
\usepackage{lmodern,color}
\usepackage[T1]{fontenc}
\usepackage{indentfirst}
\usepackage{amsmath,amssymb}

\usepackage[a4paper,top=3cm,bottom=3cm,left=3cm,right=3cm,%
bindingoffset=5mm]{geometry}
\usepackage{lmodern}
\usepackage[T1]{fontenc}
\usepackage{lettrine}
\usepackage{booktabs}
\usepackage{authblk}
\usepackage{emp}
\usepackage{amsthm}
\usepackage{amsmath,amssymb,amsthm}
\usepackage{braket}
\usepackage{chngpage}
\usepackage{cases}
\usepackage{graphicx}
\usepackage{epstopdf}
\usepackage{tabularx}
\usepackage{multirow}

\newcommand{\numberset}{\mathbb}

\newcommand{\R}{\numberset{R}}

\newcommand{\B}{\numberset{B}}
\newcommand{\Pk}{\numberset{P}}
\renewcommand{\epsilon}{\varepsilon}
\renewcommand{\theta}{\vartheta}
\renewcommand{\rho}{\varrho}
\renewcommand{\phi}{\varphi}
\newcommand{\nn}{\boldsymbol{n}}
\newcommand{\ttt}{{\boldsymbol{t}}}
\newcommand{\uu}{\boldsymbol{u}}
\newcommand{\vv}{\boldsymbol{v}}
\newcommand{\ww}{\boldsymbol{w}}
\newcommand{\ff}{\boldsymbol{f}}

\newcommand{\dd}{{\rm div}}
\newcommand{\DD}{\boldsymbol{\rm div}}
\newcommand{\gr}{\nabla}
\newcommand{\Gr}{\boldsymbol{\nabla}}

\newcommand{\VV}{\boldsymbol{X}}
\def\VVG{\boldsymbol{X}_{\Gamma}}

\newcommand{\oo}{\boldsymbol{0}}

\def\P0{{\Pi_1^{0, E}}}
\def\PN{{\Pi_1^{\Gr^s, E}}}
\def\PP0{{\boldsymbol{\Pi}_0^{0, E}}}

\newcommand{\D}{{\cal D}}
\newcommand{\xv}{\boldsymbol{x}}
\newcommand{\ov}{{\boldsymbol{o}}}
\newcommand{\pv}{{\boldsymbol{p}}}
\newcommand{\T}{\boldsymbol{T}(\boldsymbol{u},p)}
\newcommand{\Tz}{\boldsymbol{T}(\boldsymbol{z},q)}
\newcommand{\Tup}{\boldsymbol{T}(\boldsymbol{u}',p')}
\newcommand{\Tzp}{\boldsymbol{T}(\boldsymbol{z}',q')}
\newcommand{\e}{\boldsymbol{e}}
\newcommand{\OG}{\Omega}
\newcommand{\Fi}{\boldsymbol{\Phi}}
\newcommand{\zz}{\boldsymbol{z}}
\newcommand{\VVe}{\boldsymbol{V}}
\newcommand{\Ve}{{V}}

\newcommand{\Ph}{{\cal P}_h(\D)}
\newcommand{\PhG}{{\cal P}_{h,\Gamma}(\Omega)}

\def\EG{\boldsymbol{\epsilon}_{\Gamma}}

\DeclareMathOperator*{\infimum}{inf\phantom{p}\!\!\!}
\newcommand{\infsup}[2]{\infimum_{#1}\sup_{#2}}

\def\VVGh{\VV_{h,\Gamma}}

\newcommand{\jump}[1]{\lbrack\!\lbrack\,#1\,\rbrack\!\rbrack}

\usepackage{graphicx}
\usepackage{booktabs}
\usepackage{tabularx}
\usepackage{multirow}
\usepackage{enumerate}
\usepackage{color}

\usepackage[percent]{overpic}
\usepackage{ulem}

\newtheorem{remark}{Remark}[section]
\newtheorem{proposition}{Proposition}[section]
\newtheorem{corollary}{Corollary}[section]
\newtheorem{lemma}{Lemma}[section]
\newtheorem{theorem}{Theorem}[section]

\begin{document}

\author[1,4]{L. Beir\~ao da Veiga \thanks{lourenco.beirao@unimib.it}}

\author[2]{C. Canuto \thanks{claudio.canuto@polito.it}}

\author[3]{R. H. Nochetto \thanks{rhn@math.umd.edu}}

\author[1]{G. Vacca \thanks{giuseppe.vacca@unimib.it}}


\affil[1]{Dipartimento di Matematica e Applicazioni,  Universit\`a degli Studi di Milano Bicocca, Via R. Cozzi 55 - 20125 Milano, Italy}  

\affil[2]{Dipartimento di Scienze Matematiche G.L. Lagrange,  Politecnico di Torino, Corso Duca degli Abruzzi 24 - 10129 Torino, Italy} 

\affil[3]{Department of Mathematics and Institute for Physical Science and Technology, University of Maryland, College Park - 20742, MD, USA} 

\affil[4]{IMATI-CNR, Via Ferrata 1, I-27100 Pavia, Italy}

\title{Equilibrium analysis of an immersed rigid leaflet \\ by the virtual element method}
\date{}

\maketitle

\begin{abstract}
We study, both theoretically and numerically, the equilibrium of a hinged rigid leaflet with an attached rotational spring, immersed in a stationary incompressible fluid within a rigid channel. 
Through a careful investigation of the properties of the functional describing the angular momentum exerted by the fluid on the leaflet (which depends on both the leaflet angular position and its thickness), we identify sufficient conditions on the spring stiffness function for the existence (and uniqueness) of equilibrium positions. 
We propose a numerical technique that exploits the mesh flexibility of the Virtual Element Method (VEM). A (polygonal) computational mesh is generated by cutting a fixed background grid with the leaflet geometry, and the problem is then solved with stable VEM Stokes elements of degrees $1$ and $2$ combined with a bisection algorithm. We present a large array of numerical experiments to document the accuracy and robustness with respect to degenerate geometry of the proposed methodology.
\end{abstract}



\section{Introduction}
\label{sec:into}

The Virtual Element Method (VEM) is a recent numerical technology introduced in \cite{volley,autostop} for the discretization of problems governed by partial differential equations. It can be regarded as a generalization of the Finite Element Method (FEM) to meshes of general polytopes. Since its inception, the VEM enjoyed a wide success in the mathematics and engineering communities, because of its flexibility and robustness with respect to mesh design and handling. 
To cite some applicative example, VEM allows for immediate gluing of independent planar meshes in discrete fracture network simulation \cite{Berrone,Benedetto,Keilegaven}, adding nodes to ease the enforcement of contact conditions in solid mechanics \cite{wriggers}, breaking of existing elements for crack propagation problems \cite{Aldakheel,Artioli-cr,Benedetto-cr,cinesi-cr}, and reduction of directional mesh bias in topology optimization \cite{Paulino-topopt,Bruggi-Verani}.

The class of fluid-structure and immersed boundary problems is practically relevant and offers attractive possibilities and challenges to VEM and general polytopal meshes due to the interaction of different, perhaps deforming, domains\cite{DeHart}. 
A very short list of representative papers, restricting the attention to the case of interaction with a rigid body, is \cite{fsi-rig-1,fsi-rig-2,fsi-rig-3}.
One could think, for instance, of using a fixed background grid for the fluid domain (Eulerian description), which is cut by a deforming solid at each time instant or iterative procedure step (Lagrangian description).
Clearly, arbitrary mesh cuttings may generate polygonal elements of very bad quality (in terms of element anisotropy, possible non convexity, neighbor size ratio, etc.) and thus the numerical scheme must be reliable also in the presence of such hazards. We refer to \cite{Antonietti-FSI} for an application in the realm of polygonal DG.
The present paper represents a first VEM study in this setting.
%
\begin{figure}[!h]
\center
{
\begin{overpic}[scale=0.33]{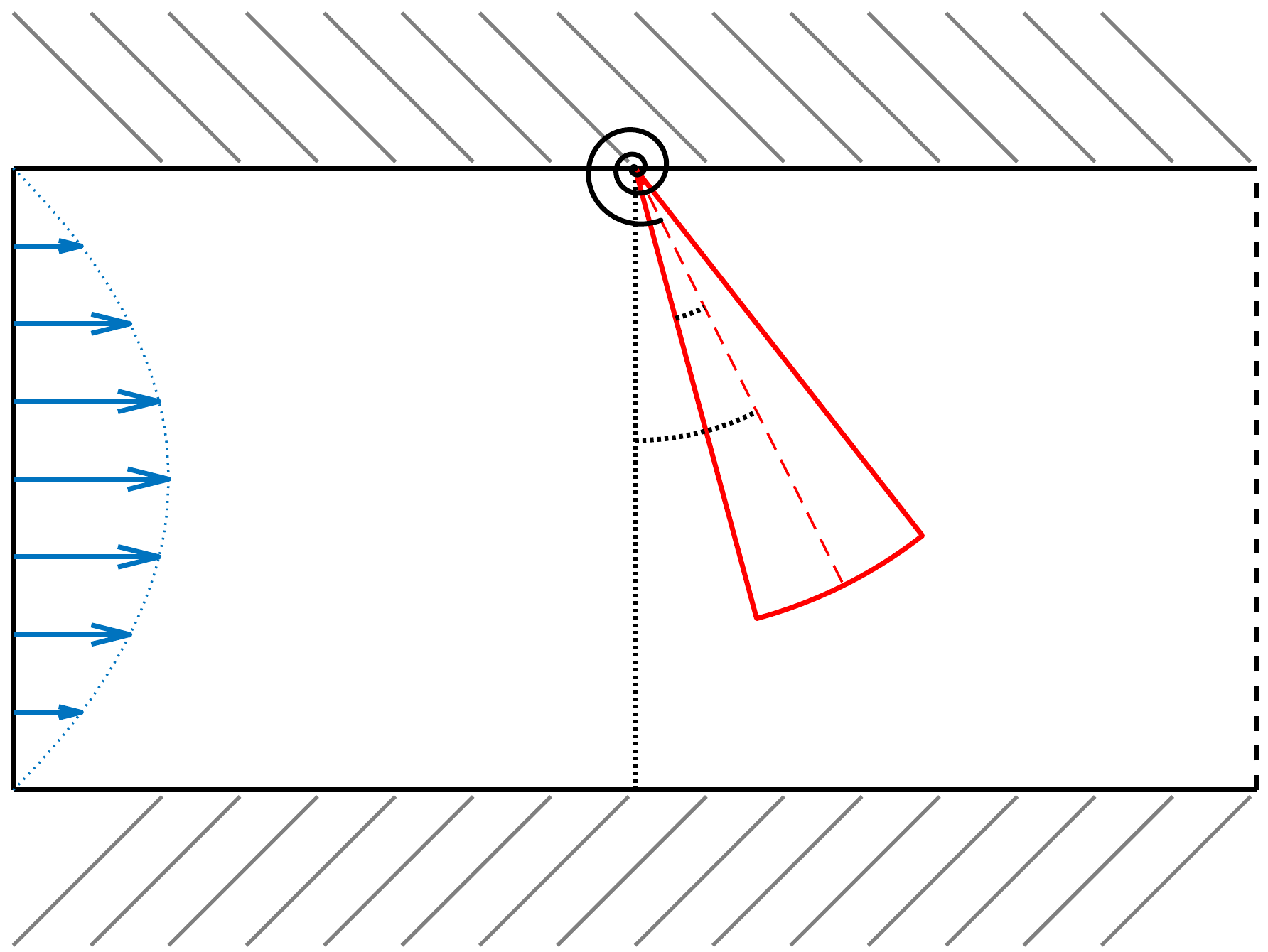} 
\put (68,40) {\Large{$\Gamma$}}
\put (54,34) {\Large{$\theta$}}
\put (55,46) {\large{$\epsilon$}}
\end{overpic}
\hfill
\begin{overpic}[scale=0.33]{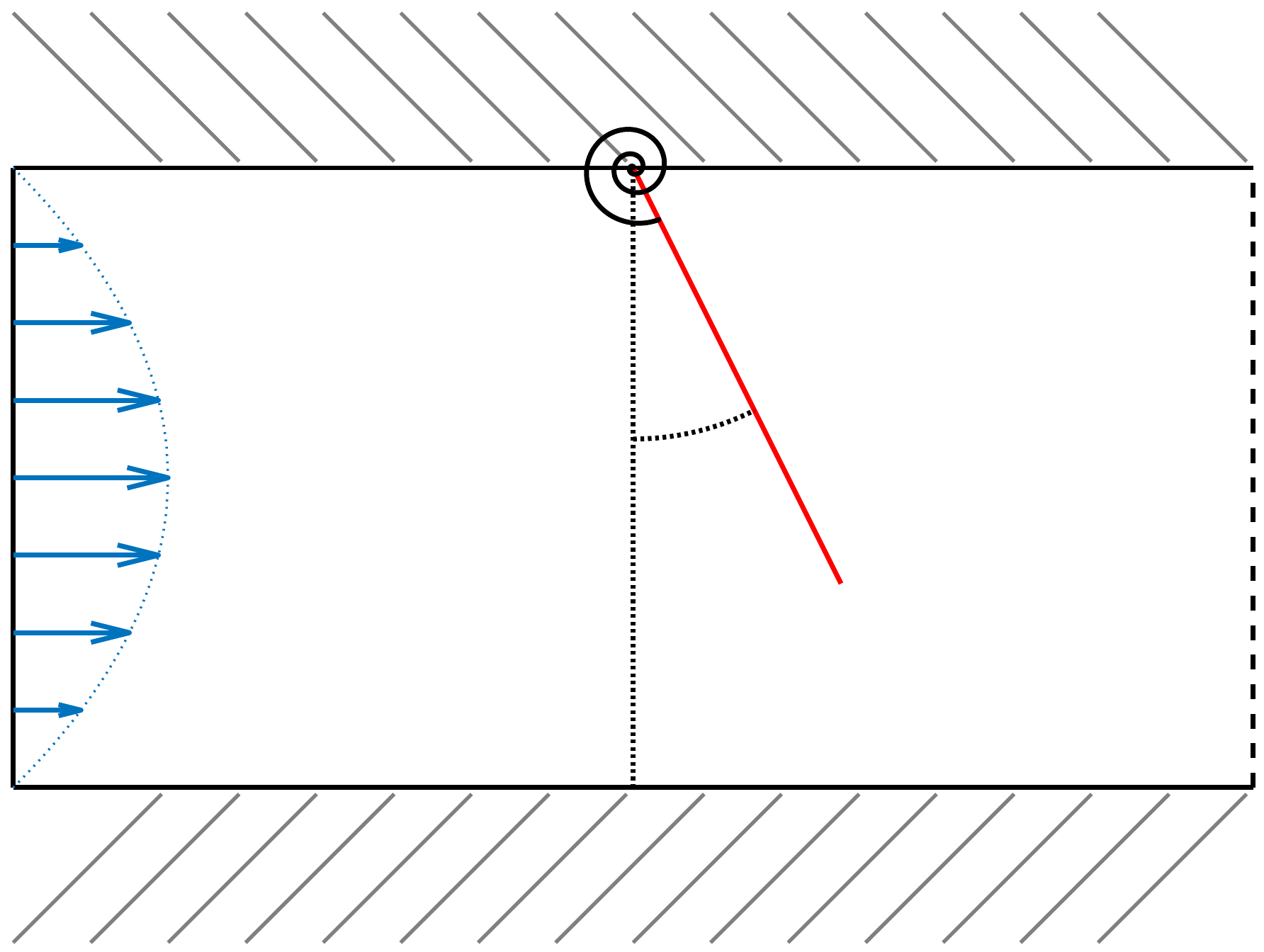} 
\put (65.5,30.5) {\Large{$\Gamma$}}
\put (54,34) {\Large{$\theta$}}
\end{overpic}
  \caption{Model problem: hinged rigid leaflet of thickness $\epsilon$ with a rotational spring attached immersed in a stationary incompressible fluid within a rigid channel. (a) fat leaflet ($\epsilon>0$), (b) thin leaflet ($\epsilon=0$).}
  \label{F:leaflet-model}
}
\end{figure}
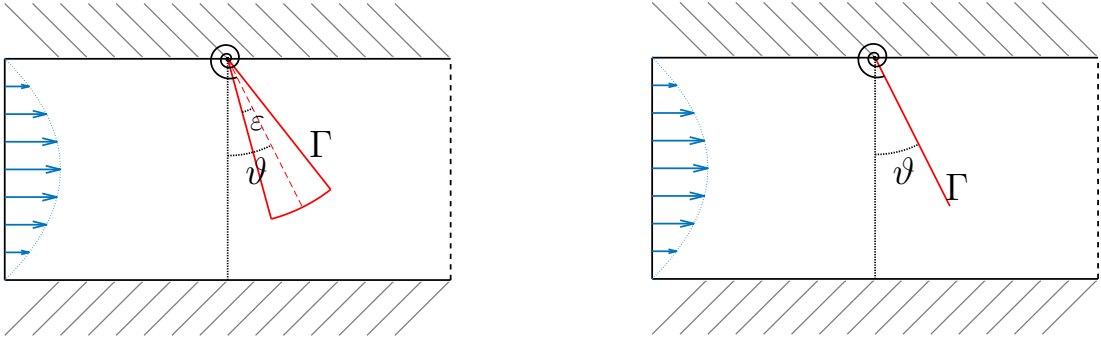
In order to design a VEM for an immersed boundary problem and study its robustness and accuracy, we propose a ``deceivingly simple'' 2D model problem inspired from the FEM analysis in \cite{Auricchio}. The problem is that of a hinged rigid structure (a leaflet) of thickness $\epsilon$ with a rotational spring attached, immersed in a stationary incompressible fluid within a rigid channel; all data are constant in time. We consider two extreme cases: the {\it fat leaflet} ($\epsilon>0$), which is the most physically realistic case, and the {\it thin leaflet} ($\epsilon=0$), which is the asymptotic limit of the former; see Figure \ref{F:leaflet-model} for a cartoon geometry of the problem.
In order to prevent interactions between the leaflet and the rigid upper wall, which in turn are not physically relevant and would lead to additional singularities, we introduce a positive parameter $\epsilon_0>0$ and assume $\epsilon \le \epsilon_0/2$. We then investigate our problem for the admissible angle range $\theta \in I_{\epsilon_0}$ where
\begin{equation}\label{eq:Iepsilon0}
I_{\epsilon_0} := [-\tfrac\pi2+\epsilon_0,\tfrac\pi2-\epsilon_0].
\end{equation}
The equilibrium position of the leaflet corresponds to a balance between the angular momentum $-\kappa(\theta)$ exerted by the rotational spring on the leaflet and the functional 
$$
\tau \ : \ I_{\epsilon_0} \rightarrow {\mathbb R}
$$
describing the torque exerted by the fluid on the leaflet as a function of its angular position $\theta$ relative to the vertical axis; hence the torque balance reads $\tau(\theta)=\kappa(\theta)$.

This problem is nonlinear because $\tau$ depends on $\theta$ in an intricate nonlinear fashion, in fact one that we decipher in this paper. Its numerical approximation requires nonlinear iterations and thus entails solving the fluidodynamics problem for several arbitrary leaflet positions.

This leads to various fundamental issues, both theoretical and computational, that have to be resolved to get a reliable and accurate numerical method. We describe them along with our contributions below.

\begin{enumerate}[$\bullet$]

\item {\it Structure of $\tau(\theta)$}. If $(r,\omega)$ are polar coordinates relative to the hinge, $\T$ is the Cauchy tensor of the fluid, written in terms of the velocity-pressure pair $(\uu,p)$, $\Gamma=\Gamma(\theta)$ is the boundary of the leaflet, and $\nn_\Gamma$ is the outer unit normal of $\Gamma$, then the force per unit of length acting on $\Gamma$ due to the leaflet is given by $\T\nn_\Gamma$. Consequently, the torque $\tau(\theta)$ exerted by the fluid on $\Gamma$ reads
\begin{equation}\label{eq:tau}
\tau(\theta) = - \int_\Gamma r \, \e_\omega^\perp \cdot \T \, \nn_\Gamma
\end{equation}
where $\e_\omega^\perp=(\cos\omega,\sin\omega)$. Since the function $\tau(\theta)$ is nonlinear and nonlocal, because $(\uu,p)$ depends on $\Gamma$, no explicit expression is available. To study the behavior of $\tau$, namely show differentiability for $\epsilon>0$ and uniform continuity for $\epsilon=0$, we resort to shape differential calculus \cite{DelfourZolesio,SokolowskiZolesio}. However, this is not straightforward because $\Gamma$ is not smooth, especially for $\epsilon=0$. We rewrite \eqref{eq:tau} as a sum of bulk integrals and compute its shape gradient, thereby avoiding dealing with curvature of $\Gamma$ which is not well defined. This is possible for $\epsilon>0$ but the underlying regularity becomes borderline for $\epsilon=0$ and we content ourselves with continuity of $\tau(\theta)$. To accomplish this program, we represent \eqref{eq:tau} variationally and use a duality argument involving an adjoint fluid system with solution $(\zz,q)$. The representation \eqref{eq:expr2-J} below of \eqref{eq:tau} is amenable to shape differentiation: we deform $\Gamma$ rigidly preserving both its shape and fluid incompressibility. We carry out this study in Section \ref{sect:shape} and obtain explicit expressions of the derivative $\tau_\epsilon'(\theta)$ in terms of $(\uu,p)$, $(\zz,q)$ and data for $\epsilon>0$. Moreover, we prove that $\tau_\epsilon$ converges uniformly as $\epsilon\to0$ and gives rise to a continuous torque for the thin leaflet. Existence (and uniqueness) of the nonlinear equilibrium equation
\begin{equation}\label{eq:equilibrium}
  \tau(\theta)=\kappa(\theta)
\end{equation}
follow upon making suitable assumptions on the spring angular momentum $\kappa$.

\medskip
\item {\it VEM discretization}.
  The iterative solution of \eqref{eq:equilibrium} requires solving the fluidodynamics problem for different and arbitrary positions of the leaflet $\Gamma$. We consider the thin leaflet $\Gamma$, i.e. $\epsilon=0$, and let $\Gamma$ cut through a background uniform grid of quadrilaterals. When the tip of $\Gamma$ falls within an element $E$, we extend $\Gamma$ with a straight line until it hits the boundary of $E$; this procedure is more accurate than dealing with the tip within $E$, which is also a viable option in the VEM context. The resulting mesh is thus geometrically conforming to $\Gamma$ but at the expense of having sometimes polygons with extremely degenerate shapes depending on the angle $\theta$: highly anisotropic elements, elements with edges that are orders of magnitude smaller than its diameter, and elements that are orders of magnitude smaller than their neighbors. The former are typical of small $\theta$'s whereas the latter typically occur for intermediate $\theta$'s. Moreover, near certain critical $\theta$'s, even small variations of the leaflet position may yield abrupt topological changes in the mesh due to the extension procedure. We exploit the capabilities of VEM to handle arbitrary polygonal elements seamlessly. We adopt the divergence-free VEM of degree $k=1,2$ for the Stokes fluid \cite{Antonietti-BeiraodaVeiga-Mora-Verani:2014,Stokes:divfree,Vacca:2018,NavierStokes:divfree}. We investigate the approximation properties of the ensuing discrete torque $\tau_h(\theta)$ and prove a quasi-optimal error estimate relative to $\tau(\theta)$, uniform in $\theta$.

\medskip
\item {\it Computational study}.
We develop a series of numerical tests to assess and document the performance of the VEM methodology in the setting of an immersed rigid boundary. We illustrate the effect of degenerate elements in the inf-sup constant and conditioning of the system for a wide range of angles $\theta$. Geometric degeneracy is usually associated with manageable spikes in both quantities. We perform a study of the role of the stabilization term of VEM. It turns out that the effect on $\tau_h(\theta)$ of abrupt topological changes of the mesh is much more pronounced for the so-called ``dofi-dofi'' stabilization form \cite{volley} than for the ``trace'' stabilization \cite{wriggers} form. The former is, however, generally more accurate than the latter. We examine this unexpected discovery in great length and present several experiments whose main parameters are the angle $\theta$ and the mesh size $h$. From the practical perspective, we conclude that, although there is some influence of the mesh quality on the results, the scheme is sufficiently robust and reliable. Considering the simplicity, and thus the efficiency, of the mesh cutting procedure when compared with other techniques, we believe our approach is viable. 
  
\end{enumerate}

The paper is organized as follows. In Section \ref{sec:problem} we present the model problem and its adjoint along with their variational formulation. In Section \ref{sect:shape} we develop the theoretical analysis of the continuous problem, namely prove properties of $\tau(\theta)$ for $\epsilon>0$ and $\epsilon=0$ that are uniform in $\theta$. In Section \ref{sec:VEMdis} we briefly review the VEM method of \cite{NavierStokes:divfree} and describe the discrete torque functional $\tau_h$, the mesh cutting procedure, and the adopted iterative scheme. We also derive an error estimate for $\tau-\tau_h$. Finally, in Section \ref{sec:tests} we document the performance and accuracy of the proposed scheme relative to degenerate elements and abrupt topological mesh transitions. Moreover, we report on variations of the inf-sup constant and condition number with respect to the angle $\theta$ and discuss robustness.

\section{Problem definition and governing equations}
\label{sec:problem}

The focus of the present study is to analyse the problem of a hinged thin rigid structure $L$ (a leaflet) with a rotational spring attached, immersed in a fluid within a rigid channel.
We assume invariance in the transversal direction $z$, hence we can adopt a 2D model in the $xy$ plane. Furthermore, we consider the stationary case, that is all problem data are independent of time, and we search for the equilibrium position of the leaflet as well as the corresponding fluid velocity and pressure. 

We assume that the channel is represented by a rectangle $\D$ aligned with the coordinate axes, with the upper and lower edges corresponding to the rigid walls. The leaflet is hinged at a point $\ov \in \partial\D$ sitting on the upper wall of the channel, as depicted in Fig. \ref{F:leaflet-model}. Introducing a system of polar coordinates $(r,\omega)$ centered at $\ov$ with principal ray $\omega=0$ placed vertically and oriented downward, we let
\begin{equation}\label{eq: eomega}
\e_\omega := (\sin \omega, -\cos \omega),
\end{equation}
and note that a generic point $\xv=(x,y)\in\D$ reads $\xv = \ov + r \e_\omega$. We assume that $L$ is axi-symmetric with respect to some axis passing through $\ov$; let $\theta \in I_{\epsilon_0}$ be the angular coordinate of such axis of symmetry, where $I_{\epsilon_0}$ is defined in \eqref{eq:Iepsilon0}. Thus, the position of $L$ is identified by the value of $\theta$. We denote by $\Gamma = \partial L$ the boundary of the leaflet, and by $\nn=\nn_{\Gamma}$  the unit normal vector to $\Gamma$ pointing inside $L$. The region occupied by the fluid (i.e., the computational domain) is $\Omega:= \D \! \setminus \! L$, whose boundary is $\partial\Omega = \partial \D \cup \Gamma$; the fluid-structure interaction takes place on $\Gamma$. 

While the formulation of the problem will be given for a generic leaflet, we will further develop our analysis for leaflets $L$ of the form
\begin{equation}\label{eq: leaflet}
  L := \{ \xv = \ov + r \e_\omega: ~ 0 \le r \le R, \ |\omega- \theta| \le \epsilon \},
\end{equation} 
for some $R$ smaller than the vertical size of the channel, and some $\epsilon$ satisfying $0 \le \epsilon \le \epsilon_0/2$ small enough.
We further assume that the channel length is sufficient to guarantee that the distance among $L$ and the vertical sides of ${\cal D}$ is (uniformly) positive for all configurations.
Note that we do allow the limit case $\epsilon = 0$, when the 2D leaflet degenerates into a 1D segment or, equivalently, when $\Gamma=L$. 
Fig. \ref{F:leaflet-notations} displays the two cases, with the associated notation.
When needed, we will append the suffix $\epsilon$ to symbols, to stress their dependence upon $\epsilon$ (e.g., $\Omega_\epsilon$, $\Gamma_\epsilon$, ...).

\begin{figure}[!h]
\center
{
\begin{overpic}[scale=0.35]{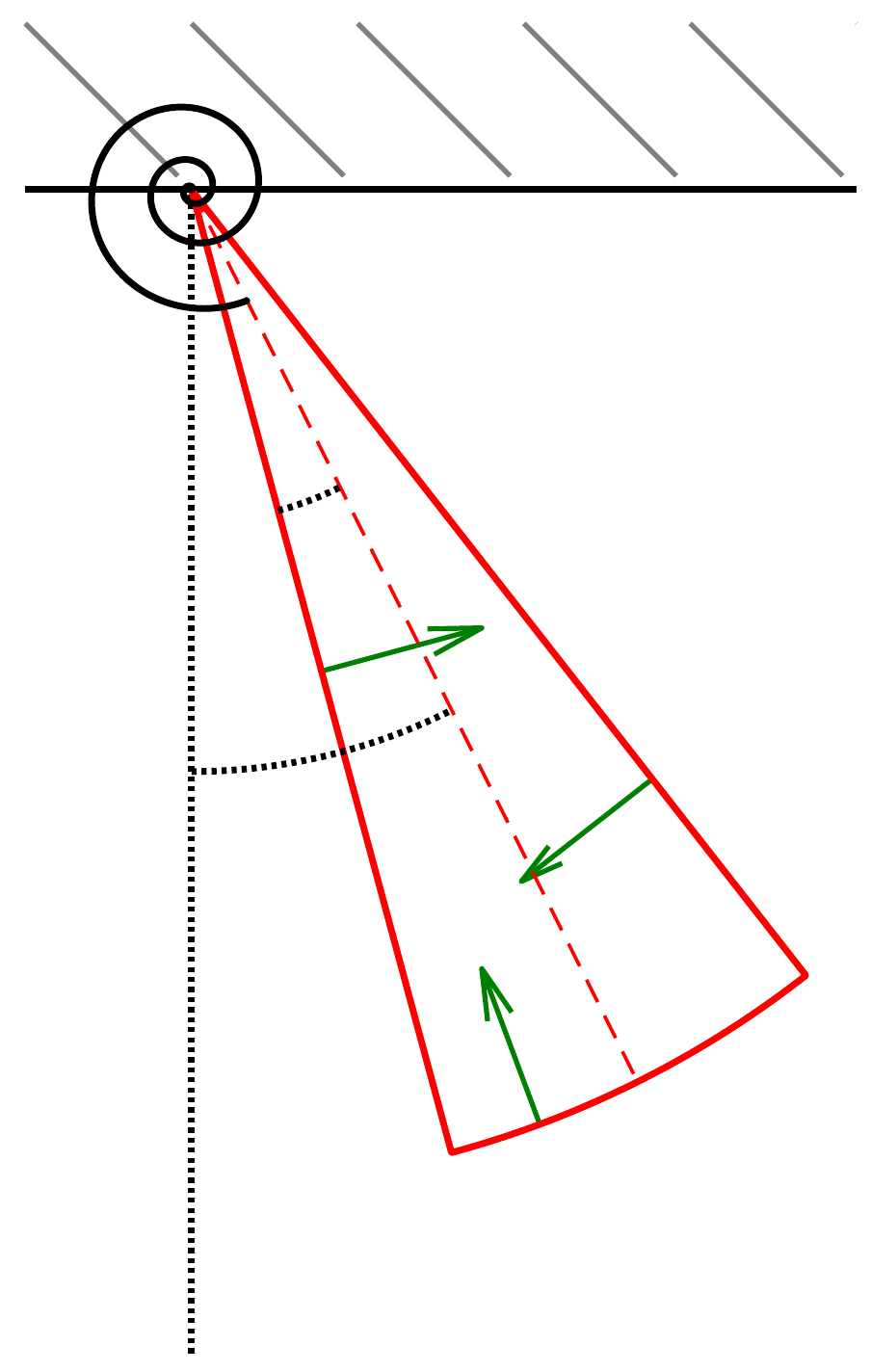} 
\put (10,94) {\Large{$\ov$}}
\put (43,25) {\Huge{$L$}}
\put (43,50) {\LARGE{$\Gamma$}}
\put (3,75) {\Large{$\kappa$}}
\put (22,36) {\LARGE{$\theta$}}
\put (36,12) {{$\boldsymbol{n}_{\Gamma}$}}
\put (22,57) {\Large{$\epsilon$}}
\end{overpic}
\qquad \qquad  \qquad \qquad
\begin{overpic}[scale=0.35]{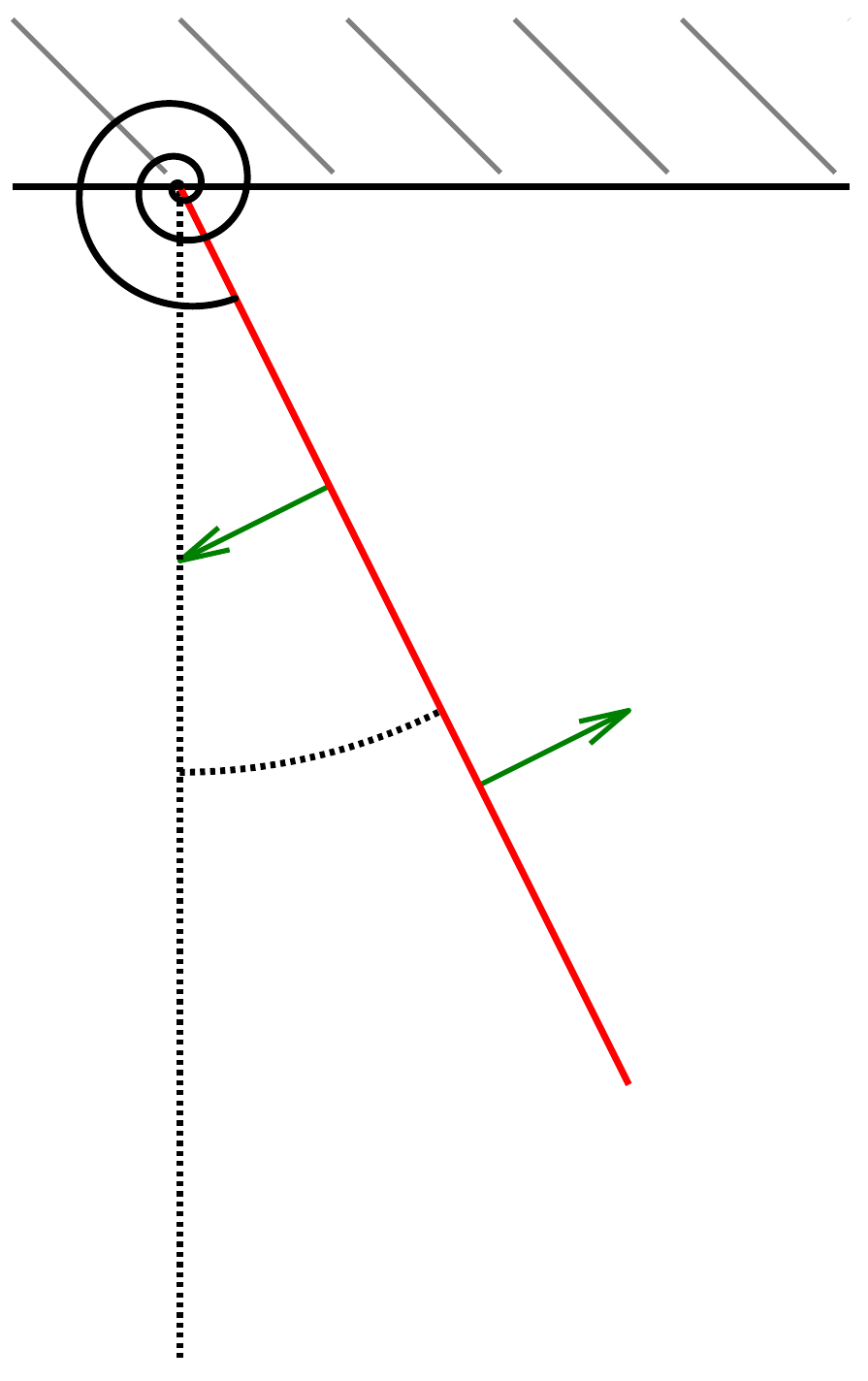} 
\put (10,94) {\Large{$\ov$}}
\put (43,30) {\LARGE{$\Gamma$}}
\put (3,75) {\Large{$\kappa$}}
\put (22,36) {\LARGE{$\theta$}}
\put (40,50) {{$\boldsymbol{n}_{\Gamma}$}}
\end{overpic}
\caption{Notations: hinged rigid leaflet of thickness $\epsilon$ with a rotational spring attached immersed in a stationary incompressible fluid within a rigid channel. (a) fat leaflet ($\epsilon>0$), (b) thin leaflet ($\epsilon=0$).}
  \label{F:leaflet-notations}
}  
\end{figure}
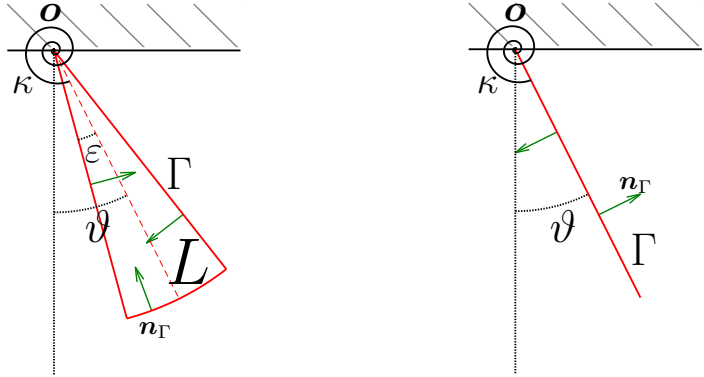

\medskip
In order to analyse the problem described above we need to define

\smallskip
\begin{itemize}
\item the equations governing the fluid motion,
\item the equilibrium equation of the leaflet,
\item the coupling equation  between the fluid and structure at the interface $\Gamma$.
\end{itemize}
%
%
\subsection{Fluid equations}
In our model system the fluid is assumed to be incompressible and Newtonian (i.e., having constant viscosity $\nu$), therefore the fluid motion is described by the classical incompressible Navier-Stokes equation
\begin{equation}
\label{eq:ns_fluido_primale1}
\left\{
\begin{aligned}
& -  \DD \, \T + (\Gr \uu ) \,\uu  = \ff \qquad  & &\text{in $\Omega$,} \\
& \ \dd \, \uu = 0 \qquad & &\text{in $\Omega$,} 
\end{aligned}
\right.
\end{equation}
with boundary conditions
\begin{equation}
\label{eq:ns_fluido_primale2}
\left\{
\begin{aligned}
& \uu = \boldsymbol{g}                      & &\text{on $\partial \D_D$,} \\
& \T \, \nn = \boldsymbol{h}   & &\text{on $\partial \D_N$,} \\
& \uu = \oo   & &\text{on $\Gamma$,} 
\end{aligned}
\right.
\end{equation}
where $\uu$, $p$ are the velocity and pressure fields, respectively, and $\T = \nu \Gr^s \uu +p \boldsymbol{I}$ is the Cauchy stress tensor.  
Here, $\DD$ and $\Gr$ (resp. $\dd$ and $\gr$) denote the divergence and gradient operator for vector fields (resp. for scalar functions), and $\Gr^s$ is the symmetric gradient.
Finally, $\nn$ designates the outward unit normal to $\partial \Omega$, while $\ff$ represents  the external force and $\boldsymbol{h}$ the outflow traction.
For what concerns the conditions in \eqref{eq:ns_fluido_primale2} prescribed on the boundary of the channel $\D$, we consider a partition of $\partial \D$ into two disjoints subsets  $\partial \D_D$, $\partial \D_N$ such that
$\partial \D = \partial \D_D \cup \partial \D_N$.
We assume that $\partial \D_N$ is not empty, so that the pressure is uniquely defined. 
In our model problem, we may consider that at the top and bottom wall of the channel (see Fig. \ref{F:leaflet-model})  no-slip boundary conditions are applied (i.e. $\bf g = \boldsymbol{0}$ therein). We denote by $\partial\D_\text{in} \!:= \!\textrm{supp}({\bf g})$ the ``inflow'' part of the boundary (since we have in mind that ${\bf g}\cdot {\bf n} <0$ therein, although this condition is not needed from a mathematical point of view); correspondingly,
we call $\partial \D_N = \partial\D_\text{out}$ the ``outflow'' part of the boundary.
We also remark that different boundary conditions on $\partial \D_N$ can be treated as well. 

\subsection{Structure and equilibrium equations}
Let us denote by $\kappa=\kappa(\theta)$ the angular momentum xerted by the leaflet on the rotational spring at angle $\theta$. Recalling the definiton \eqref{eq:Iepsilon0} of $I_{\epsilon_0}$, we assume that
\begin{equation}\label{eq:def-kappa}
\kappa : I_{\epsilon_0} \, \to \, \mathbb{R} \cup {\pm\infty}
\end{equation}
is a continuous, non-decreasing, possibly nonlinear function. It vanishes at some rest position $\theta=\theta_0$, in a neighborhood of which the spring response is supposed to be linear $\kappa(\theta)=\kappa_s (\theta-\theta_0)$, for some constant elastic modulus $\kappa_s$. Note that $\theta>\theta_0$ yields a clockwise torque $-\kappa(\theta)<0$ exerted by the spring on the leaflet.

The equilibrium of the leaflet is expressed by the momentum balance equation
\begin{equation}
\label{eq:leaflet_primale}
\kappa (\theta) = \tau(\theta) \,,
\end{equation}
where $\tau(\theta)$ represents the torque (or total angular momentum) with respect to the point $\ov$ exerted by the fluid on the structure $L$.
In order to express it, let us consider any point $\xv = \ov + r \e_\omega \in \Gamma$ having polar coordinates $(r,\omega)$ with respect to the hinge $\ov$, where $\e_\omega$ is defined in \eqref{eq: eomega}. We note that the unit vector $\e_\omega^\perp := (\cos \omega, \sin \omega)$, orthogonal to $\e_\omega$ and oriented counterclockwise, reads
$$
\e_\omega^\perp  = \begin{cases} \phantom{+} \nn_\Gamma & \text{for } \omega=\theta - \epsilon \,,\\
- \nn_\Gamma & \text{for } \omega=\theta+ \epsilon\,.
\end{cases}
$$
Then, the angular momentum $m=m(\xv)$ per unit of length of the force $\T \, \nn_\Gamma$ exerted by the leaflet $\Gamma$ on the fluid at $\xv\in\Gamma$ relative to $\ov$  is given by
$$
m = r \, \e_\omega^\perp \cdot \T \, \nn_\Gamma\,.
$$
Consequently, the torque $\tau(\theta)$ exerted by the fluid on $\Gamma$ is given by 
\begin{equation}
\label{eq:torque}
\tau(\theta) = -\int_{\Gamma} r \, \e_\omega^\perp \cdot \T \, \nn_\Gamma\,.
\end{equation}
Note that in the limit case $\epsilon\to0$, in which the leaflet $L=\Gamma$ is just a segment, formally one has
\begin{equation}\label{eq:torque-jump}
\tau(\theta) = -\int_{\Gamma} r \, \e_\omega^\perp \cdot \jump{ \T \, \nn_\Gamma}\,,
\end{equation}
where $\jump{\cdot}$ denotes the jump operator across the interface $\Gamma$. This formal limit will be justified later on.

In conclusion, view of \eqref{eq:leaflet_primale} and \eqref{eq:torque}, the angular momentum balance reads
\begin{equation}
\label{eq:balance}
\kappa (\theta) + \int_{\Gamma} r \, \e_\omega^\perp \cdot \T \, \nn_\Gamma\, = 0.
\end{equation}

\subsection{Variational formulation}
%
We are now ready to describe the system of equations for our model problem. 
Collecting the fluid motion equations \eqref{eq:ns_fluido_primale1}, the boundary conditions \eqref{eq:ns_fluido_primale2},
and the balance equation \eqref{eq:balance},
the strong formulation of the coupled problem reads
as follows: find $(\uu,\, p, \, \theta)$ such that
\begin{equation}
\label{eq:fsi_primale}
\left\{
\begin{aligned}
& -  \DD \, \T + (\Gr \uu ) \,\uu  = \ff\qquad  & \qquad &\text{in $\Omega$,} \\
& \dd \, \uu = 0 \qquad & &\text{in $\Omega$,} \\
& \uu = \boldsymbol{g}                         & \qquad &\text{in $\partial \D_D$,} \\
& \T \, \nn = \boldsymbol{h}  & \qquad &\text{in $\partial \D_N$,} \\
& \uu = \boldsymbol{0}                           &  \qquad &\text{on $\Gamma$,} \\
& \kappa (\theta) + \int_\Gamma r \, \e_\omega^\perp \cdot \T \, \nn_\Gamma\, = 0.
\end{aligned}
\right.
\end{equation}
We emphasize that the number of boundary conditioms on $\Gamma$ (last two lines of \eqref{eq:fsi_primale}) is overdetermined. This is typical of free boundary problems and accounts for the fact that the angular position $\theta$ of the leaflet is unknown.

The next step in the description of our fluid-structure interaction model problem 
is to introduce a suitable variational formulation of system \eqref{eq:fsi_primale}.
In particular we need to define a weak form of system \eqref{eq:fsi_primale} fitting the virtual element discretization that will be described in Section \ref{sec:VEMdis}.
We start by introducing the following Sobolev spaces for vector fields (i.e., the velocity spaces): 
\begin{equation}
\label{eq:spazi continui}
\VV := \left[ H^1(\OG) \right]^2\;, \qquad   
\VVG^{\boldsymbol{g}} : = \left\{ \vv \in \VV: \quad 
\vv_{|_{\partial \D_D}} = \boldsymbol{g}\,, \quad  \vv_{|_{\Gamma}} = \boldsymbol{0}  \right\}  \,,
\end{equation}
where $\boldsymbol{g}\in [H^{1/2}(\partial \D_D)]^2$. For pressures we consider the space $Q := L^2(\OG)$. 
These  spaces are endowed with the natural norms
\begin{equation}
\label{eq:norme continue}
\| \vv \|_{\VV} := \| \vv \|_{\left[ H^1(\Omega) \right]^2} \quad , \qquad 
\|q\|_Q := \| q\|_{L^2(\Omega)}. 
\end{equation}
Let us now define the following multi-linear forms 
\begin{align}
\label{eq:forma a}
&a(\cdot, \cdot) \colon \VV \times \VV \to \R,    
&\,
&a (\uu, \, \vv) := 
\int_{\Omega}  \, \Gr^s \uu : \Gr^s  \vv \,{\rm d} \Omega,
\\
\label{eq:forma b}
&b(\cdot, \cdot) \colon \VV \times Q \to \R,
&\, 
&b(\boldsymbol{v}, q) :=   \int_{\Omega}q\, {\rm div} \,\boldsymbol{v} \,{\rm d}\Omega,
\\
\label{eq:forma c}
&c(\cdot; \, \cdot, \cdot) \colon \VV \times \VV \times \VV \to \R, 
&
&c(\ww; \, \uu, \vv) :=  \int_{\Omega} ( \Gr \uu ) \, \ww \cdot \vv  \,{\rm d}\Omega,
\end{align}
for all $\uu, \vv, \ww \in \VV$ and $q \in Q$.
Furthermore, we assume that $\ff \in [L^2(\D)]^2$ and $\boldsymbol{h} \in [L^2(\partial\D_N)]^2$, and we denote by $(\cdot , \cdot)_{0,\Omega}$ and $(\cdot , \cdot)_{0,\partial\D_N}$ the $L^2$-inner products on $\OG$ and $\partial\D_N$, respectively.
We also denote by $\EG \in \VV$ the harmonic extension in $\Omega$ of the function defined on the skeleton by 
\begin{equation}
\label{eq:EG}
\EG :=
\left \{
\begin{aligned}
& r \, \e_\omega^\perp & \qquad  & \text{on $\Gamma$,}
\\
& \boldsymbol{0}        & \qquad  & \text{on $\partial \D$.}
\end{aligned}  
\right . 
\end{equation}

Among the various possible variational formulations of Problem \eqref{eq:fsi_primale}, we introduce the following one: find $\theta \in I_{\epsilon_0}$ 
and $(\uu, \, p) \in \VVG^{\boldsymbol{g}} \times Q$, such that
\begin{equation}
\label{eq:fsi_variazionale}
\left\{
\begin{aligned}
\nu \, a(\uu, \vv + \sigma \,\EG) &+
c(\uu; \, \uu, \vv + \sigma \, \EG) + 
b(\vv + \sigma \, \EG, p)
+ \sigma \, \kappa (\theta)
\\
& =(\ff, \vv + \sigma \, \EG)_{0,\Omega} 
+ (\boldsymbol{h}, \vv)_{0,\partial\D_N},  & &
\\
 b(\uu, q) &= 0, & &
\end{aligned}
\right. 
\end{equation}
for all $\vv \in \VVG^{\boldsymbol{0}}$, $\sigma \in \R$, and $q \in Q$.
It is straightforward to see that, taking $\sigma = 0$ in 
\eqref{eq:fsi_variazionale} we obtain the weak form of the Navier-Stokes equation (coupled with the boundary conditions on $\partial \Omega$) associated with the strong formulation in \eqref{eq:fsi_primale}.
On the other hand, taking $\vv = \boldsymbol{0}, \: \sigma=1$ in \eqref{eq:fsi_variazionale} we get
\begin{equation}\label{eq:weak-torque}
\nu \, a(\uu, \EG) +
c(\uu; \, \uu,  \EG) + 
b(\EG, p) - (\ff, \EG)_{0,\Omega} + \kappa (\theta)  = 0\,,
\end{equation}
which gives \eqref{eq:balance} after integration by parts. Expression \eqref{eq:weak-torque} is numerically better than \eqref{eq:balance} because it avoids evaluating explicitly the trace of $\T$ on $\Gamma$.
%

\section{Torque as a function of geometry}
\label{sect:shape}

In order to assess the solvability of Problem \eqref{eq:fsi_variazionale}, we aim at deriving suitable properties of the torque functional $\tau(\theta)$ introduced in \eqref{eq:torque} and \eqref{eq:torque-jump}, as a function of the angle $\theta$. To keep the technical burden at a minimum, in this section we assume that the velocity is so small, that the convective effects may be neglected; in other words, we assume that $(\uu,p)$ satisfies the Stokes problem
\begin{equation}
\label{eq:Stokes}
\left\{
\begin{aligned}
-  \DD \, \T  &= \ff  & \quad &\text{in $\Omega$,} \\
  \dd \, \uu  &= 0  & &\text{in $\Omega$,}
\end{aligned}
\right.
\qquad
\left\{
\begin{aligned}
 \uu & = \boldsymbol{0}                           &  \quad &\text{on $\Gamma$,} \\
 \uu & = \boldsymbol{g}                         & \quad &\text{on $\partial \D_D$,} \\
 \T \, \nn & = \boldsymbol{h}  & \quad &\text{on $\partial \D_N$ , } \\
\end{aligned}
\right.
\end{equation}
that corresponds to eliminating the term $c(\cdot;\cdot,\cdot)$ in \eqref{eq:fsi_variazionale}.

Furthermore, we assume that the leaflet has the form given in \eqref{eq: leaflet} for some $R$ and {$0\le \epsilon\le\epsilon_0/2$}. Thus, the geometry of the fluid domain, hence the torque functional $\tau$, depends on the three parameters $\theta$, $\epsilon$ and $R$. We restrict $\theta$ to satisfy $\theta \in I_{\epsilon_0}$, i.e. $|\theta| \le \frac\pi2 - \epsilon_0$, in order to avoid the contact of the leaflet with the upper wall. For the analysis we have in mind, it is convenient to think the torque as a function of the boundary of the leaflet (which in turn depends on these parameters), i.e., we rephrase \eqref{eq:torque} as
\begin{equation}\label{eq:def-J}
\tau(\theta) = J[\Gamma] := -\int_{\Gamma} r \, \e_\omega^\perp \cdot \T \, \nn_\Gamma\,.
\end{equation}
For the moment, we consider $R$ and $\epsilon\ge0$ as fixed, and we just allow rigid changes in $\Gamma=\Gamma(\theta)$ produced by changes in $\theta$. To this end, it is convenient to rewrite $J[\Gamma]$ in terms of integrals in the bulk $\OG$ instead of $\Gamma$. This is useful for differentiation of $J$ with respect to shape, because it avoids the appearance of terms involving the curvature of $\Gamma$ which is not well defined at the tip of the leaflet $\Gamma$ for any $\epsilon\ge0$.

\subsection{Equivalent form of $J[\Gamma]$}\label{S:equiv-J}
%
The following derivation includes both cases $\epsilon>0$ and $\epsilon=0$. Let us define in $\D$ the vector field 
\begin{equation}\label{eq:def-Phi}
\Fi(r,\omega) := r \, \e_\omega^\perp \,, \qquad \forall r, \omega \,,
\end{equation}
which allows us to rewrite \eqref{eq:def-J} (or equivalently \eqref{eq:torque}) for $\epsilon>0$
\begin{equation}\label{eq:expr1-J}
J[\Gamma] = -\int_{\Gamma} \Fi \cdot \T \, \nn_\Gamma\,,
\end{equation}
and similarly \eqref{eq:torque-jump} for $\epsilon=0$.
Note that in cartesian coordinates one has $\Fi(x,y)=(y_\ov-y,x-x_\ov)$, where $(x_\ov,y_\ov)$ are the cartesian coordinates of the hinge $\ov$. Let $(\zz,q)$ be the solution of the {\sl adjoint problem}
\begin{equation}
\label{eq:adjoint-pb}
\left\{
\begin{aligned}
 -  \DD \, \Tz   &= \oo \quad  & &\text{in $\Omega$,} \\
   \dd \, \zz & = 0 \quad & &\text{in $\Omega$,}
\end{aligned}
\right.
\qquad
\left\{
\begin{aligned}
 \zz &= \Fi  \quad &  &\text{on $\Gamma$,} \\
 \zz &= \oo                      & &\text{on $\partial \D_D$,} \\
 \Tz \, \nn &= \oo   & &\text{on $\partial \D_N$.} 
\end{aligned}
\right.
\end{equation}
This, and the boundary-value problem \eqref{eq:Stokes} satisfied by $(\uu,p)$, allows us to express $J[\Gamma]$ in \eqref{eq:expr1-J} as follows:
\begin{eqnarray*}
J[\Gamma] &=& -\int_{\Gamma} \zz \cdot \T \, \nn_\Gamma = 
-\int_{\Omega} \DD ( \T \, \zz) + \int_{\partial\D_N} \zz \cdot \T \, \nn \\
&=& -\int_{\Omega} \DD \, \T \cdot \zz - \int_{\Omega} \T : \Gr \zz + \int_{\partial\D_N} \boldsymbol{h} \cdot \zz  \\
&=&  \int_{\Omega} \ff \cdot \zz - \nu  \int_{\Omega} \Gr^s \uu : \Gr^s \zz + \int_{\partial\D_N} \boldsymbol{h} \cdot \zz \,,
\end{eqnarray*}
because $p\boldsymbol{I} : \Gr \zz = p \, \dd \, \zz = 0$. Therefore, from now on we focus on the expression
\begin{equation}\label{eq:expr2-J}
J[\Gamma] = -\nu  \int_{\Omega} \Gr^s \uu : \Gr^s \zz + \int_{\Omega} \ff \cdot \zz  + \int_{\partial\D_N} \boldsymbol{h} \cdot \zz  \, ,
\end{equation}
which is also valid for $\epsilon=0$. 
\begin{proposition}[boundedness of {$J[\Gamma]$}]\label{L:bound-J}
There exists a constant $C(\ff,\boldsymbol{g},\boldsymbol{h})$ depending on 
$\Vert \ff \Vert_{[L^2(\D)]^2}$, $\Vert \boldsymbol{g} \Vert_{[H^{1/2}(\partial\D_D)]^2}$, and $\Vert \boldsymbol{h} \Vert_{[L^2(\partial\D_N)]^2}$, but uniform in $ \epsilon \in [0,\frac{\epsilon_0}{2}]$ and $\theta \in I_{\epsilon_0}$ such that
\[
\big| J[\Gamma] \big| \le C(\ff,\boldsymbol{g},\boldsymbol{h}).
\]
\end{proposition}
\begin{proof}
This entails a priori bounds for $\|\uu\|_{[H^1(\Omega)]^2}$ and $\|\zz\|_{[H^1(\Omega)]^2}$ that account for the boundary conditions in \eqref{eq:Stokes} and \eqref{eq:adjoint-pb} and are uniform in $\epsilon$ and $\theta$. It is not restrictive, in this proof, to assume that $\D=[-1,1]\times[0,1]$ as depicted in Figure \ref{F:leaflet-model}.
Using polar coordinates $(r,\omega)$ with respect to the hinge $\ov$ and vertical dotted line of Figure \ref{F:leaflet-notations}, we let $\Omega_0$ be a set that contains all admissible positions of the leaflet $\Gamma$: \looseness=-1
\[  
\Omega_0 := \Big \{(r,\omega): \quad 0 \le r \le R <1, \quad -\frac{\pi}2+\frac{\epsilon_0}{2} \le \omega \le \frac{\pi}2-\frac{\epsilon_0}{2} \Big\}.
\]
Hence, the set $\D\setminus\Omega_0$ contains the U-shaped domain
\[
\Xi := \Big\{(x,y): \quad |x \pm 1| < \delta_0  \ \,  \textrm{or} \ \, y < \delta_0 \Big\}
\]
for any $0< \delta_0 < 1-R$. We let $\uu_{\boldsymbol{g}}\in [H^1(\Xi)]^2$ solve the Stokes equation on $\Xi$ with vanishing Dirichlet condition on $\partial\Xi$ except on $\D_{\textrm{in}}$ where $\uu_{\boldsymbol{g}}=\boldsymbol{g}$ and on $\D_{\textrm{out}}$ where we assume homogeneous Neumann conditions. 

We now extend $\uu_{\boldsymbol{g}}$ by zero to $\D\setminus\Xi$, without relabelling, and realize that $\uu_{\boldsymbol{g}}$ is divergence free in $\D$ and $\Vert \uu_{\boldsymbol{g}} \Vert_{\left[ H^1(\OG) \right]^2} \le c \Vert \boldsymbol{g} \Vert_{[H^{1/2}_{00}(\partial\D_\text{in})]^2}$ with $c$ independent of $\epsilon$ and $\theta$.

We now split $\uu = \uu_{\boldsymbol{0}}+\uu_{\boldsymbol{g}}$ with $\uu_{\boldsymbol{0}}$ vanishing on $\partial\Omega \setminus \partial\D_\text{out}$ and write the variational formulation of the momentum equation of \eqref{eq:Stokes} for the pair $(\uu_{\boldsymbol{0}},p)$, bringing  $\uu_{\boldsymbol{g}}$ on the right-hand side. Choosing the divergence-free test functions $\uu_{\boldsymbol{0}}$ eliminates the pressure $p$ and yields
\begin{equation*}\label{eq:bound-upH1-1}
\Vert \uu_{\boldsymbol{0}} \Vert_{\left[ H^1(\OG) \right]^2} \le A \Big(\Vert \ff \Vert_{[L^2({\cal D})]^2} + \Vert \boldsymbol{g} \Vert_{[H^{1/2}_{00}(\partial\D_\text{in})]^2} + \Vert  \boldsymbol{h} \Vert_{[L^2(\partial\D_\text{out})]^2}\Big),
\end{equation*}
where $A$ is independent of $\epsilon$ and $\theta$. A similar bound is thus valid for $\uu$.

Regarding the regularity of $\zz$, we observe that the Dirichlet data $\boldsymbol{\Phi}$ defined in \eqref{eq:def-Phi} is divergence-free. Let $\zz_{\boldsymbol{\Phi}} \in [H^1(\D\setminus\Omega_0)]^2$ solve the Stokes equation with vanishing Dirichlet condition on $\partial\D$ and $\zz_{\boldsymbol{\Phi}} = \boldsymbol{\Phi}$ on $\partial\Omega_0$. Extending $\zz_{\boldsymbol{\Phi}}$ by $\boldsymbol{\Phi}$ within $\Omega_0$, without relabelling, we notice that $\zz_{\boldsymbol{\Phi}}$ is divergence-free and $\|\zz_{\boldsymbol{\Phi}}\|_{[H^1(\Omega)]^2}\le c \|\boldsymbol{\Phi}\|_{[H^1(\Omega_0)]^2}$. Splitting $\zz = \zz_0 + \zz_{\boldsymbol{\Phi}}$ and arguing as before yields
\begin{equation*}
\|\zz\|_{[H^1(\Omega)]^2} \le B \|\boldsymbol{\Phi}\|_{[H^1(\Omega_0)]^2},
\end{equation*}
where $B$ is independent of $\epsilon$ and $\theta$. This concludes the proof.
\end{proof}

The argument in Proposition \ref{L:bound-J} circumvents dealing with the pressures $p$ and $q$. However, they can also be bounded uniformly as the following lemma reveals. This result is useful later in estimating $\T$ and $\Tz$.

\begin{lemma}[uniform lower bound of inf-sup constant]\label{L:inf-sup}
The inf-sup constant $\beta=\beta(\Omega)$ of the domain $\Omega$ for the space pair $(\VVG^{\boldsymbol{0}},Q)$ defined in \eqref{eq:spazi continui},
\[
\beta = \infsup{q\in Q}{\vv\in \VVG^{\boldsymbol{0}}}
\frac{\int_\Omega q \, \dd \vv}{\|q\|_{L^2(\Omega)}|\vv|_{[H^1(\Omega)]^2}} \,,
\]
is bounded away from 0  uniformly with respect to $ \epsilon \in [0,\frac{\epsilon_0}{2}]$ and $\theta \in I_{\epsilon_0}$. 
\end{lemma}
\begin{proof}
We proceed in three steps. We first decompose the domain $\Omega$ into two subdomains upon extending the bisector of the leaflet $L$ starting at the hinge $\ov$ until it intersects the boundary of $\D$. This divides $\Omega$ into two disjoint subdomains $\Omega_1$ and $\Omega_2$ with reentrant corners separated by a straight segment $S$ (the bisector extension); see Figure \ref{F:leaflet-notations}. We show that these domains possess a uniform inf-sup constant in the spaces $[H^1_0(\Omega_i)]^2$, namely with zero trace. We next prove a uniform global inf-sup constant in $[H^1_0(\Omega)]^2$. We finally extend the inf-sup to the space $\VVG^{\boldsymbol{0}}$ of velocities that vanish only on $\partial\Omega\setminus\partial\D_N$. It is not restrictive to assume again that $\D=[-1,1]\times[0,1]$ as well as that the leaflet length satisfies $R\le 1-\frac{\epsilon_0}{2}$.

\medskip\noindent  
1. {\it Local inf-sup constants.}
Since $\epsilon\le\frac{\epsilon_0}{2}$, the smallest angle made by the boundary $\Gamma$ of $L$ and the upper wall of $\D$ is bounded below by $\frac{\epsilon_0}{2}$. Moreover, the distance from $\Gamma$ to the lower wall of $\D$ is also bounded below by $\frac{\epsilon_0}{2}$. Therefore, there exist two balls $B_1$ and $B_2$ with radii $\frac{\epsilon_0}{4}$ and centers within $\Omega_1$ and $\Omega_2$, depending on $\theta$, such that $\Omega_1$ and $\Omega_2$ are star-shaped with respect to $B_1$ and $B_2$, respectively.

The inf-sup contant in $\Omega_i$ is the reciprocal of the stability constant of the right inverse of the operator $\dd: [H^1_0(\Omega_i)]^2 \to L^2_0(\Omega_i)$ for $i=1,2$, where $L^2_0(\Omega_i)$ stands for functions in $L^2(\Omega_i)$ with vanishing mean \cite{GiraultRaviart:86,DuranMuschietti:01,Duran:12}. According to Remark 3.1 of \cite{Duran:12}, such constant is bounded above by $r|\log r|$, where $r$ is the ratio between the radius of a uniform ball containing $\Omega_i$, say 2, and the radius $\frac{\epsilon_0}{4}$ of $B_i$ irrespective of the location of $B_i$ within $\Omega_i$. This shows the existence of an inf-sup constant $\beta_0$ for $\Omega_i, i=1,2$ with uniform lower bound solely depending on $\epsilon_0$.

\medskip\noindent  
2. {\it Global inf-sup constant in $[H^1_0(\Omega)]^2$.}
We now follow \cite{BolandNicolaides:83} to glue $\Omega_1,\Omega_2$ together; see Section 1.4 of \cite{GiraultRaviart:86}. Given $q \in L^2_0(\Omega)$, we decompose it as $q = \widetilde{q} + \overline{q}$, where in $\Omega_i$ the function $\widetilde{q}$ has zero mean whereas $\overline{q}$ is constant and given by the mean value of $q$ within $\Omega_i$. We thus have the $L^2$-orthogonal decomposition
\[
\| q \|_{L^2(\Omega)}^2 = \| \widetilde{q} \|_{L^2(\Omega)}^2  + \| \overline{q} \|_{L^2(\Omega)}^2. 
\]
In view of Step 1 and \cite{GiraultRaviart:86}, we can associate $\widetilde{\vv}_i\in [H^1_0(\Omega_i)]^2$ to $\widetilde{q}_i=\widetilde{q}|_{\Omega_i}$ so that
\[
\int_{\Omega_i} \widetilde{q}_i \, \dd \widetilde{\vv}_i = \|\widetilde{q}_i\|_{L^2(\Omega)}^2,
\qquad
|\widetilde{\vv}_i|_{[H^1(\Omega_i)]^2} \le \frac{1}{\beta_0} \|\widetilde{q}_i\|_{L^2(\Omega_i)}.
\]
Let $\widetilde{\vv}\in[H^1_0(\Omega)]^2$ be so that $\widetilde{\vv}|_{\Omega_i}=\widetilde{\vv}_i$ and note that $|\widetilde{\vv}|_{[H^1(\Omega)]^2} \le \frac{1}{\beta_0} \|\widetilde{q}\|_{L^2(\Omega)}$. Since $q\in L^2_0(\Omega)$ yields $\overline{q}_1 |\Omega_1| + \overline{q}_2 |\Omega_2| = 0$, we deduce that $\overline{q}_1=\overline{q}|_{\Omega_1}$ and $\overline{q}_2=\overline{q}|_{\Omega_2}$ have opposite signs. Let $\sigma=\pm 1$ be the sign of $\overline{q}_1$ and $\overline{\ww} \in [H^1_0(\Omega)]^2$ satisfy
\[
\int_S \overline{\ww} \cdot \nn_1 = \sigma,
\qquad
|\overline{\ww}|_{[H^1(\Omega)]^2} \simeq 1,
\]
where $\nn_1$ is the unit outer normal to $\Omega_1$. Consequently, integrating by parts gives
\[
\int_\Omega \overline{q} \, \dd \overline{\ww} = \int_S (\overline{q}_1 - \overline{q}_2)
\, \overline{\ww}\cdot\nn_1 = |\overline{q}_1| + |\overline{q}_2|
\simeq \|\overline{q}\|_{L^2(\Omega)}
\]
and let $C\simeq 1$, $\gamma_0 \simeq 1$ be so that $\overline{\vv} = C \|\overline{q}\|_{L^2(\Omega)} \overline{\ww}\in [H^1_0(\Omega)]^2$ satisfies
\[
\int_{\Omega} \overline{q} \, \dd \overline{\vv} = \|\overline{q}\|_{L^2(\Omega)}^2,
\qquad
|\overline{\vv}|_{[H^1(\Omega)]^2} \le \frac{1}{\gamma_0} \|\overline{q}\|_{L^2(\Omega)}.
\]
To prove the inf-sup property in $\Omega$, we construct a velocity $\vv = \widetilde{\vv} + \alpha \overline{\vv}$ with $\alpha>0$ to be determined. We observe that a direct calculation yields
\[
\int_\Omega q \, \dd \vv = \|\widetilde{q}\|_{L^2(\Omega)}^2 + \alpha \|\overline{q}\|_{L^2(\Omega)}^2 + \alpha \int_\Omega \widetilde{q} \, \dd \overline{\vv}.
\]
Since $\|\dd\overline{\vv}\|_{L^2(\Omega)} \le |\overline{\vv}|_{[H^1(\Omega)]^2}$, in view of Lemma 2.1 of \cite{NochettoPyo:04}, the Cauchy-Schwarz and Young inequalities imply
\[
\int_\Omega q \, \dd \vv \ge \Big(1-\frac{\alpha}{2\gamma_0\delta} \Big)
\|\widetilde{q}\|_{L^2(\Omega)}^2
+ \alpha \Big( 1 - \frac{\delta}{2\gamma_0}  \Big) \|\overline{q}\|_{L^2(\Omega)}^2
\ge \frac{1}{2} \min\{1,\gamma_0^2\} \|q\|_{L^2(\Omega)}^2,
\]
provided $\delta=\gamma_0$ and $\alpha=\gamma_0^2$, along with
\[
|\vv|_{[H^1(\Omega)]^2} \le \frac{1}{\beta_0} \|\widetilde{q}\|_{L^2(\Omega)}
+ \frac{\alpha}{\gamma_0} \|\overline{q}\|_{L^2(\Omega)} \le \Big(\frac{1}{\beta_0^2} + \gamma_0^2\Big)^{\frac12} \|q\|_{L^2(\Omega)}.
\]
The uniform inf-sup constant in $[H^1_0(\Omega)]^2$ is thus
$\beta_1=\frac12 \min\{1,\gamma_0^2\} (\beta_0^{-2}+\gamma_0^2)^{-\frac12}$.

\medskip\noindent  
3. {\it Global inf-sup constant in $\VVG^{\boldsymbol{0}}$.} Let $q=\widetilde{q}+\overline{q} \in Q=L^2(\Omega)$ be given, with $\widetilde{q} \in L^2_0(\Omega)$ and $\overline{q}$ being the mean-value of $q$. We let $\widetilde{\vv}\in [H^1_0(\Omega)]^2$ satisfy
\[
\int_\Omega \widetilde{q} \, \dd \widetilde{\vv} = \|\widetilde{q}\|_{L^2(\Omega)}^2,
\qquad
|\widetilde{\vv}|_{[H^1(\Omega)]^2} \le \frac{1}{\beta_1} \|\widetilde{q}\|_{L^2(\Omega)},
\]
in light of Step 2. We proceed as in Step 2. We first let $\ww \in \VVG^{\boldsymbol{0}}$ satisfy
\[
\int_\Omega \dd \ww = \int_{\partial\Omega} \ww\cdot\nn = \int_{\partial\D_N} \ww\cdot\nn = 1,
\qquad |\ww|_{H^1(\Omega)} \simeq 1,
\]
and next let $C\simeq 1$, $\gamma_1 \simeq 1$ and $\overline{\vv}=C\|\overline{q}\|_{L^2(\Omega)}\ww\in \VVG^{\boldsymbol{0}}$ be so that 
\[
\int_\Omega \overline{q} \, \dd \overline{\vv} = \|\overline{q}\|_{L^2(\Omega)}^2,
\qquad |\overline{\vv}|_{[H^1(\Omega)]^2} \le \frac{1}{\gamma_1} \|\overline{q}\|_{L^2(\Omega)}.
\]
A straightforward calculation shows that the function $\vv=\widetilde{\vv} + \gamma_1^2 \overline{\vv}\in\VVG^{\boldsymbol{0}}$ satisfies
\[
\int_\Omega q \, \dd \vv \ge \frac12 \min\{1,\gamma_1^2\} \|q\|_{L^2(\Omega)}^2,
\qquad |\vv|_{[H^1(\Omega)]^2} \le \Big(\frac{1}{\beta_1^2} + \gamma_1^2 \Big)^{\frac12}
\|q\|_{L^2(\Omega)},
\]
which is the asserted inf-sup property with $\beta=\frac12 \min\{1,\gamma_1^2\} (\beta_1^{-2}+\gamma_1^2)^{-1/2}$. 
\end{proof}

In the proof of Proposition \ref{L:bound-J} we show, in particular, uniform bounds for the velocity solutions ${\uu}$ and ${\zz}$ of \eqref{eq:Stokes} and \eqref{eq:adjoint-pb}. Combining such bounds with \eqref{eq:Stokes}, \eqref{eq:adjoint-pb} and using Lemma \ref{L:inf-sup}, deriving uniform bounds on the natural norms for the velocity-pressure pairs $(\uu,p)$ and $(\zz,q)$ is immediate.
\begin{corollary}[uniform stability]\label{cor:energy-uz} 
There exists a constant $C$ independent of $ \epsilon \in [0,\frac{\epsilon_0}{2}]$ and $\theta\in I_{\epsilon_0}$ such that
\begin{equation}\label{eq:energy-bound}
  \Vert \uu \Vert_{\left[ H^1(\Omega)) \right]^2} + \Vert p \Vert_{L^2(\Omega)}
  \le C \Big( \Vert \ff \Vert_{[L^2({\cal D})]^2} + \Vert \boldsymbol{g} \Vert_{[H^{1/2}_{00}(\partial\D_\text{in})]^2} + \Vert  \boldsymbol{h} \Vert_{[L^2(\partial\D_\text{out})]^2}  \Big)
\end{equation}
and
\begin{equation}\label{eq:energy-z}
   \Vert \zz \Vert_{\left[ H^1(\Omega)) \right]^2} + \Vert q \Vert_{L^2(\Omega)}
  \le C \|\boldsymbol{\Phi}\|_{H^1(\Omega_0)} \lesssim C\,.
\end{equation}

\end{corollary}

\subsection{Case $\epsilon>0$: Shape derivative of {$J[\Gamma]$}}\label{S:shape-derivative}
%
We use rules of shape differential calculus (Reynolds Theorem) to compute the rate of variation of $J[\Gamma]$ produced by an infinitesimal rotation of $\Gamma$ around the hinge $\ov$. More precisely, we consider a rotation given by the velocity
\begin{equation}\label{eq:def-V}
\VVe := \Fi \,;
\end{equation} 
this corresponds to a flow dictated by the ODE $\dot{\xv}=\VVe(\xv(t))$, which preserves the rigid structure (and form) of the leaflet. Define the {\sl normal velocity on} $\Gamma$ by
\begin{equation}\label{eq:def-Ve}
\Ve := \VVe \cdot \nn_\Gamma \,.
\end{equation} 
Then, the {\sl shape derivative of $J[\Gamma]$ in the direction $\VVe$} is (formally) given by \cite{DelfourZolesio,SokolowskiZolesio}
\begin{equation}\label{eq:def-shapeJ}
\begin{split}
\delta J[\Gamma; \VVe] &:= - \nu  \int_{\Gamma} \Gr^s \uu : \Gr^s \zz \, \Ve + \int_{\Gamma} \ff \cdot \zz \, \Ve \\
& - \ \nu  \int_{\Omega} \Gr^s \uu' : \Gr^s \zz  - \nu  \int_{\Omega} \Gr^s \uu : \Gr^s \zz' + \int_{\Omega} \ff \cdot \zz' + \int_{\partial\D_N} \boldsymbol{h} \cdot \zz' \,,
\end{split}
\end{equation}
where $\uu'=\uu'(\Gamma; \VVe)$ and $\zz'=\zz'(\Gamma; \VVe)$ are the shape derivatives of $\uu$ and $\zz$ in the direction $\VVe$, and are the solutions of the boundary-value problems
\begin{equation}
\label{eq:def-uprime}
\left\{
\begin{aligned}
 -  \DD \, \Tup & = \oo  & &\text{in $\Omega$,} \\
   \dd \, \uu' & = 0 & &\text{in $\Omega$,}
\end{aligned}
\right.
\quad
\left\{
\begin{aligned}
 \uu' &= - (\Gr \uu) \VVe  & &\text{on $\Gamma$,} \\
 \uu' &= \oo                      & &\text{on $\partial \D_D$,} \\
 \Tup \, \nn &= \oo   & &\text{on $\partial \D_N$\,,} 
\end{aligned}
\right.
\end{equation} 
and
\begin{equation}
\label{eq:def-zprime}
\left\{
\begin{aligned}
 -  \DD \, \Tzp  &= \oo  & &\text{in $\Omega$,} \\
   \dd \, \zz' &= 0  & &\text{in $\Omega$,}
\end{aligned}
\right.
\quad
\left\{
\begin{aligned}
 \zz' &= - (\Gr(\zz-\Fi)) \VVe  & &\text{on $\Gamma$,} \\
 \zz' &= \oo                      & &\text{on $\partial \D_D$,} \\
 \Tzp \, \nn &= \oo   & &\text{on $\partial \D_N$.} 
\end{aligned}
\right.
\end{equation} 
Note that in order to give a meaning to the second integral on the right-hand side of \eqref{eq:def-shapeJ}, we have to assume more regularity on $\ff$, so that its trace on $\Gamma$ is well-defined. This occurs, e.g., if $\ff \in [W^{1,1}(\D)]^2$ because then $\ff \in [L^1(\Gamma)]^2$.
 
Now, we manipulate certain integrals appearing in \eqref{eq:def-shapeJ} and we show that $\delta J[\Gamma; \VVe] $ only depends upon $\uu$ and $\zz$ on $\Gamma$, which will imply that $\delta J[\Gamma; \VVe] $ is well-defined and finite. Let us first observe that
\begin{equation}\label{eq:gradu}
\Gr \uu = \Gr \uu \, \boldsymbol{I} = \Gr \uu \, (\nn \otimes \nn + \ttt \otimes \ttt) = \partial_n \uu \otimes \nn \qquad \text{on } \Gamma\,, 
\end{equation}
since $\uu = \oo$ on $\Gamma$. It follows that
\begin{equation}\label{eq:first-integral}
\begin{split}
I_1 & := -\nu  \int_{\Gamma} \Gr^s \uu : \Gr^s \zz \, \Ve = -\nu  \int_{\Gamma} \Gr \uu : \Gr^s \zz \, \Ve \\
& = -\nu  \int_{\Gamma} \partial_n \uu \otimes \nn :  \Gr^s \zz \, \Ve = -\nu  \int_{\Gamma} \partial_n \uu  \cdot (\Gr^s \zz) \nn \, \Ve \,.
\end{split}
\end{equation}
On the other hand, using \eqref{eq:adjoint-pb} and \eqref{eq:def-uprime}, we have
\begin{equation*}
\begin{split}
I_2 & := -\nu  \int_{\Omega} \Gr^s \uu' : \Gr^s \zz   =  -\nu  \int_{\Omega} \Gr \uu' : \Gr^s \zz  = -\int_{\Omega} \Gr \uu' : \Tz  \\
& =  \int_{\Omega}  \uu' \cdot \DD \, \Tz - \int_\Gamma \uu' \cdot \Tz \nn =  \int_\Gamma (\Gr \uu)\VVe \cdot  \Tz \nn \,.
\end{split}
\end{equation*} 
By \eqref{eq:gradu} we obtain
$$
\nu  \int_\Gamma (\Gr \uu)\VVe \cdot  (\Gr^s \zz) \nn  = \nu  \int_\Gamma \partial_n \uu  \cdot (\Gr^s \zz) \nn \, \Ve 
$$ 
and
$$
\int_\Gamma (\Gr \uu)\VVe \cdot  (q \, \nn)  = \int_\Gamma \partial_n \uu  \cdot \nn \ q \, V = 0 \,, 
$$
which easily follows by combining $\partial_t \uu =\oo$ and $\dd \, \uu = 0$ on $\Gamma$. Thus,
\begin{equation}\label{eq:second-integral} 
I_2 =  \nu  \int_{\Gamma} \partial_n \uu  \cdot (\Gr^s \zz) \nn \, \Ve \,.
\end{equation}
At last, using \eqref{eq:Stokes} and \eqref{eq:def-zprime} we obtain
\begin{equation*}
\begin{split}
I_3  & :=  -\nu  \int_{\Omega} \Gr^s \uu : \Gr^s \zz' = -\nu  \int_{\Omega} \Gr^s \uu : \Gr \zz' = - \int_{\Omega} \T : \Gr \zz' \\
& =  \int_{\Omega} \DD \, \T \cdot \zz' - \int_{\partial \Omega} \T\nn \cdot \zz' \\
& =  -\int_{\Omega} \ff \cdot \zz'- \int_{\partial\D_N} \boldsymbol{h} \cdot \zz' + \int_\Gamma  \T\nn \cdot (\Gr(\zz-\Fi)) \VVe \,.
\end{split}
\end{equation*} 
We need to examine the last term. To this end, we set $\ww=\zz-\Fi$ and note that $\ww=\oo$ on $\Gamma$. Hence, as in \eqref{eq:gradu}, $\Gr \ww = \partial_n \ww \otimes \nn$ and $\Gr \ww \VVe = \partial_n \ww \, V = \partial_n \zz \, V - \partial_n \Fi \, V$.
To proceed further, let us split $\Gamma$ as  
\begin{equation}\label{eq:split-Gamma}
\begin{split}
\Gamma_1 & := \{ \xv=\ov + r(\sin \omega, -\cos \omega) :  0 \le r \le R, \ \omega = \theta - \epsilon \} \,, \\
\Gamma_2 & := \{ \xv=\ov + r(\sin \omega, -\cos \omega) :  0 \le r \le R, \ \omega = \theta + \epsilon \} \,, \\
\Gamma_3 & := \{ \xv=\ov + r(\sin \omega, -\cos \omega) :  r  = R, \ |\omega- \theta| \le \epsilon \} \,.
\end{split}
\end{equation}
Then, since $\frac1r \partial_\omega (r \, \e_\omega^\perp ) = - \e_\omega$ and $\partial_r (r \, \e_\omega^\perp ) = \e_\omega^\perp$, we obtain
$$
\partial_n \Fi = \begin{cases} - \e_\omega & \text{on } \Gamma_1 \,, \\
\phantom{-}\e_\omega & \text{on } \Gamma_2 \,, \\
\phantom{-}\e_\omega^\perp  & \text{on } \Gamma_3 \,,
\end{cases}
\qquad
V = r \e_\omega^\perp  \cdot \nn = \begin{cases} \phantom{-}r & \text{on } \Gamma_1\,, \\
-r & \text{on } \Gamma_2\,, \\
\phantom{-}0 & \text{on } \Gamma_3\,. \\
\end{cases}
$$
This yields
$$
\int_\Gamma  \T\nn \cdot (\Gr(\zz-\Fi)) \VVe  = \int_{\Gamma_1 \cup \Gamma_2}  \T\nn \cdot  \partial_n \zz \, V + \int_{\Gamma_1 \cup \Gamma_2}  \T\nn \cdot  \e_\omega  \, |V|
$$
Since $\nn \cdot \e_\omega = 0$ on $\Gamma_1 \cup \Gamma_2$, we deduce
$$
\int_{\Gamma_1 \cup \Gamma_2}  \T\nn \cdot  \e_\omega  \, |V| = \nu \int_{\Gamma_1 \cup \Gamma_2}  (\Gr^s \uu) \nn \cdot \e_\omega \, |V| \,.
$$
We conclude that
\begin{equation*}\label{eq:third-integral}
 I_3 = -\int_{\Omega} \ff \cdot \zz'- \int_{\partial\D_N} \boldsymbol{h} \cdot \zz' + \int_{\Gamma_1 \cup \Gamma_2}  \T\nn \cdot  \partial_n \zz \, V +\nu \int_{\Gamma_1 \cup \Gamma_2}  (\Gr^s \uu) \nn \cdot \e_\omega \, |V| \,.
\end{equation*} 
Substituting this expression along with \eqref{eq:first-integral} and \eqref{eq:second-integral} into \eqref{eq:def-shapeJ}, we obtain the following formal expression for the shape derivative of $J[\Gamma]$.

\begin{proposition}[formal shape derivative]\label{prop:shapeder1}
The shape derivative $\delta J[\Gamma;\VVe]$ of $J[\Gamma]$ in the direction $\VVe$ is given by
\begin{equation*}\label{eq:shapeder1}
\delta J[\Gamma; \VVe] =  \int_{\Gamma_1 \cup \Gamma_2}  \T\nn \cdot  \partial_n \zz \, V + \nu \int_{\Gamma_1 \cup \Gamma_2}  (\Gr^s \uu) \nn \cdot \e_\omega \, |V|+ \int_{\Gamma_1 \cup \Gamma_2} \ff \cdot \zz \, \Ve \,.
\end{equation*}
\end{proposition}

In order to check that the integrals on the right-hand side of $\delta J[\Gamma; \VVe]$ are finite, we must invoke regularity of $\uu$ and $\zz$ higher than $H^1$, at least in a neighborhood of $\Gamma$. Consequently, we must improve upon Lemma \ref{L:bound-J}. This is our next task.

\begin{proposition}[boundedness of {$\delta J[\Gamma; \VVe]$}]\label{prop:shapeder2}
Let $0<\epsilon\le\epsilon_0/2$ be fixed and $\ff \in [W^{1,1}(\D)]^2$, $\boldsymbol{g} \in [H^{1/2}_{00}(\partial\D_\text{in})]^2$ and $\boldsymbol{h} \in [L^2(\partial\D_\text{out})]^2$. Then, the shape derivative $\delta J[\Gamma; \VVe]$ is well-defined, and there exists a constant $C = C(\ff, \boldsymbol{g}, \boldsymbol{h}) >0$ depending on $\|\ff\|_{[W^{1,1}(\D)]^2}$, $\|\boldsymbol{g}\|_{[H^{1/2}_{00}(\partial\D_\text{in})]^2}$ and $\|\boldsymbol{h}\|_{[L^2(\partial\D_\text{out})]^2}$ such that
\begin{equation}\label{eq:shapeder2}
\big| \, \delta J[\Gamma; \VVe] \, \big| \le C(\ff, \boldsymbol{g}, \boldsymbol{h}) \,.
\end{equation}
Furthermore, the constant $C(\ff, \boldsymbol{g}, \boldsymbol{h})$ is uniform with respect to $\theta\in I_{\epsilon_0}$.
\end{proposition}
\begin{proof} 
In order to establish local regularity estimates around $\Gamma$ beyond $H^1$, it is convenient to adopt a reference system centered at the hinge $\ov$, with the symmetry axis of $L$ as one of the coordinate axis. In this system, $\Omega$ has walls and inflow/outflow boundaries that change with $\theta$, whereas $L$ is fixed. Since by assumption $\ff \in [L^2({\cal D})]^2$ and $\uu=\boldsymbol{0}$ on $\Gamma$ and the upper wall, $(\uu,p)$ has regularity $H^2-H^1$ in a neighborhood of $\Gamma$, except possibly around the hinge $\ov$ and the two corners $C_1:=\Gamma_1\cap \Gamma_3$ and $C_2:=\Gamma_2\cap \Gamma_3$ at the tip. Let us examine these cases.

As $\Gamma$ and the upper wall of ${\cal D}$ form angles smaller than $\pi$, $\Omega$ is convex in a neighborhood ${\cal N}_{\ov,1}$ (resp. ${\cal N}_{\ov,2}$) of the hinge comprised between $\Gamma_1$ (resp. $\Gamma_2$) and the upper wall. Hence, $(\uu,p)$ has regularity $H^2-H^1$ in these neighborhoods (see e.g. Chapter 7 in \cite{Grisvard}). In view of the restrictions on $\epsilon$ and $\theta$, these neighborhoods ${\cal N}_{\ov,i}$ cannot degenerate to segments, and the $H^2-H^1$ norms of $(\uu,p)$ in such neighborhoods can be bounded uniformly with respect to $\epsilon$ and $\theta$.

At the tip corners $C_i$, $i=1,2$, $\OG$ forms angles of measure $\frac{3\pi}2$; in this case, according to \cite{Kondratev67} (see also \cite{Chorfi14}) $\uu$ can be decomposed in a neighborhood ${\cal N}_{C_i}$ of $C_i$ into the sum of a regular part $\uu_\text{reg}$ which is locally $H^2$, and a singular part $\uu_\text{sing}$, which -- in a polar coordinate system $(\varrho,\varphi)$ centered at $C_i$ -- behaves like $\rho^{\alpha} \boldsymbol{U}(\phi)$ with $\alpha \simeq 0.544484$ and $\boldsymbol{U}$ smooth. Thus, $(\uu,p)$ has regularity $H^{1+s}-H^s$ in ${\cal N}_{C_i}$ for any $s < \alpha$, and the neigborhoods ${\cal N}_{C_i}$ can be chosen independent of $\theta$.

In conclusion, by localizing the analysis near $\Gamma$ by a partition-of-unity argument, we can find a tubular neighborhood ${\cal N}_\Gamma$ of $\Gamma$ and constant $C_s$ such that $(\uu,p) \in [H^{1+s}({\cal N}_\Gamma)]^2 \times H^{s}({\cal N}_\Gamma)$ with
\begin{equation*}
\Vert \uu \Vert_{\left[ H^{1+s}({\cal N}_\Gamma) \right]^2} + \Vert p \Vert_{H^{s}({\cal N}_\Gamma) } \le C_s \left( \Vert \ff \Vert_{[L^2({\widetilde{\cal N}}_\Gamma)]^2}  + \Vert \uu \Vert_{\left[ H^1({\widetilde{\cal N}}_\Gamma)) \right]^2} + \Vert p \Vert_{L^2({\widetilde{\cal N}}_\Gamma)} \right)\;,
\end{equation*}
where ${\widetilde{\cal N}}_\Gamma$ is an extension of ${\cal N}_\Gamma$.  The radius of ${\cal N}_\Gamma$ and the constant $C_s$ can be chosen independent of $\theta$. Combining this with \eqref{eq:energy-bound} yields
\begin{equation*}\label{eq:bound-upHs}
\Vert \uu \Vert_{\left[ H^{1+s}({\cal N}_\Gamma) \right]^2} + \Vert p \Vert_{H^{s}({\cal N}_\Gamma) } \le C  \left(\Vert \ff \Vert_{[L^2({\cal D})]^2} + \Vert \boldsymbol{g} \Vert_{[H^{1/2}_{00}(\partial\D_\text{in})]^2} + \Vert  \boldsymbol{h} \Vert_{[L^2(\partial\D_\text{out})]^2}  \right) \; ,
\end{equation*}
and $\|\ff\|_{[L^2(\D)]^2} \lesssim \|\ff\|_{[W^{1,1}(\D)]^2}$.
A similar argument, together with \eqref{eq:energy-z}, applies to the pair $(\zz-\zz_{\boldsymbol{\Phi}},q)$, where $\zz_{\boldsymbol{\Phi}}$ has been defined in Proposition \ref{L:bound-J}, and gives
\begin{equation}\label{eq:bound-zqHs}
\Vert \zz \Vert_{\left[ H^{1+s}({\cal N}_\Gamma) \right]^2} + \Vert q \Vert_{H^{s}({\cal N}_\Gamma) } \lesssim \Vert \boldsymbol{\Phi} \Vert_{[H^1(\Omega_0)]^2} \lesssim 1. 
\end{equation}

Since the tubular neighborhoods and constants in the previous bounds are independent of $\theta\in I_{\epsilon_0}$, choosing $s$ satisfying $\frac12 < s < \alpha$ we deduce that the traces of $\T\nn$, $\Gr^s \uu$, and $\partial_n \zz$ are in $[L^2(\Gamma)]^2$, with norms controlled by the right-hand sides of the bound for $(\uu,p)$ and $(\zz,q)$. To obtain the desired estimate, we further observe that $\zz\in[ H^{1+s}({\cal N}_\Gamma)]^2\subset[L^\infty({\cal N}_\Gamma)]^2$ and $\ff \in [W^{1,1}(\D)]^2 \subset [L^1(\Gamma)]^2$.
 \end{proof}

\begin{corollary}[Lipschitz property of torque]\label{cor:der-tau} 
Let $0<\epsilon\le\epsilon_0/2$ be fixed. Under the regularity assumptions of Proposition \ref {prop:shapeder2}, the torque functional $\tau=\tau(\theta)$ is differentiable for all $\theta \in I_{\epsilon_0}$ with \vskip -0.6cm
\begin{equation}\label{eq:der-tau}
\frac{{\rm d} \tau}{{\rm d} \theta} = \delta J[\Gamma; \VVe] \, ,
\end{equation}
and $\frac{{\rm d} \tau}{{\rm d} \theta}$ is bounded in $I_{\epsilon_0}$.
\end{corollary} 
\begin{proof}
If we rotate $\Gamma=\Gamma(\theta)$ by an angle $\Delta\theta$, leading to $\Gamma'=\Gamma(\theta+\Delta\theta)$, a generic point $\xv = \ov + r\,\e_\omega$ on $\Gamma$ is sent to the point $\xv' = \ov + r\,\e_{\omega+\Delta\theta} \in \Gamma'$. Hence,
$$
\lim_{\Delta\theta \to 0} \frac{\xv' - \xv}{\Delta\theta} = r\, \e_\omega^\perp = \VVe (r, \omega).
$$
In view of the equality $\tau(\theta)=J[\Gamma]$ in \eqref{eq:def-J}, we readily deduce
$$
\lim_{\Delta \theta \to 0} \frac{\tau(\theta+\Delta\theta) - \tau(\theta)}{\Delta\theta} = \delta J[\Gamma; \VVe] \, ,
$$
whence $\tau$ is differentiable in $I_{\epsilon_0}$. Proposition \ref{prop:shapeder2} gives the uniform bound in $I_{\epsilon_0}$.
\end{proof}

We point out that the Lipschitz bound in Corollary \ref{cor:der-tau} might depend on $\epsilon$. To see this, consider the limit $\epsilon\to0$ in which the leaflet $\Gamma$ degenerates into a segment and the asymptotic behaviour near the tip of the functions $\uu$ and $\zz$, in the expression of $\delta J[\Gamma,\boldsymbol{V}]$ of Proposition \ref{prop:shapeder1}, becomes
\[
\uu = \rho^{1/2} \boldsymbol{U}(\phi),
\quad
\zz = \rho^{1/2} \boldsymbol{Z}(\phi),
\]
with both $\boldsymbol{U}$ and $\boldsymbol{Z}$ smooth. The first term in $\delta J[\Gamma,\boldsymbol{V}]$ involves the computation of integrals of the form
\[
\int_{\Gamma_1 \cup \Gamma_2} \partial_n \uu \cdot \partial_n \zz \approx
\int_{\Gamma_1 \cup \Gamma_2} \rho^{-1} d\rho = \infty,
\]
unless special cancellation occurs and the principal value is finite. This explains why the current technical tools at hand are inadequate to derive differentiability of $\tau(\theta)$ for $\epsilon=0$. We content ourselves with continuity in Section \ref{S:shape-derivative}.

\subsection{Case $\epsilon = 0$: Continuity of {$J[\Gamma]$}}\label{S:shape-derivative}
%
We established in Proposition \ref{L:bound-J} that $J[\Gamma]$ is bounded for $\epsilon=0$
uniformly in $\theta\in I_{\epsilon_0}$. We now prove that $J[\Gamma]$ is uniformly continuous in $I_{\epsilon_0}$.

Our departing point is the expression \eqref{eq:expr2-J} for $J[\Gamma]$. Using that $\dd\,\uu=0$ and integrating by parts we rewrite the first term as follows:
\begin{align*}
-\nu  \int_{\Omega} \Gr^s \uu : \Gr^s \zz
& = - \int_\Omega \Gr^s \uu : \Tz = - \int_\Omega \Gr \uu : \Tz
\\ 
& = \int_\Omega \uu \cdot \dd \, \Tz - \int_{\partial\Omega} \uu \cdot \Tz \nn
= -\int_{\partial\D_\text{in}} \boldsymbol{g} \cdot \Tz \nn,
\end{align*}
because $\dd\,\Tz=0$. Consequently, we get the equivalent form of \eqref{eq:expr2-J}
\[
J[\Gamma] = -\int_{\partial\D_\text{in}} \boldsymbol{g} \cdot \Tz \, \nn + \int_{\Omega} \ff \cdot \zz  + \int_{\partial\D_N} \boldsymbol{h} \cdot \zz  \, ,
\]
valid for $0\le \epsilon\le \frac{\epsilon_0}{2}$ and linear in $(\zz,q)$. We denote by $J[\Gamma_\epsilon]$ the functional corresponding to the pair $(\zz_\epsilon,q_\epsilon)$ for $\epsilon>0$, and by $J[\Gamma_0]$ the functional for the pair $(\zz_0,q_0)$ and $\epsilon=0$. The pair $(\ww,s)=(\zz_0 - \zz_\epsilon, q_0 - q_\epsilon)$ satisfies the Stokes system
\begin{equation*}
\left\{
\begin{aligned}
 -  \DD \, \boldsymbol{T}(\ww,s) & = \oo  & &\text{in $\Omega_\epsilon$,} \\
   \dd \, \ww & = 0 & &\text{in $\Omega_\epsilon$,}
\end{aligned}
\right.
\qquad
\left\{
\begin{aligned}
 \ww &= \zz_0 - \Fi  & &\text{on $\Gamma_\epsilon$,} \\
 \ww &= \oo                      & &\text{on $\partial \D_D$,} \\
 \boldsymbol{T}(\ww,s) \, \nn &= \oo   & &\text{on $\partial \D_N$\,,} 
\end{aligned}
\right.
\end{equation*}
where $\Omega_\epsilon = \D \setminus L_\epsilon, \Gamma_\epsilon=\partial L_\epsilon$. We now have an explicit formula for the error
\begin{equation}\label{eq:J-Je}
  J[\Gamma_0] - J[\Gamma_\epsilon] =
  -\int_{\partial\D_\text{in}} \boldsymbol{g} \cdot \boldsymbol{T}(\ww,s) \, \nn + \int_{\Omega_\epsilon} \ff \cdot \ww  + \int_{\partial\D_N} \boldsymbol{h} \cdot \ww + \int_{\Omega_0\setminus\Omega_\epsilon} f \cdot \zz_0.
\end{equation}
This leads to the following statement.

\begin{proposition}[continuity of {$J[\Gamma_0]$}]\label{P:continuity}
The following error estimate
\begin{equation}\label{eq:J0-Je}
\big| J[\Gamma_0] - J[\Gamma_\epsilon] \big| \lesssim \epsilon^{\frac12} |\log\epsilon|^{\frac14}
\end{equation}
is valid uniformly in $\theta\in I_{\epsilon_0}$. Therefore, the function $\tau(\theta)=J[\Gamma_0(\theta)]$ is uniformly continuous in $I_{\epsilon_0}$.
\end{proposition}
\begin{proof}
Uniform continuity of $J[\Gamma_0(\theta)]$ is a consequence of \eqref{eq:J0-Je} and the uniform continuity of $J[\Gamma_\epsilon(\theta)]$ from Corollary \eqref{cor:der-tau}. To show \eqref{eq:J0-Je}, we first note that the estimate \eqref{eq:bound-zqHs} of Proposition \ref{prop:shapeder2} is valid for $\epsilon=0$, and so for $\zz_0$, provided $0<s<\frac12$. We thus deduce $\|\zz_0\|_{[L^\infty({\cal N}_\Gamma)]^2} \lesssim \|\zz_0\|_{[H^{1+s}({\cal N}_\Gamma)]^2} \lesssim 1$ for $0<s<\frac12$ along with \looseness=-1
\[
\Big| \int_{\Omega_0\setminus\Omega_\epsilon} f \cdot \zz_0  \Big| \lesssim \|\ff\|_{L^2(\Omega_0)}
\big|\Omega_0\setminus\Omega_\epsilon \big|^{\frac12} \lesssim \epsilon^{\frac12}.
\]
Therefore, to obtain \eqref{eq:J0-Je} it suffices to prove the error estimate
\begin{equation}\label{eq:H1/2}
  \| \ww \|_{H^{\frac12}_{00}(\Gamma_\epsilon)} \lesssim \epsilon^{\frac12} |\log\epsilon|^{\frac14},
\end{equation}
whence the extension of $\ww$ by $\oo$ to the rest of $\partial\Omega_\epsilon$ satisfies $\|\ww \|_{H^{\frac12}(\partial\Omega_\epsilon)} \lesssim \epsilon^{\frac12} |\log\epsilon|^{\frac14}$.
In fact, this controls $\|\ww\|_{[H^1(\Omega_\epsilon)]^2}$ and takes care of the second and third terms in \eqref{eq:J-Je}. 
{The first term is more problematic, but interpreting the integral as a duality in $H^{1/2}_{00}(\partial\D_\text{in})$ and recalling that $\DD \, \boldsymbol{T}(\ww,s) = \oo$, it is sufficient to bound $\|\boldsymbol{T}(\ww,s)\|_{L^2(\Omega_\varepsilon)}$ in terms of $\|\ww\|_{[H^1(\Omega_\epsilon)]^2}$. This in turn follows from Lemma \ref{L:inf-sup}.}

To prove \eqref{eq:H1/2}, we decompose $\Gamma_\epsilon$ in three disjoint pieces $\Gamma_i=\Gamma_{\epsilon,i}$, the straight sides $\Gamma_1,\Gamma_2$ and the circular arc $\Gamma_3$, but omit writing a subscript $\epsilon$ for simplicity. It turns out to be convenient to represent the geometry as follows: let the leaflet $\Gamma_0=\{(x,0): 0\le x\le 1\}$ with the tip at the origin and the hinge at $(1,0)$, and let
\[
\Gamma_1=\{(x,y): y=\alpha(1-x)\},
\qquad
\Gamma_3=\{(x,y):  x \approx y^2\},
\]
with $\alpha = \tan \epsilon \approx \epsilon$.
Since $\zz_0-\Fi$ is the solution of a Stokes equation with smooth right hand side and $\zz_0-\Fi=\oo$ on $\Gamma_0$, the function $\ww = \zz_0-\Fi$ on $\Gamma_\epsilon$ exhibits the singular behavior of a Stokes velocity near the tip
\[
\ww(r,\phi) = r^{\frac12} \boldsymbol{W}(\phi),
\]
with $\boldsymbol{W}$ smooth satisfying $|\boldsymbol{W}(\phi)| \approx \phi$ near $\phi=0$, plus a regular $H^2$-component. The proof now splits into four steps.

\smallskip\noindent
1. {\it Estimate of $\|\ww\|_{H^{\frac12}(\Gamma_1)}$.}
Note that on $\Gamma_1$ (or $\Gamma_2$) we have $r=\big(x^2+\alpha^2(1-x)^2\big)^{\frac12}$ and $\phi=\arctan\big(\alpha\frac{1-x}{x}\big)$, whence the following approximations are
valid
\[
r \approx
  \begin{cases}
  \epsilon & 0 < x < \epsilon, \\
  x & \epsilon < x \le 1,
  \end{cases}
  \qquad
  \phi \approx
  \begin{cases}
  1 & 0 < x < \epsilon, \\
  \epsilon \frac{1-x}{x} & \epsilon < x \le 1.
  \end{cases}
\]
We decompose the interval $(0,1)$ dyadically, namely let $I_0=[0,\epsilon)$ and $I_k = [\epsilon 2^{k-1}, \epsilon 2^k)$ for all $1\le k \le K\approx |\log \epsilon|$. Since the unit tangent vector to $\Gamma_\epsilon$ is $\ttt = (\cos\epsilon, -\sin\epsilon)$, we can estimate $\ww$ and $\partial_\ttt\ww$ on $I_k$ as follows:
\[
|\ww| \approx\epsilon \big(\epsilon 2^k\big)^{-\frac12},
  ~ |\partial_\ttt\ww| \approx \epsilon \big(\epsilon 2^k \big)^{-\frac32}
  \quad\Rightarrow\quad
  \int_{I_k} |\ww|^2 \approx \epsilon^2,
  ~ \int_{I_k} |\partial_\ttt\ww|^2 \approx 2^{-2k}.
\]
This leads to
\begin{equation*}
  \int_{\Gamma_1} |\ww|^2 \approx \epsilon^2 K \approx \epsilon^2 |\log\epsilon|,
  \quad
  \int_{\Gamma_1} |\partial_\ttt\ww|^2 \approx \sum_{k=0}^K 2^{-2k} \approx 1,
\end{equation*}
and combined with space interpolation yields
\[
\|\ww\|_{H^{\frac12}(\Gamma_1)} \approx \Big( \|\ww\|_{L^2(\Gamma_1)}\|\ww\|_{H^1(\Gamma_1)} \Big)^{\frac12} \approx \epsilon^{\frac12} |\log\epsilon|^{\frac14}.
\]

\noindent
2. {\it Estimate of $\|\ww\|_{H^{\frac12}(\Gamma_3)}$.} Since $x \approx y^2$ on $\Gamma_3$,
we may approximate $|\ww| \approx |y|^{\frac12}$. This function is known to belong to $H^{\frac12}(-1,1)$ so to get its $H^{\frac12}$-norm in the interval $(-\epsilon,\epsilon)$ we simply use a scaling argument. This gives
\[
\|\ww\|_{H^{\frac12}(\Gamma_3)} \lesssim \epsilon^{\frac12}.
\]

\noindent
3. {\it Estimate of $\|\ww\|_{H^{\frac12}(\Gamma_\epsilon)}$.} We have estimates for the
$H^{\frac12}$-norms on the disjoint pieces $\Gamma_1,\Gamma_2,\Gamma_3$, but this does
not give an estimate for $H^{\frac12}(\Gamma_\epsilon)$ because fractional norms are not
subadditive with respect to domain partitions. To get around this issue, we resort to
a location result of Faermann \cite{Faermann}, which states that domains should overlap
with an amount of overlap commensurate with their size. The global $H^{\frac12}$-seminorm square is then bounded by the $H^{\frac12}$-seminorms of the individual pieces plus the $L^2$-norms on each piece scaled
by the reciprocal of the overlap.

In our case, we simply extend the domain $\Gamma_1$ to the upper quarter of $\Gamma_3$, say $\Gamma_3^u$, thus
avoiding to include a neighborhood of the origin.
We next argue that to compute the $H^1$-seminorm in the extended domain it suffices
to add the new piece on $\Gamma_3$ because the function $\ww$ has traces that agree
on both sides of $\Gamma_1\cap\Gamma_3$. Therefore, since $\partial_\ttt \ww \approx y^{-\frac12}$, we get
\[
\int_{\Gamma_3^u} |\partial_\ttt \ww|^2 \approx \int_{\epsilon/2}^\epsilon |\partial_\ttt \ww|^2 dy
\approx \int_{\epsilon/2}^\epsilon y^{-1} dy = \log \epsilon - \log \frac{\epsilon}{2}
= \log 2.
\]
It remains to estimate the scaled $L^2$-norms, namely
\[
\frac{1}{\epsilon} \int_{\Gamma_1}|\ww|^2 \approx \epsilon |\log\epsilon|,
\quad
\frac{1}{\epsilon} \int_{\Gamma_3}|\ww|^2 \approx \epsilon,
\]
because $|\ww|\lesssim\epsilon^{\frac12}$ on $\Gamma_3$. Coupling the three steps gives
the asserted estimate.

\medskip\noindent
4. {\it Estimate of $\|\ww\|_{H^{\frac12}_{00}(\Gamma_\epsilon)}$.}
According to Theorem 1.5.2.3 of \cite{Grisvard}, We need to estimate the quantity $\int_{\Gamma_1} |\ww(z)|^2\textrm{dist} (z,z_0)^{-1} d\sigma(z)$, where $z_0=(1,0)$ is the hinge point of $\Gamma_\epsilon$. To do so, we simply refine the expression of $\ww(z)$ for $z=(x,y)\in\Gamma_1$ from Step 1 in the sense that $\ww(z) \approx \epsilon x^{\frac12}(1-x)$. Since $\textrm{dist} (z,z_0)\approx 1-x$, this yields
\[
\int_{\Gamma_1} \frac{|\ww(z)|^2}{\textrm{dist} (z,z_0)} d\sigma(z)  \approx
\sum_{0\le k\le K} \int_{I_k} \frac{|\ww(z)|^2}{1-x} dx \approx
\epsilon^2 K \approx \epsilon^2 |\log\epsilon|.
\]
Adding this bound to the estimate for $\|\ww\|_{H^{\frac12}(\Gamma_\epsilon)}^2$ of Step 3 concludes the proof.
\end{proof}

\subsection{Solvability}\label{S:solvability}
%
A simple consequence of Corollary \ref{cor:der-tau} and Proposition \ref{P:continuity} is the existence of a solution for the problem under consideration for all $\epsilon\ge0$, provided a suitable but reasonable condition on the spring elastic torque is assumed. This is tackled next.

\begin{proposition}[existence of solution] \label{rm:limit}
Let the spring angular momentum $\kappa$ introduced in \eqref{eq:def-kappa} be continuous in the interval $I_{\epsilon_0}^{\circ}=(-\frac\pi2+\epsilon_0,\frac\pi2-\epsilon_0)$ and satisfy
\begin{equation}\label{eq:stiff}
\lim_{\theta \rightarrow \pm(\tfrac\pi2-\epsilon_0)} \kappa(\theta) = \pm \infty \; .
\end{equation}
Then, under the regularity assumptions of Proposition \ref{prop:shapeder2}, the balance equation \eqref{eq:leaflet_primale} (or equivalently \eqref{eq:balance}) has at least one solution in $I_{\epsilon_0}^{\circ}$ for all $0\le \epsilon\le\frac{\epsilon_0}{2}$. Moreover, if $\epsilon>0$ and $\kappa$ is differentiable with a sufficiently large derivative depending on $\epsilon$, then the solution of \eqref{eq:leaflet_primale} is unique.
\end{proposition}
\begin{proof} 
{The function $\tau(\theta)$ is bounded in $I_{\epsilon_0}$ in view of Lemma \ref{L:bound-J} and is continuous according to Corollary \ref{cor:der-tau} and Proposition \ref{P:continuity}. Since $\kappa(\theta)$ is continuous in $I_{\epsilon_0}^{\circ}$} and tends to $\pm\infty$ at the end points, there is clearly a solution of \eqref{eq:leaflet_primale}. On the other hand, if $\epsilon>0$, then Corollary \ref{cor:der-tau} states that $\tau(\theta)$ is Lipschitz in $I_{\epsilon_0}$ with a constant that might depend on $\epsilon$. If the derivative of $\kappa(\theta)$ exceeds this constant, then the function $\kappa(\theta)-\tau(\theta)$ is strictly increasing in $I_{\epsilon_0}$ and thus the solution is unique.
\end{proof}

Note that the blow-up condition \eqref{eq:stiff} can be interpreted as a stiffening of the spring or as originated by a penalty approximation of the contact condition among the leaflet and the vessel walls.

\section{Virtual Element discretization}
\label{sec:VEMdis}

In this section we first briefly recall the VEM discretization of the stationary fluid equations, then present the discrete coupled problem and finish with the proposed nonlinear iteration scheme.

\subsection{Virtual Elements for the Navier-Stokes equation}
\label{sub:3.1}
%
We now describe briefly various tools from the virtual element technology; we refer the interested reader to the papers \cite{Antonietti-BeiraodaVeiga-Mora-Verani:2014,Stokes:divfree,Vacca:2018,NavierStokes:divfree}.
In particular we recall that the proposed Virtual Elements family presents interesting advantages, such as its flexibility in terms of meshes and its capability of yielding a divergence-free discrete velocity solution.
For the sake of simplicity, here we address the lowest-degree case\cite{Antonietti-BeiraodaVeiga-Mora-Verani:2014} 
that delivers first-order accuracy for both velocity and pressure.
Similar constructions can be used for higher-order schemes 
\cite{Stokes:divfree,Vacca:2018,NavierStokes:divfree}.
Let $\{ \Ph \}_h$ be a sequence of partitions of $\mathcal{D}$ into general polygonal elements $E$ with
\[
 h_E := {\rm diam}(E) , \quad
h := \sup_{E \in \Ph} h_E .
\]
We suppose that for all $h$, each element $E$ in $\Ph$ fulfils the following assumptions:
\begin{description}
\item [$\boldsymbol{(A1)}$] $E$ is star-shaped with respect to a ball $B_E$ of radius $ \geq\, \rho \, h_E$, 
\item [$\boldsymbol{(A2)}$] the distance between any two vertexes of $E$ is $\geq \, \rho  \, h_E$, 
\end{description}
where $\rho$ is a  positive constant. 
We remark that the hypotheses listed above are classical in the virtual element approach (see for instance \cite{volley,Ahmed-et-al:2013}).
Assumption $\boldsymbol{(A2)}$ can be further relaxed, as investigated in ~\cite{BdV-Lovadina-Russo,Brenner-Sung}, 
allowing for more general cases such as meshes with arbitrarily small edges (with respect to the element diameter).
In contrast, very few theoretical results about avoiding assumption $\boldsymbol{(A1)}$ exist currently in the literature; see for instance \cite{LongChen}. The latter relates to anisotropic elements.

\subsection*{Virtual Element Spaces}
%
On each element $E \in \Ph$ we define the following finite dimensional local virtual
spaces of velocities
\begin{multline}
\label{eq:V_hE}
\VV_h^E := \biggl\{  
\boldsymbol{v} \in [H^1(E)]^2 \ : \ 
 \boldsymbol{v}_{|{\partial E}} \in \B_1(\partial E) \,, \quad
 {\rm div} \, \boldsymbol{v} \in \Pk_{0}(E)\,, \biggr.
\\[-10pt]
\bigl. 
 \hskip -1cm \boldsymbol{\Delta}    \boldsymbol{v}  +  \nabla s = \boldsymbol{0} \quad \text{ for some $s \in L_0^2(E)$} \biggr\} \,, 
\end{multline}
\vskip -0.3cm
\noindent with
\[
\B_1(\partial E) := \bigl\{ 
\boldsymbol{v}_{|{\partial E}} \in [C^0(\partial E)]^2 : \ 
\vv_{|_e} \cdot \ttt_e \in \Pk_1(e), \quad 
\vv_{|_e} \cdot \nn_e \in \Pk_2(e) \quad
\text{for all $e \in \partial E$}
\bigr\} \,,
\]
where  $\nn_e$ is the outward unit normal to $E$ and  $\ttt_e$ is the tangent unit vector defined as the counterclockwise rotation of $\nn_e$ by $\pi/2$. 
All the operators and equations in \eqref{eq:V_hE} are to be interpreted in the distributional sense. 
It is easy to realize that $[\Pk_1(E)]^2 \subseteq \VV_h^E$ and this will guarantee the optimal approximation property of the space.
The definition of $\VV_h^E$ above is associated to a Stokes-like variational problem on $E$; in particular
we remark that all functions $\vv \in \VV_h^E$ are uniquely determined by their boundary values $\vv_{|\partial E} \in \B_1(\partial E)$ because 
${\rm div}\, \vv = \frac{1}{|E|} \int_{\partial E} \vv \cdot \nn_e \, {\rm d}s$.
This leads to the following result \cite{Antonietti-BeiraodaVeiga-Mora-Verani:2014}.
\begin{proposition}[dimension and DoFs]
\label{prp:dofs}
Let $\VV_h^E$ be the space defined in \eqref{eq:V_hE}. Then \vskip -.6cm
\begin{equation*}
\dim \, \VV_h^E  = 
\dim \, \B_1(\partial E)
= 3 \, n_E  
\end{equation*}
where $n_E$ is the number of vertexes of $E$. Moreover the following linear forms $\boldsymbol{D_X}$, which split into two subsets (see Fig. \ref{fig:dofsloc}), provide a set of DoFs for $\VV_h^E$:
\begin{itemize}
\item $\boldsymbol{D_X1}$:  the values of $\boldsymbol{v}$ at the vertices of the polygon $E$,
\item $\boldsymbol{D_X2}$: the values of the normal components $\vv \cdot \nn_e$
at the midpoint of each edge of $E$.
\end{itemize}
\end{proposition}
\begin{figure}[!h]
\center{
\includegraphics[scale=0.18]{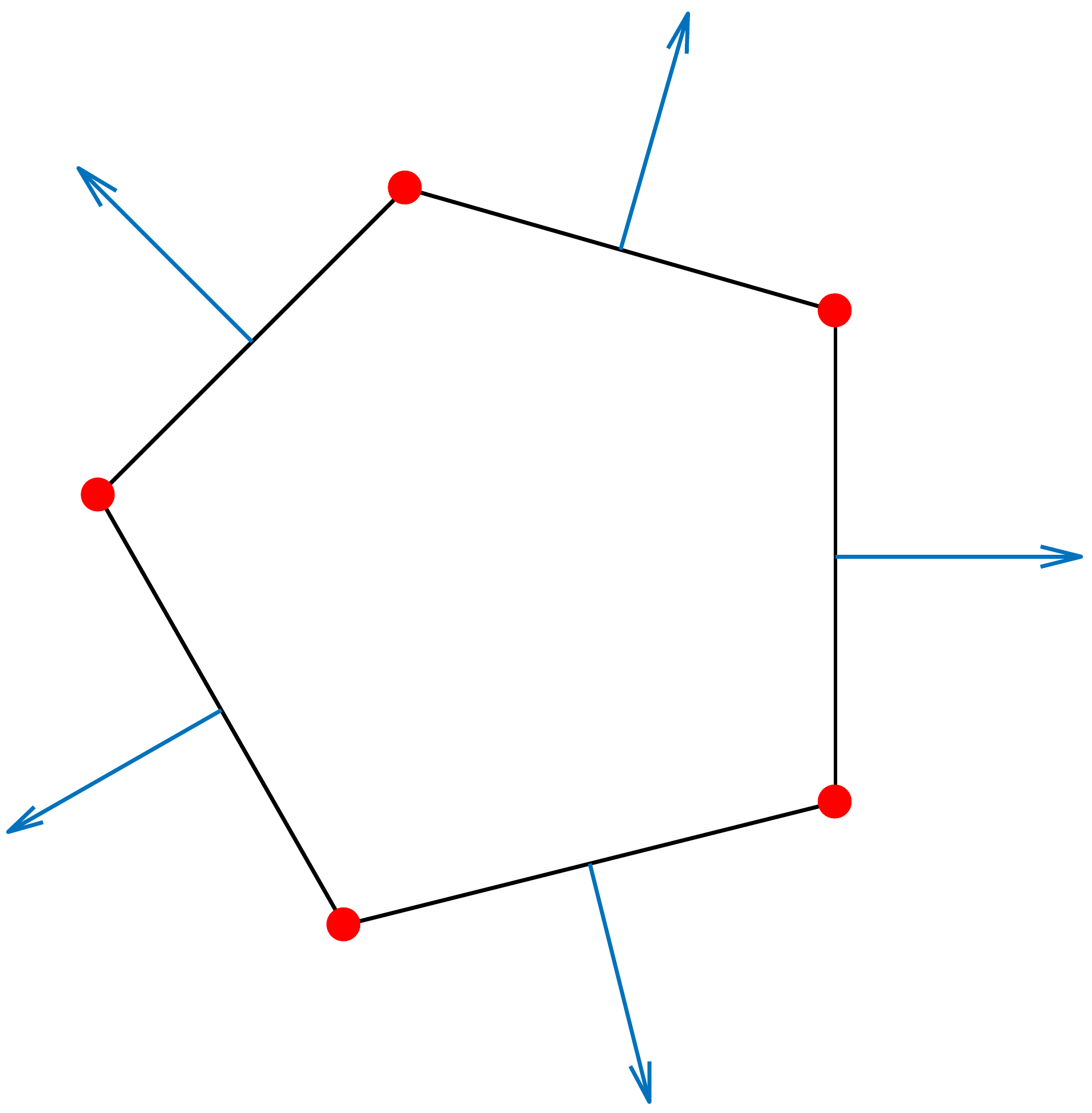}
\qquad \qquad \qquad
\includegraphics[scale=0.18]{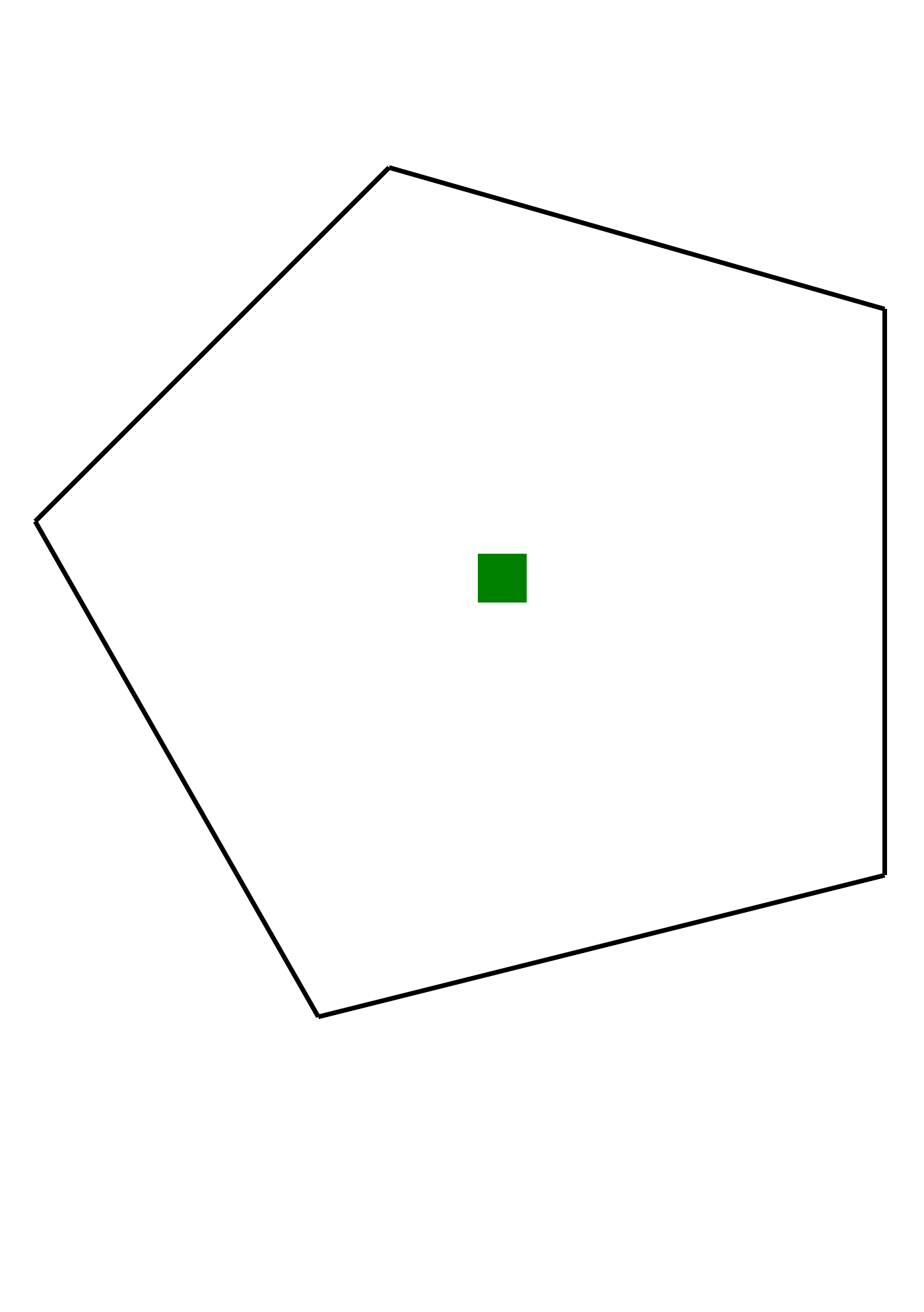}
\caption{{ 
Left: velocity DoFs $\boldsymbol{D_X1}$ denoted by dots, $\boldsymbol{D_X2}$ denoted by arrows.
Right: pressure DoF $\boldsymbol{D_Q}$ denoted by a square.}}
\label{fig:dofsloc}
}
\end{figure}
\noindent We highlight that the degrees of freedom $\boldsymbol{D_X1}$-$\boldsymbol{D_X2}$ are directly related to the piecewise polynomial boundary space $\B_1(\partial E)$: linear
tangent component and quadratic normal component on each edge $e$.

For what concerns pressures, we take the standard finite dimensional space
$Q_h^E := \Pk_{0}(E) $
and the corresponding degree of freedom $\boldsymbol{D_Q}$ is one per element, given by the value of the function on the element.

Finally, we define the global virtual element spaces as
\begin{equation}
\label{eq:V_h}
\VV_h := \{ \boldsymbol{v} \in [H^1(\Omega)]^2  ~ : \quad 
\boldsymbol{v}_{|E} \in \VV_h^E  \quad \text{for all $E \in \Ph$} \}
\end{equation} 
and 
\begin{equation}
\label{eq:Q_h}
Q_h := \{ q \in L^2(\Omega) ~ : \quad q_{|E} \in  Q_h^E \quad \text{for all $E \in \Ph$}\},
\end{equation}
with the obvious associated sets of global degrees of freedom. In view of the degrees of freedom $\boldsymbol{D_X1}$ and $\boldsymbol{D_X2}$ from Proposition \ref{prp:dofs}, a simple computation shows that
\begin{equation*}
\dim \, \VV_h =  2 \,n_V +  n_e
\qquad
\text{and}
\qquad
\dim \, Q_h = n_P,
\end{equation*}
where $n_P$ is the number of elements, $n_e$, $n_V$ is the number of  edges and vertexes in $\Ph$.
We highlight the fundamental property of the proposed virtual elements, namely \vskip -.6cm
\begin{equation}\label{eq:divfree}
{\rm div}\, \VV_h\subseteq Q_h ,
\end{equation}
a key property that will lead to a {\it divergence-free} discrete solution.

\begin{remark}
\label{rm:k2}
In this paper we limit ourselves to present the lowest-order Virtual Element Method $(k=1)$ for the Navier--Stokes equation.
However, in order to compare and validate the performance of the proposed scheme, we also show the results obtained by employing the VEM of order $k=2$ in Section \ref{sec:tests}.
For completeness in Fig. \ref{fig:dofsloc2} we display the DoFs diagram for such VEM as well. For a deeper analysis of higher-order VEMs for the Navier--Stokes equation we refer to \cite{Stokes:divfree,bricks}.
\begin{figure}[!h]
\center{
\includegraphics[scale=0.18]{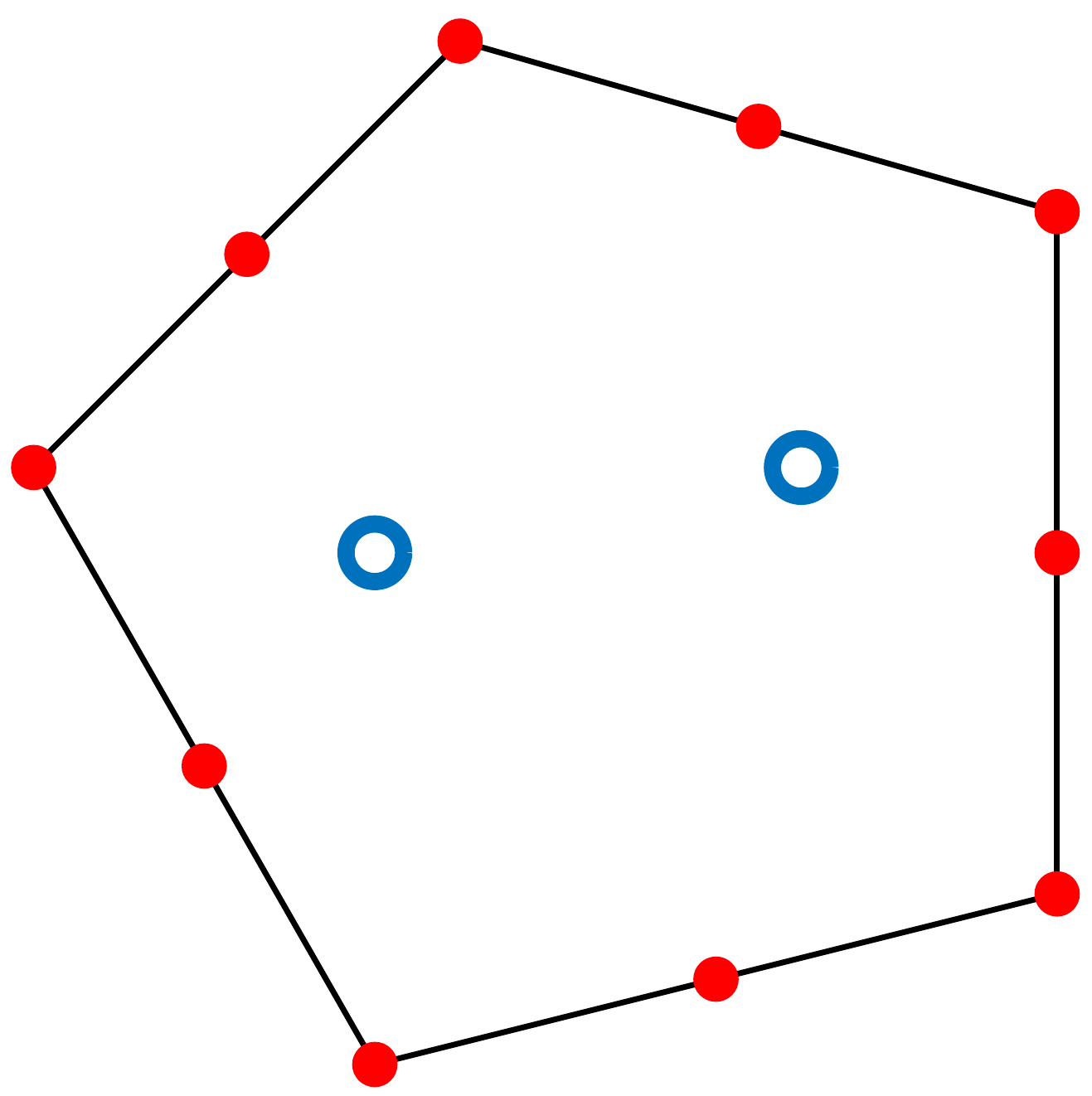}
\qquad \qquad \qquad
\includegraphics[scale=0.18]{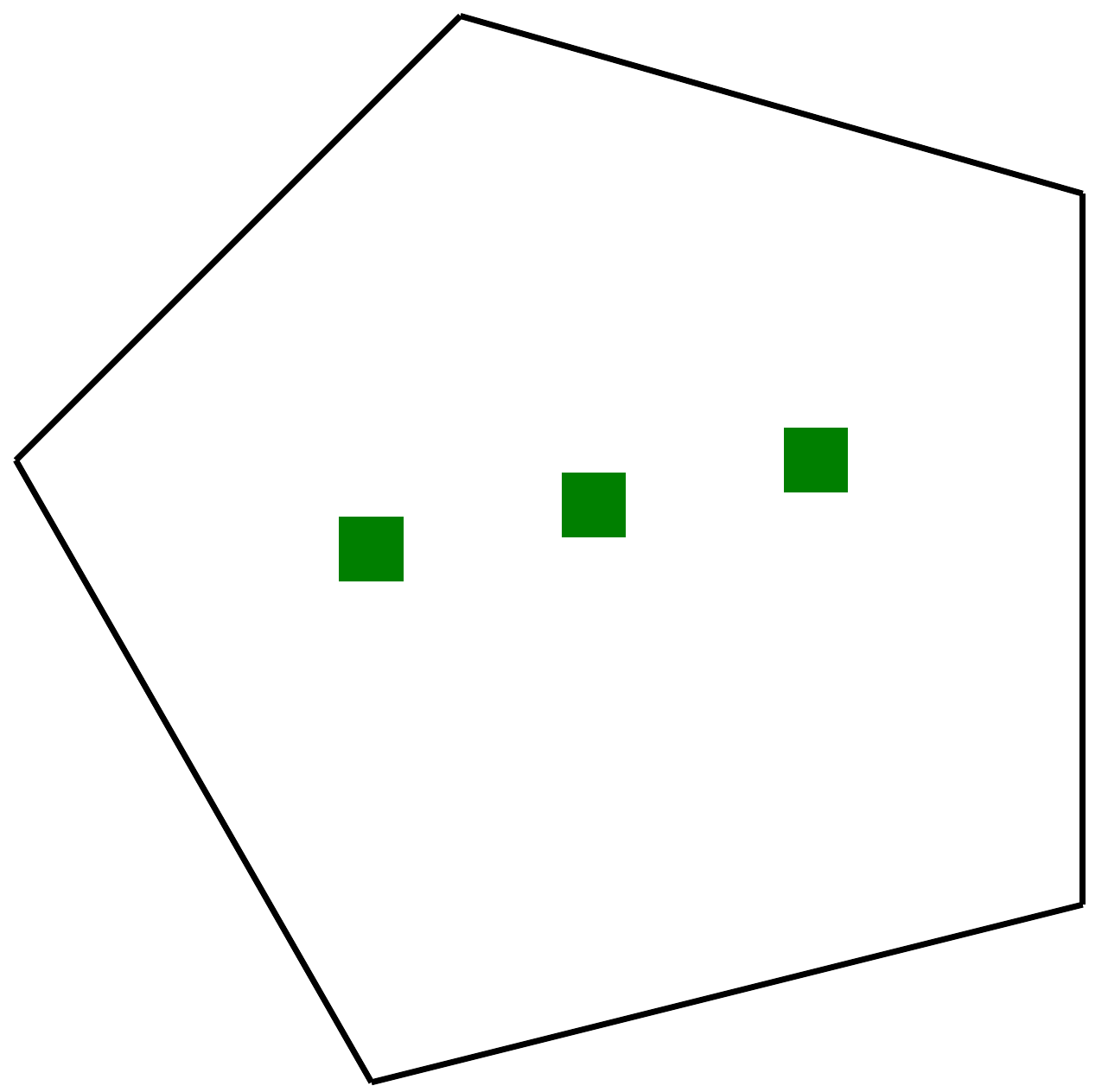}
\caption{{ 
VEM of order $k=2$.
Left: velocity DoFs.
Right: pressure DoFs.}}
\label{fig:dofsloc2}
}
\end{figure}
\end{remark}

\subsection*{Multi-linear forms}

In what follows we briefly recall  the basic  steps in the construction of  discrete versions of the bilinear forms $a(\cdot, \cdot)$ and $b(\cdot, \cdot)$ given in \eqref{eq:forma a} and \eqref{eq:forma b} and trilinear form $c(\cdot; \cdot, \cdot)$ in \eqref{eq:forma c}.
%
%
First of all, we decompose these forms as well as  the norms $\|\cdot\|_{\VV}$, $\|\cdot \|_Q$ into local contributions, by defining 
\[
\begin{aligned}
a (\boldsymbol{u},  \boldsymbol{v}) &=: \sum_{E \in \Ph} a^E (\boldsymbol{u},  \boldsymbol{v}) \qquad &\text{for all $\boldsymbol{u},  \boldsymbol{v} \in \VV$}\,,
\\
b (\boldsymbol{v},  q) &=: \sum_{E \in \Ph} b^E (\boldsymbol{v},  q) \qquad &\text{for all $\boldsymbol{v} \in \VV$ and $q \in Q$}\,,
\\
c(\ww; \, \uu, \vv) &=: \sum_{E \in \Ph} c^E(\ww; \, \uu, \vv) \qquad &\text{for all $\ww, \uu, \vv \in \VV$}\,.
\end{aligned}
\]
and 
\begin{equation*}
\|\boldsymbol{v}\|_{\VV}^2 =: \sum_{E \in \Ph} \|\boldsymbol{v}\|^2_{\VV, E}\,\, \text{for all $\boldsymbol{v} \in \VV$,} \quad \|q\|_Q^2 =: \sum_{E \in \Ph} \|q\|^2_{Q, E} \,\, \text{for all $q \in Q$.}
\end{equation*}
Concerning the form $b(\cdot, \cdot)$, we simply  observe that for all $\boldsymbol{v} \in \VV_h$, $q \in Q_h$ it holds 
\begin{equation}\label{bhform}
b^E(\boldsymbol{v}, q) = 
 \int_E {\rm div} \, \boldsymbol{v} \, q \,{\rm d}E =
 q_{|E} \, \int_{\partial E} \vv \cdot \nn_e \,{\rm d}s \,,
\end{equation}
a quantity that  is exactly computable from the degrees of freedom $\boldsymbol{D_X1}$, $\boldsymbol{D_X2}$ and $\boldsymbol{D_Q}$, therefore we do not introduce any approximation of the bilinear form. 
We now define discrete versions of the forms $a(\cdot, \cdot)$  and  $c(\cdot; \,\cdot, \cdot)$, that need  to be dealt with in a more careful way.
First of all, we note that for an arbitrary triplet $(\ww, \, \uu, \, \vv )\in \big[ \VV_h^E \big]^3$, the quantities $a^E(\boldsymbol{u}, \boldsymbol{v})$ and $c^E(\ww; \, \uu,  \vv)$
are not computable.  
Therefore, following a standard procedure in the VEM framework \cite{volley,Ahmed-et-al:2013}, 
for every element $E \in \Ph$ we introduce the following useful polynomial projections:
\begin{enumerate}[$\bullet$]
\item the $\boldsymbol{H^1}$ \textbf{semi-norm projection} ${\Pi}_{1}^{\Gr^s,E} \colon \VV \to [\Pk_1(E)]^2$, defined for all $\vv \in \VV$ by 
\begin{equation}
\label{eq:Ps_k^E}
\left\{
\begin{aligned}
& \int_E \Gr^s \boldsymbol{q}_1 : \Gr^s (\boldsymbol{v}- \, {\Pi}_{1}^{\Gr^s,E}   \boldsymbol{v}) \, {\rm d} E = 0 &\qquad  \text{for all $\boldsymbol{q}_1 \in [\Pk_1(E)]^2$,} \\
& \int_{\partial E} \boldsymbol{x}^{\perp} \cdot (\boldsymbol{v}- \, {\Pi}_{1}^{\Gr^s,E}   \boldsymbol{v}) \, {\rm d} s = 0 \, , &\qquad \text{where $\boldsymbol{x}^{\perp} := (-y, \, x)^T$,} \\
& \int_{\partial E} (\boldsymbol{v}- \, {\Pi}_{1}^{\Gr^s,E}   \boldsymbol{v}) \, {\rm d} s = \boldsymbol{0} \, ;
\end{aligned}
\right.
\end{equation} 

\item the $\boldsymbol{L^2}$\textbf{-projection for scalar functions} $\Pi_j^{0, E} \colon L^2(E) \to \Pk_j(E)$, given by
\begin{equation}
\label{eq:P0_k^E}
\int_E q_j (v - \, {\Pi}_{j}^{0, E}  v) \, {\rm d} E = 0 \qquad  \text{for all $v \in L^2(E)$  and for all $q_j \in \Pk_j(E)$,} 
\end{equation} 
with obvious extension for vector functions 
$\Pi_j^{0, E} \colon [L^2(\Omega)]^2 \to [\Pk_j(E)]^2$, 
and tensor functions 
$\boldsymbol{\Pi}_{j}^{0, E} \colon [L^2(E)]^{2 \times 2} \to [\Pk_{j}(E)]^{2 \times 2}$ (for $j=0,1$). 
\end{enumerate}
\begin{remark}[projections and computability]\label{prp:projections}
The operator $\PN$ is well defined because the last two conditions in \eqref{eq:Ps_k^E} account for the kernel of $\nabla^s$.
In \cite{Antonietti-BeiraodaVeiga-Mora-Verani:2014,Stokes:divfree} 
it has been shown that
the DoFs $\boldsymbol{D_X}$ are sufficient to compute exactly 
\[
\PN \colon \VV_h^E \to [\Pk_1(E)]^2, \qquad
\PP0 \colon \Gr(\VV_h^E) \to [\Pk_{0}(E)]^{2 \times 2}.
\]
In fact, given any $\vv_h \in \VV_h^E$, we are able to determine the polynomials $\PN \vv_h$ and $\PP0\nabla\vv_h$ using solely the information within the DoFs $\boldsymbol{D_X} \vv_h$.
Furthermore, using a different definition of the virtual space $\VV_h^E$ (sharing the same DoFs), it is possible to compute exactly the $L^2$-projection
$
\P0 \colon \VV_h^E \to [\Pk_1(E)]^2
$
from the DoFs $\boldsymbol{D_X}$.
We avoid this technicality and refer to \cite{Ahmed-et-al:2013,NavierStokes:divfree,Vacca:2018}
for more details.
\end{remark}

In the standard procedure of VEM framework,
we introduce a computable discrete local bilinear form
\begin{equation}
\label{eq:a_h^E} 
a_h^E(\cdot, \cdot) \colon \VV_h^E \times \VV_h^E \to \R
\end{equation}
approximating the continuous form $a^E(\cdot, \cdot)$ by setting
\begin{equation}
\label{eq:a_h^E def}
a_h^E(\boldsymbol{u}, \boldsymbol{v}) := a^E \left(\PN \boldsymbol{u}, \, \PN \boldsymbol{v} \right) + \mathcal{S}^E \left((I - \PN) \boldsymbol{u}, \, (I -\PN) \boldsymbol{v} \right)
\end{equation}
for all $\boldsymbol{u}, \boldsymbol{v} \in \VV_h^E$, where the (symmetric) stabilizing bilinear form $\mathcal{S}^E \colon \VV_h^E \times \VV_h^E \to \R$, satisfies 
\begin{equation}
\label{eq:S^E}
\alpha_* a^E(\boldsymbol{v}, \boldsymbol{v}) \le  \mathcal{S}^E(\boldsymbol{v}, \boldsymbol{v}) \le \alpha^* a^E(\boldsymbol{v}, \boldsymbol{v}) \qquad \text{for all $\boldsymbol{v} \in \VV_h$ s.t. ${\Pi}_{1}^{\Gr^s ,E} \boldsymbol{v}= \boldsymbol{0}$} \,,
\end{equation}
with $\alpha_*$ and $\alpha^*$  positive  constants independent of the element $E$.
It is straightforward to check that Definition~\eqref{eq:Ps_k^E} and property~\eqref{eq:S^E} imply 
\begin{enumerate}[$\bullet$]
\item $\boldsymbol{k}$\textbf{-consistency}: for all $\boldsymbol{q}_1 \in [\Pk_1(E)]^2$ and $\boldsymbol{v} \in \VV_h^E$
\begin{equation}\label{eq:consist}
a_h^E(\boldsymbol{q}_1, \boldsymbol{v}) = a^E( \boldsymbol{q}_1, \boldsymbol{v});
\end{equation}
\item \textbf{stability}:  there exist  two positive constants $\alpha_*$ and $\alpha^*$, independent of $h$ and $E$, such that, for all $\boldsymbol{v} \in \VV_h^E$, it holds
\begin{equation}\label{eq:stabk}
\alpha_* a^E(\boldsymbol{v}, \boldsymbol{v}) \le a_h^E(\boldsymbol{v}, \boldsymbol{v}) \le \alpha^* a^E(\boldsymbol{v}, \boldsymbol{v}).
\end{equation}
\end{enumerate}
Under suitable mesh assumptions \cite{BdV-Lovadina-Russo,Brenner-Sung}, two admissible choices for $\mathcal{S}^E$ that guarantee \eqref{eq:S^E} will be given below in \eqref{eq:dofidofi} and \eqref{eq:trace}.

The global approximate bilinear form $a_h(\cdot, \cdot) \colon \VV_h \times \VV_h \to \R$ is obtained by simply summing the local contributions:
\begin{equation}
\label{eq:a_h}
a_h(\boldsymbol{u}_h, \boldsymbol{v}_h) := \sum_{E \in \Ph}  a_h^E(\boldsymbol{u}_h, \boldsymbol{v}_h) \qquad \text{for all $\boldsymbol{u}_h, \boldsymbol{v}_h \in \VV_h$.}
\end{equation}
For what concerns the approximation of the local trilinear form $c^E(\cdot; \, \cdot, \cdot)$, we set
\begin{equation*}
c_h^E(\ww_h; \uu_h, \vv_h) := \int_E [ (\PP0 \, \Gr \uu_h )  (\P0 \ww_h ) ] \cdot \P0 \vv_h  {\rm d}E 
\quad \text{for all $\ww_h, \uu_h, \vv_h \in \VV_h$}
\end{equation*}
and note that all quantities in the previous formula are computable, in the sense of Remark \ref{prp:projections}. 
As usual we define the global approximate trilinear form by adding the local contributions:
\begin{equation}
\label{eq:c_h}
c_h(\ww_h; \, \uu_h, \vv_h) := \sum_{E \in \Ph}  c_h^E(\ww_h; \, \uu_h, \vv_h), \qquad \text{for all $\ww_h, \uu_h, \vv_h \in \VV_h$.}
\end{equation}
We notice that the form $c_h(\cdot; \, \cdot, \cdot)$ is immediately extendable to the whole $\VV$. 
Moreover we recall from \cite{NavierStokes:divfree} that $c_h(\cdot; \, \cdot, \cdot)$ 
is continuous on $\VV$, uniformly in $h$, i.e., 
there exists a positive constant
$\widehat{C}$, independent of $h$, such that
\begin{equation*}
|c_h(\ww; \, \uu, \vv)| \le
\widehat{C} \, \|\ww\|_{\VV} \|\uu\|_{\VV} \|\vv\|_{\VV},
\qquad \text{for all } \ww, \uu, \vv \in \VV.
\end{equation*}

\subsection*{Linear forms {and boundary data}}
%
The last step consists in constructing computable approximations of the right-hand side {$\boldsymbol{f}$ and boundary data $\boldsymbol{g},\boldsymbol{h}$ in \eqref{eq:fsi_variazionale}.}  We define the approximate load term $\boldsymbol{f}_h$ as \looseness=-1
\begin{equation}
\label{eq:f_h}
\boldsymbol{f}_h := \Pi_{1}^{0,E} \boldsymbol{f} \qquad \text{for all $E \in \Ph$,}
\end{equation}
and consider:
\begin{equation}
\label{eq:right}
(\boldsymbol{f}_h, \boldsymbol{v}_h)_{0,\Omega}  = 
\sum_{E \in \Ph} \int_E \Pi_{1}^{0,E} \boldsymbol{f} \cdot \boldsymbol{v}_h \, {\rm d}E =
 \sum_{E \in \Ph} \int_E \boldsymbol{f} \cdot \Pi_{1}^{0,E}  \boldsymbol{v}_h \, {\rm d}E.
\end{equation}
We observe that \eqref{eq:right} can be computed from $\boldsymbol{D_X}$ for all $\boldsymbol{v}_h \in \VV_h$ (see again Remark \ref{prp:projections}), once a suitable quadrature rule is available for polygonal domains. Details on such an issue can be found for instance in \cite{vianello1,Mousavi-Sukumar,Chin-Lasserre-Sukumar}. 

{If $\boldsymbol{g} \in [C(\overline{\partial\D_D})]^2$, let $\boldsymbol{g}_h$ be the DoFs interpolant on $\partial\D_D$ of $\boldsymbol{g}$, i.e., let
$\boldsymbol{g}_h \cdot \ttt$ be the continuous  piecewise linear approximation of  $\boldsymbol{g} \cdot \ttt$ and $\boldsymbol{g}_h \cdot \nn$ be the continuous  piecewise quadratic approximation of  $\boldsymbol{g} \cdot \nn$. Let $\boldsymbol{h}_h$ be a piecewise polynomial interpolant of $\boldsymbol{h}\in [C(\overline{\partial\D_N})]^2$ that accounts for the effect of quadrature on $\partial\D_N$.} 

\subsection{Virtual Elements for the coupled problem}
\label{sub:3.2}
%
The aim of the present Section is to describe the Virtual Element discretization of Problem \eqref{eq:fsi_variazionale}.
{Here, as later in Section \ref{sec:tests}, we} assume that $\varepsilon=0$, so that the leaflet can be represented by a segment. This is a good approximation for small values of $\varepsilon$, and allows to use a simple and effective mesh-cutting technique in the numerical tests of {Section \ref{sec:tests}.}  

Let $\{\Ph\}_h$ be a sequence of decompositions of the channel $\D$ into general polygonal elements $E$ given independently of the position of the structure $\Gamma$ and  satisfying the mesh assumptions  $\boldsymbol{(A1)}$  and $\boldsymbol{(A2)}$.
From a mesh $\Ph$ in $\D$, we define the mesh 
$\PhG$  in $\Omega$ by cutting with $\Gamma$ the elements of $\Ph$.
In order to have a clear overview of the situation, let us consider the simplest case when $\Ph$ is a square decomposition of $\D$ (see Fig. \ref{fig:meshOmegahG}).

\begin{figure}[!h]
\center
{
\begin{overpic}[scale=0.22]{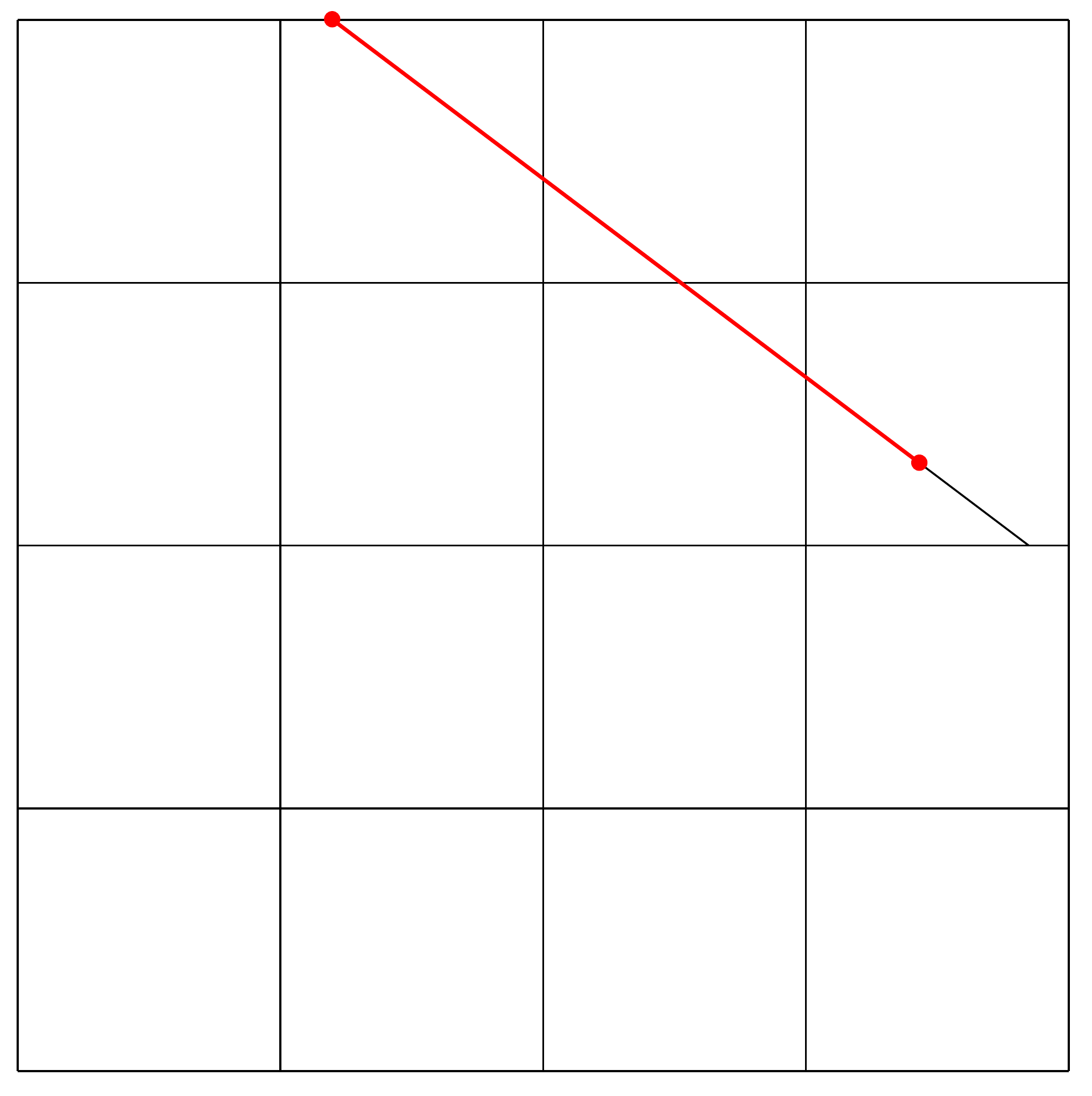} 
\put (28,99) {\large{$\ov$}}
\put (56,60) {\huge{$\Gamma$}}
\put (28,85) {\large{$E_1$}}
\put (40,91) {\large{$E_2$}}
\put (72,51) {\large{$E_{N-1}$}}
\put (85,68) {\large{$E_{N}$}}
\put (28,28) {\huge{$\Omega$}}
\put (85,59) {\large{$T$}}
\end{overpic}
\caption{Example of mesh $\PhG$ obtained from a square decomposition $\Ph$ of the channel by cutting  the elements across the straight segment $\Gamma$.}
\label{fig:meshOmegahG}
}
\end{figure}
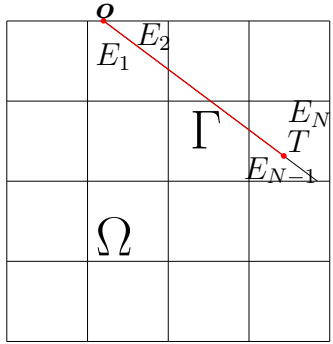

Depending on the position of the cut, from a single square we may generate two sub-polygons, possibly violating assumptions $\boldsymbol{(A1)}$  and $\boldsymbol{(A2)}$.
We observe that also in this simple situation, starting from a square decomposition, we need to handle a general polygonal mesh containing for instance pentagons, therefore the virtual element approach turns out to be particularly appropriate in this context.
We stress that in the presence of an internal cut, i.e.,  if the tip $T$ of $\Gamma$ does not belong to an edge of the underlying decomposition $\Ph$, we extend the segment $\Gamma$ until we obtain the full cut of the element containing $T$. In that case, with reference to Fig. \ref{fig:meshOmegahG}, the elements $E_{N-1}$ and $E_{N}$ have to be considered as a quadrilateral and a hexagon, respectively, since the prolongation of $\Gamma$ is considered as a separate edge.

\begin{remark}[{cracked polygon}]
\label{rm:cutting}
An alternative choice to treat the case of an internal cut is {to consider the polygon containing the tip of $\Gamma$  as a ``cracked'' element rather than prolonging the leaflet.} Indeed the virtual element technology can handle also this type of polygons. A preliminary numerical investigation of the schemes obtained with the ``leaflet prolongation'' and the ``cracked polygon'' strategies {revealed} that the former approach appears more robust in terms of behaviour of the discrete torque functional $\tau_h(\theta)$  defined in  \eqref{eq:jh}. Therefore in the following we focus only on the first strategy.
\end{remark}

{We now discuss the discretization of problem \eqref{eq:fsi_variazionale}. In view of \eqref{eq:V_h} and \eqref{eq:Q_h}, the discrete spaces subordinate to the partition $\PhG$ are} \looseness=-1
\begin{gather*}
\VV_h := \{ \boldsymbol{v} \in [H^1(\Omega)]^2  ~ :  \quad 
\boldsymbol{v}_{|E} \in \VV_h^E  \quad \text{for all $E \in \PhG$} \},
\\
Q_h := \{ q \in L^2(\Omega) ~ : \quad q_{|E} \in  Q_h^E \quad \text{for all $E \in \PhG$}\},
\end{gather*}
{and the virtual discretization of the affine manifold in \eqref{eq:spazi continui} reads}
\begin{equation}
\begin{aligned}
\label{eq:spazi discreti}
&\VVGh^{\boldsymbol{g}} := \left\{ \vv \in \VV_h ~ : 
\quad\vv_{|_{\partial \Omega_D}} = \boldsymbol{g}_h \,,
\quad \vv_{|_{\Gamma}} = \boldsymbol{0}  \right\}  \, ,
\end{aligned}
\end{equation}
{where $\boldsymbol{g}_h$ is the VEM interpolant of $\boldsymbol{g}$.}
In  light of Proposition \ref{prp:dofs} 
and the definitions above, 
the linear operator $\boldsymbol{\widehat{D_X}}$ constitutes a set of DoFs for the virtual {space $\VVGh^{\boldsymbol{0}}$:}
%
%
for any element $E \in \PhG$ we consider 
\begin{itemize}
\item ${\boldsymbol{\widehat{D_X1}}}$:  the values of $\boldsymbol{v}$ at the vertices of the element that do not belong to $\Gamma \cup \partial \Omega_D$,
\item ${\boldsymbol{\widehat{D_X2}}}$: the values of the normal components $\vv \cdot \nn_e$
at the midpoint of each edge of $E$ that is not contained in $\Gamma \cup \partial \Omega_D$.
\end{itemize}
In Fig. \ref{fig:dofs_hat} we display an example of such DoFs for some sample elements adjacent to $\Gamma$. We observe that no DoFs are given on $\Gamma$, since homogeneous Dirichlet conditions are enforced therein. 
\begin{figure}[!h]
\center{
\begin{overpic}[scale=0.22]{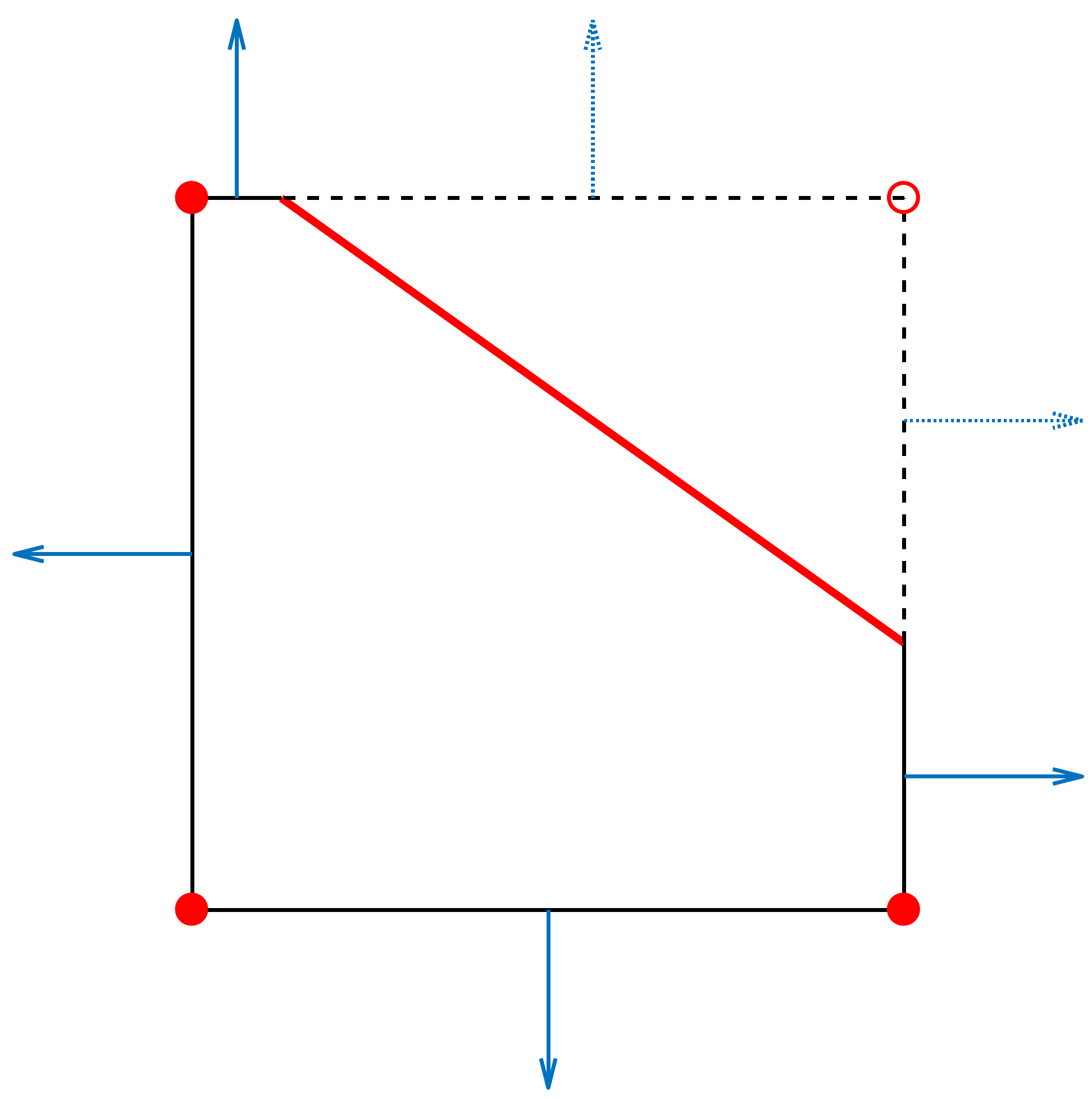}
\put (55,42) {\huge{$\Gamma$}}
\put (30,31) {\huge{$E_{i-1}$}}
\put (63,65) {\huge{$E_i$}}
\end{overpic}
\qquad \qquad 
\begin{overpic}[scale=0.22]{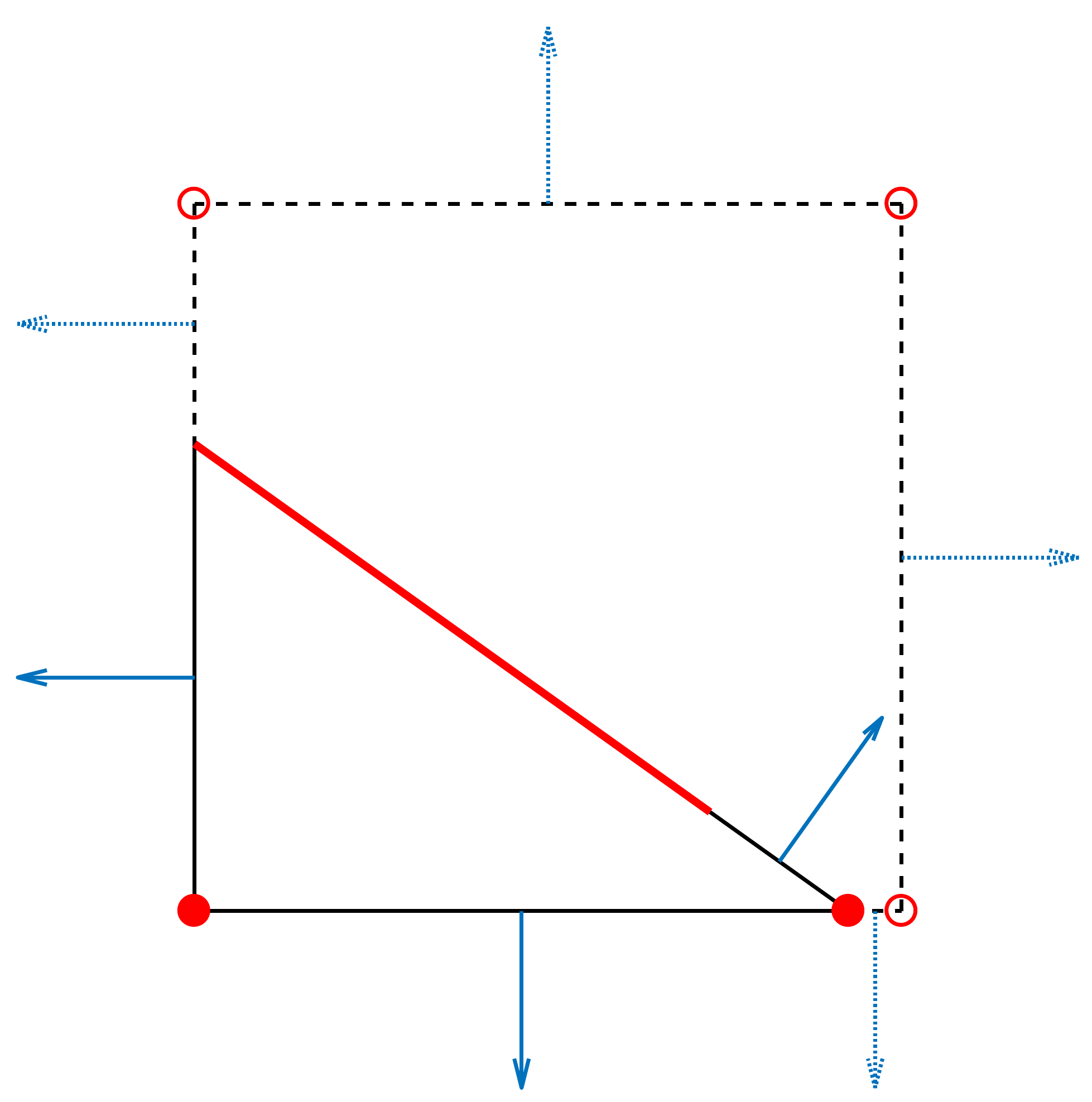} 
\put (65,27) {\large{$T$}}
\put (45,42.5) {\huge{$\Gamma$}}
\put (20,25) {\huge{$E_{N-1}$}}
\put (58,60) {\huge{$E_N$}}
\end{overpic}
\caption{Degrees of freedom: We denote ${\boldsymbol{\widehat{D_X1}}}$ by dots {and} ${\boldsymbol{\widehat{D_X2}}}$ by arrows.
Leaflet $\Gamma$ cutting through an element $E$ (left) and prolongation of $\Gamma$ to the boundary of $E$ (right). The letter $T$ above indicates the position of the tip of the leaflet.
}}
\label{fig:dofs_hat}
\end{figure}
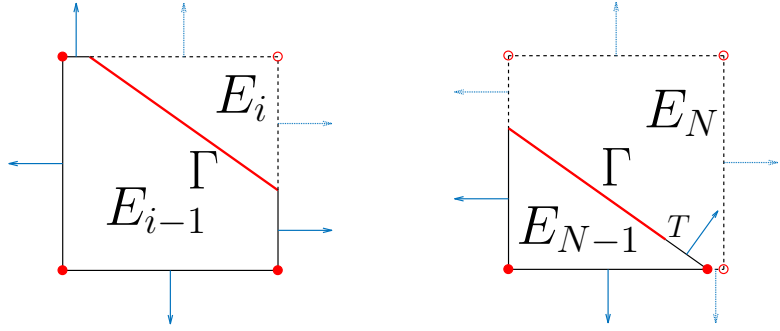

We denote by $\EG^h \in \VV_h$ the function defined by the following DoFs values:
{
\begin{equation}
\label{eq:EGh}
\begin{aligned}
{\EG^h}(\pv)&=
\left \{
\begin{aligned}
& |\pv - \ov| \, \e_{\theta}^\perp & \,\,  & \text{if $\pv \in \Gamma$,}
\\
& \boldsymbol{0}        & \,\,  & \text{otherwise}
\end{aligned}  
\right .
\quad\qquad \text{$\pv$ mesh vertex,} 
\\
{\EG^h}(\pv) \cdot \nn_e &=
\left \{
\begin{aligned}
& |\pv - \ov| \, \e_{\theta}^\perp\cdot \nn_e & \,\,  & \text{if $\pv \in \Gamma$,}
\\
& 0       & \,\,  & \text{otherwise}
\end{aligned}  
\right .
\quad  \text{$\pv$ midpoint of mesh edge $e$.} 
\end{aligned}
\end{equation}
}
Notice  that \eqref{eq:EG} and \eqref{eq:EGh} imply 
\begin{equation}
\label{eq:EGh_b}
\EG^h|_{\Gamma} = \EG|_{\Gamma} = r \,  \e_{\theta}^\perp = \Fi \,.
\end{equation}

We are now ready to state the proposed discrete problem. Referring to~\eqref{eq:spazi discreti}, \eqref{eq:Q_h}, \eqref{eq:EGh}, ~\eqref{eq:a_h},  ~\eqref{eq:c_h} and ~\eqref{bhform}, we consider the \textbf{virtual element problem}: find 
$\theta_h \in I_{\epsilon_0}$ 
and $(\uu_h, \, p_h) \in \VVGh^{\boldsymbol{g}} \times Q_h$, such that
\begin{equation}
\label{eq:fsi_vem}
\!\! \left\{
\begin{aligned}
\! \nu \, a_h(\uu_h, \vv_h \! + \! \sigma \,\EG^h) & +
c_h(\uu_h; \, \uu_h, \vv_h \!+ \! \sigma \, \EG^h) +  
b(\vv_h \! + \! \sigma \, \EG^h, p_h) + \sigma \, \kappa (\theta_h)    & & \\
& = (\ff_h, \vv_h + \sigma \, \EG^h)_{0, \Omega} + 
{(\boldsymbol{h}_h,  \vv_h)_{0, \partial\mathcal{D}_N}} \, , & &
\\
 b(\uu_h, q_h) &= 0 \, , & &
\end{aligned}
\right. 
\end{equation}
{for all $\sigma \in \R$ and $(\vv_h,q_h)\in \VVGh^{\boldsymbol{0}}\times Q_h$.
Finally, in view of \eqref{eq:divfree} and what is observed in \cite{Stokes:divfree,NavierStokes:divfree}, the last line in \eqref{eq:fsi_vem} implies} that the velocity solution $\uu_h$ is pointwise divergence-free.

\subsection{Discrete torque functional}
\label{sub:3.3}

In the present section, in accordance with Section \ref{sect:shape}, we modify \eqref{eq:fsi_vem} to get the discretization of the (linear) Stokes model \eqref{eq:Stokes}:
find 
$\theta_h \in I_{\epsilon_0}$ 
and $(\uu_h, \, p_h) \in \VVGh^{\boldsymbol{g}} \times Q_h$, such that
for all {$\sigma \in \R$ and $(\vv_h,q_h) \in \VVGh^{\boldsymbol{0}} \times Q_h$}
\begin{equation}
\label{eq:fsi_vem-stokes}
\left\{
\begin{aligned}
 \nu \, a_h(\uu_h, \vv_h + \sigma \,\EG^h) &+
b(\vv_h + \sigma \, \EG^h, p_h) +  
\sigma \, \kappa (\theta_h)
 & &
\\
& 
=(\ff_h, \vv_h + \sigma \, \EG^h)_{0, \Omega} + 
{(\boldsymbol{h}_h,  \vv_h)_{0, \partial\mathcal{D}_N}} \, ,  & &\\
 b(\uu_h, q_h) &= 0\, . & &
\end{aligned}
\right. 
\end{equation}

In order to study the solvability of the discrete Problem \eqref{eq:fsi_vem-stokes}, it is convenient to introduce the discrete torque functional $\theta \mapsto \tau_h(\theta)$. {Before doing so, we recall an equivalent expression for the continuous torque functional $\theta\mapsto\tau(\theta)$
\begin{equation}\label{eq:weak-torque2}
\tau(\theta) = -\nu \, a(\uu, \EG) - b(\EG, p) + (\ff, \EG)_{0,\Omega}\, ,
\end{equation}
which hinges on \eqref{eq:torque}, \eqref{eq:balance} and \eqref{eq:weak-torque}.
Similarly, if $(\uu_h, \, p_h) \in \VVGh^{\boldsymbol{g}} \times Q_h$ is the solution of the following discrete Stokes equations for a given $\Gamma=\Gamma(\theta)$
\begin{equation}
\label{eq:vem-stokes-fsi}
\left\{
\begin{aligned}
 \nu \, a_h(\uu_h, \vv_h) +
b(\vv_h, p_h) &= 
(\ff_h, \vv_h)_{0, \Omega} + 
(\boldsymbol{h}_h,  \vv_h)_{0, \partial\mathcal{D}_N}\,, \\
 b(\uu_h, q_h) &= 0 \, , & &
\end{aligned}
\right. 
\end{equation}
for all $(\vv_h,q_h) \in \VVGh^{\boldsymbol{0}} \times  Q_h$, we define the discrete torque $\theta \mapsto \tau_h(\theta)$ to be
\begin{equation}\label{eq:jh}
\tau_h(\theta) := 
-\nu \, a_h(\uu_h, \, \EG^h) - 
b(\EG^h, \, p_h)
+ (\ff_h, \, \EG^h)_{0, \Omega} \,,
\end{equation}
Taking $\vv_h=0, \sigma=1$ in \eqref{eq:fsi_vem-stokes}, we obtain the leaflet momentum balance
\begin{equation}\label{eq:h-cond}
\kappa (\theta_h) = \tau_h(\theta_h) \, ,
\end{equation}
which is the discrete analogue of \eqref{eq:leaflet_primale}.
In contrast to $\tau$, we will see in Sect. \ref{sec:tests} that the functional $\tau_h$ need not be continuous with respect to $\theta$, because a small change in the position of $\Gamma$ may induce a jump in the number of degrees of freedom that affect the definition of the stabilization form $\mathcal{S}^E$. However, if jumps in $\tau_h$ exist, they should tend to $0$ as the mesh parameter $h$ tends to $0$. This will be elucidated next.

\smallskip
We next quantify the torque error $\tau(\theta)-\tau_h(\theta)$ for any fixed $\theta$. To this end, we need the approximation errors $\mathcal{E}(\vv):=\mathcal{I}(\vv) + \mathcal{P}(\vv)$ for any $\vv \in \VV$, $\mathcal{F}(\ff)$ for $\ff \in [L^2(\Omega)]^2$,  $\mathcal{G}(\boldsymbol{g})$ for $\boldsymbol{g} \in [H^{\frac32}(\partial\D_\text{in})]^2$, and $\mathcal{H}(\boldsymbol{h})$ for $\boldsymbol{h}\in [H^1(\partial\D_N)]^2$, where
\[
\mathcal{I}(\vv) := \min_{\vv_h \in \VV_h} \|\vv - \vv_h\|_{\VV} \,,
\quad
 {\mathcal{P}(\vv) := \min_{\vv_{\pi} \in [\Pk_{1,h}]^2} \left(\sum_{E \in \PhG} \|\vv - \vv_{\pi}\|^2_{\VV,E} \right)^{1/2}},
\]
$[\Pk_{1,h}]^2$ denotes the space of piecewise polynomials of degree one over $\PhG$ and
\[
\mathcal{F}(\ff) := \|\ff - \ff_h \|_{\VV^*} \, ,
\quad
\mathcal{G}(\boldsymbol{g}) := \|\boldsymbol{g} - \boldsymbol{g}_h \|_{[H^{1/2}(\partial\D_D)]^2} \, ,
\quad    
\mathcal{H}(\boldsymbol{h}) := \|\boldsymbol{h} - \boldsymbol{h}_h\|_{[L^2(\partial\D_N)]^2} \, .
\]
The following energy error estimate, that takes into account also the influence of the boundary data approximation, is a trivial extension of well known results in the literature \cite{Antonietti-BeiraodaVeiga-Mora-Verani:2014,Stokes:divfree,NavierStokes:divfree}
\begin{equation}\label{eq:error-vem}
\|\uu - \uu_h\|_{[H^1(\Omega)]^2} \lesssim \mathcal{E}(\uu) + \mathcal{F}(\ff) + \mathcal{G}(\boldsymbol{g}) + \mathcal{H}(\boldsymbol{h}).
\end{equation}

\begin{proposition}[approximation of torque]\label{prp:intermedia}
For a given $\theta \in I_{\epsilon_0}$, let $\Gamma=\Gamma(\theta)$ and $(\uu,p)$ and $(\zz,q)$ be the solutions of the Stokes problem \eqref{eq:Stokes} and adjoint problem \eqref{eq:adjoint-pb}.
Let $\tau(\theta)$ and $\tau_h(\theta)$ be the continuous and discrete torque functionals satisfying \eqref{eq:weak-torque2} and \eqref{eq:jh}, respectively. Then the following error estimate holds
\begin{equation}
\label{eq:stima}
\big|\tau(\theta) - \tau_h(\theta) \big| \lesssim
\mathcal{E}(\uu) \mathcal{E}(\zz) +  \big(\mathcal{F}(\ff) + \mathcal{G}(\boldsymbol{g}) +\mathcal{H}(\boldsymbol{h}) \big) \, \big(\mathcal{E}(\zz) + \|\zz\|_{\VV} \big)  \, .
\end{equation}
\end{proposition}
\begin{proof}
We start with a simple but crucial observation: the function $\EG$ in \eqref{eq:weak-torque2} can be replaced by any function $\vv\in\VV$ with the same Dirichlet boundary condition as $\EG$ on $\Gamma\cup\partial\D_D$ because $\EG-\vv\in\VV_\Gamma^{\boldsymbol{0}}$ is an admissible test function for the weak Stokes equation for $(\uu,p)$; one only needs to add the Neumann boundary term $(\boldsymbol{h}, \, \cdot)_{0, \partial\mathcal{D}_N}$ that in \eqref{eq:weak-torque2} is missing since $\EG$ vanishes on $\partial\mathcal{D}_N$.
The same comment applies to \eqref{eq:jh}. To choose $\vv$, we recall that $(\zz,q)\in\VV \times Q$ solves the adjoint problem \eqref{eq:adjoint-pb}, whose weak form reads
\begin{equation}\label{eq:adjoint-weak}
\left\{
\begin{aligned}
    \nu a(\vv,\zz) + b(\vv,q) &= 0 \qquad\textrm{for all } \vv\in\VV_\Gamma^{\boldsymbol{0}} \, ,
    \\
    b(\zz,s) &= 0 \qquad \textrm{for all } s\in Q.
\end{aligned}
\right. 
\end{equation}
Let $(\zz_h,q_h)\in\VV_h\times Q_h$ be the corresponding VEM counterpart of Section \ref{sub:3.1}, and note that $\zz_h = \EG = \Phi$ on $\Gamma$, $\zz_h=\EG=\boldsymbol{0}$ on $\partial\D_D$ and $\dd \, \zz_h=0$ in $\Omega$. We thus choose $\vv=\zz_h\in\VV_h\subset\VV$ to write
\begin{equation}
\label{eq:jh-1}
\begin{split}
\tau_h(\theta) &= -\nu \, a_h(\uu_h, \, \zz_h) - 
b(\zz_h, \, p_h)
+ (\ff_h, \, \zz_h)_{0, \Omega} + 
(\boldsymbol{h}_h, \, \zz_h)_{0, \partial\mathcal{D}_N},
\\
\tau(\theta) &= -\nu \, a(\uu, \, \zz_h) - 
b(\zz_h, \, p)
+ (\ff, \, \zz_h)_{0, \Omega} + 
(\boldsymbol{h}, \, \zz_h)_{0, \partial\mathcal{D}_N} \, .
\end{split}
\end{equation}
Therefore, we obtain the error decomposition $\tau(\theta)-\tau_h(\theta) = I + II + III$ with
\begin{align*}
  I & := -\nu a(\uu-\uu_h,\zz_h) - b(\zz_h,p-p_h) \, ,
  \\
  II & := - \nu \big[ a(\uu_h,\zz_h) - a_h(\uu_h, \zz_h)  \big] \, ,
  \\
  III & :=  \big[(\ff,\zz_h)_{0,\Omega} - (\ff_h,\zz_h)_{0,\Omega}\big]
  + \big[ (\boldsymbol{h}, \, \zz_h)_{0, \partial\mathcal{D}_N} - (\boldsymbol{h}_h, \, \zz_h)_{0, \partial\mathcal{D}_N} \big] \, .
\end{align*}
The rest of the proof consists of estimating these three terms separately.

\smallskip\noindent
1. {\it Estimate of $I$:} We utilize that $b(\zz_h,p-p_h)=0$, because $\dd \zz_h=0$, to deduce
$$
I = -\nu a(\uu-\uu_h,\zz_h).
$$
In view of \eqref{eq:adjoint-weak}, we would like to exploit the fact that $a(\vv,\zz)=0$ for $\vv\in\VV_\Gamma^{\boldsymbol{0}}$ and $\dd \, \vv = 0$, but we cannot take $\vv=\uu-\uu_h$ because this $\vv\ne\boldsymbol{0}$ on $\partial\D_D$. Let $(\ww,s)\in\VV\times Q$ be the solution of the Stokes problem \eqref{eq:Stokes} with data $\ff = \boldsymbol{h} = \boldsymbol{0}$ and Dirichlet condition $\ww=\boldsymbol{g} - \boldsymbol{g}_h$ on $\partial\D_D$ and $\ww=\boldsymbol{0}$ on $\Gamma$. Therefore, we have
\[
\|\ww\|_{[H^1(\Omega)]^2} \lesssim \|\boldsymbol{g} - \boldsymbol{g}_h\|_{[H^{1/2}(\partial\Omega_D)]^2} = \mathcal{G}(\boldsymbol{g}),
\]
according to \eqref{eq:energy-bound}. Since $\vv = \uu-\uu_h - \ww \in \VV_\Gamma^{\boldsymbol{0}}$ and $\dd\,\vv=0$, \eqref{eq:adjoint-weak} implies
\[
\nu a(\uu-\uu_h,\zz) - \nu a(\ww,\zz) = 0,
\]
which added to $I$ yields
$
I = \nu a(\uu-\uu_h,\zz-\zz_h) - \nu a(\ww,\zz).
$
We observe that \eqref{eq:error-vem} is valid for both $\uu$ and $\zz$, the latter without data approximation because $\Phi\in\VV_h$. Consequently, the Cauchy-Schwarz inequality gives
\begin{equation*}\label{eq:I}
  | I | \lesssim \mathcal{E}(\uu) \, \mathcal{E}(\zz)
  + \big( \mathcal{F}(\ff) + \mathcal{G}(\boldsymbol{g}) +\mathcal{H}(\boldsymbol{h}) \big) \big(\mathcal{E}(\zz) + \|\zz\|_{\VV} \big).
\end{equation*}

\noindent
2. {\it Estimate of $II$:} 
We decompose $II$ elementwise and write for any $E\in\PhG$
\begin{align*}
  \nu^{-1} II^E :&= a^E(\uu_h,\zz_h) - a^E_h(\uu_h,\zz_h)
  \\
  & = a^E(\uu_h,\zz_h) - a^E(\Pi_1^E \uu_h, \Pi_1^E \zz_h)
  - S^E((I-\Pi_1^E) \uu_h, (I-\Pi_1^E)\zz_h)
  \\
  & = a^E(\uu_h - \Pi_1^E \uu_h,\zz_h - \Pi_1^E \zz_h) 
  - S^E((I-\Pi_1^E) \uu_h, (I-\Pi_1^E)\zz_h)
\end{align*}
where we have used the $a-$orthogonality property of the projector operator $\Pi_1^E := \Pi_1^{\nabla^s,E}$ defined in \eqref{eq:Ps_k^E}.
In view of \eqref{eq:S^E}, we infer that
\begin{align*}
| II^E | & \lesssim \|\nabla^s(\uu_h - \Pi_1^E \uu_h)\|_{L^2(E)} \, \|\nabla^s(\zz_h - \Pi_1^E\zz_h)  \|_{L^2(E)}
\\
& \le \|\nabla^s(\uu_h - \Pi_1^E \uu)\|_{L^2(E)} \, \|\nabla^s(\zz_h - \Pi_1^E\zz)\|_{L^2(E)} ,
\end{align*}
because $\Pi_1^E \uu_h$ is a projection with respect to the operator $\nabla^s$.
Adding and subtracting $\uu$, $\zz$, and combining the Cauchy-Schwarz inequality with \eqref{eq:error-vem} yields
\[
| II | \le \sum_{E\in\PhG} | II^E | \lesssim
\big(\mathcal{E}(\uu) + \mathcal{F}(\ff) + \mathcal{G}(\boldsymbol{g}) +\mathcal{H}(\boldsymbol{h}) \big) \mathcal{E}(\zz).
\]

%
%

\noindent
3. {\it Estimate of $III$:} We split $III = III_1 + III_2$ and recall \eqref{eq:right} to find out
\begin{gather*}
| III_1 | \le \sum_{E\in\PhG} \int_E \big| (\ff - \Pi_1^{0,E} \ff) \cdot \zz_h \big|
\le \|\ff - \ff_h\|_{\VV^*} \|\zz_h\|_{\VV} \lesssim \mathcal{F}(\ff) \|\zz\|_{\VV},
\\
| III_2 | \le \|\boldsymbol{h} - \boldsymbol{h}_h\|_{[L^2(\partial\D_N)]^2} \|\zz_h\|_{\VV}
\lesssim \mathcal{H}(\boldsymbol{h}) \|\zz\|_{\VV},
\end{gather*}
because $\|\zz_h\|_{\VV} \lesssim \|\zz\|_{\VV}$. This concludes the proof.
\end{proof}

\begin{lemma}[interpolation]
\label{lem:interpolation}
Let the mesh $\PhG$ satisfy $\boldsymbol{(A1)}$ and $\boldsymbol{(A2)}$.
Let $\vv \in \VV$ be so that $\vv|_E \in \left[H^{s+1}(E)\right]^2$  for all $E\in\PhG$  with $0 < s \le 1$.  Then 
\begin{equation*}
\mathcal{E}(\vv) \lesssim  \, \sum_{E \in \PhG} h_E^s \, |\vv|_{H^{s+1}(E)}
\end{equation*}
\end{lemma}
\begin{proof}
Use the interpolation estimate of Theorem 4.1 in Ref. {NavierStokes:divfree} for $\mathcal{I}(\vv)$, and standard polynomial approximation estimates \cite{brennerscott} for $\mathcal{P}(\vv)$.
\end{proof}

The previous outcomes  can be summarized in the following convergence result.
\begin{theorem}[torque error estimates] \label{thm:finale}
Let the mesh $\PhG$ satisfy $\boldsymbol{(A1)}$ and $\boldsymbol{(A2)}$ and let $h=\max_{E\in \PhG} h_E$. If $\ff\in [L^2(\Omega)]^2$, $\boldsymbol{g} \in [H^{3/2}(\partial\D_\text{in})]^2$, and $\boldsymbol{h}\in [H^1(\partial\D_N)]^2$, then the continuous and discrete torque functionals $\tau(\theta)$ and $\tau_h(\theta)$, given in \eqref{eq:weak-torque2} and \eqref{eq:jh}, satisfy the error estimate
\begin{equation}
\label{eq:estimate}
\big| \tau(\theta) - \tau_h(\theta) \big| \lesssim h |\log h|.
\end{equation}
\end{theorem}
\begin{proof}
The regularity of $\uu$ and $\zz-\Phi$ in $\Omega$ is dictated by the singularity at the tip of the leaflet. Such singularity, already used in Proposition \ref{P:continuity}, is of the form $r^{1/2}$ in polar coordinates centered at the tip provided $\ff\in[L^2(\Omega)]^2$; note that $\Phi=r \e_\omega^\perp$ is smooth. Therefore, a fractional derivative of order $1+s$ is square integrable
\[
\int_0^1 r^{2(\frac12-1-s)+1} dr = \int_0^1 r^{-2s} dr = \frac{1}{1-2s} < \infty
\]
provided $s<\frac12$. Take now $s=\frac12-\delta$, for $\delta>0$ sufficiently small, to obtain
\[
|\uu|_{H^{\frac32-\delta}(\Omega)} \lesssim \delta^{-\frac12},
\qquad
|\zz|_{H^{\frac32-\delta}(\Omega)} \lesssim \delta^{-\frac12}.
\]
Combining these estimates with Lemma \ref{lem:interpolation}, and choosing $\delta=|\log h|^{-1}$, yields
\[
\mathcal{E}(\uu), \, \mathcal{E}(\zz) \lesssim h^{\frac12} h^{-\delta}\delta^{-\frac12} \approx \big( h \, |\log h|\big)^{\frac12}.
\]
On the other hand, if $\ff\in [L^2(\Omega)]^2$, $\boldsymbol{g} \in [H^{3/2}(\partial\D_\text{in})]^2$, and $\boldsymbol{h}\in [H^1(\partial\D_N)]^2$, then
\[
\mathcal{F}(\ff)\lesssim h \|\ff\|_{[L^2(\Omega)]^2},
\quad
\mathcal{G}(\boldsymbol{g}) \lesssim h \|\boldsymbol{g}\|_{[H^{3/2}(\partial\D_\text{in})]^2},
\quad
\mathcal{H}(\boldsymbol{h}) \lesssim h \|\boldsymbol{h}\|_{[H^1(\partial\D_N)]^2} \, .
\]
The asserted estimate \eqref{eq:estimate} follows from \eqref{eq:stima} of Proposition \ref{prp:intermedia}.
\end{proof}

Whenever $\boldsymbol{h}_h$ is constructed so that its average on each edge  agrees with that of $\boldsymbol{h}$, the regularity requirement on $\boldsymbol{h}$ above can be easily relaxed to $\boldsymbol{h}\in [H^{1/2+s}(\partial\D_N)]^2$, $s>0$, without changing the convergence rate (where the positive $s$ is included only to guarantee the applicability of Gauss-like integration rules, which require pointwise evaluation).

\begin{remark}[optimality of \eqref{eq:estimate}]
We stress that the rate in \eqref{eq:estimate} is twice that associated with $\mathcal{E}(\uu)$. This is due to the use of the variational expressions \eqref{eq:weak-torque2} and \eqref{eq:jh}, which avoid evaluating the trace of $\T$ and $\boldsymbol{T}(\uu_h,p_h)$ on $\Gamma$ and allow for additional cancellation. The numerical experiments of Test 2 in Section \ref{sec:tests} confirm that the linear rate \eqref{eq:estimate} is optimal (up to the logarithm).
\end{remark}

\begin{remark}[discontinuous $\tau_h$]
  The discrete torque $\tau_h(\theta)$ might be discontinuous according to our experiments in Section \ref{sec:tests} for the {\tt dofi-dofi} stabilization. Since \eqref{eq:estimate} is uniform for $\theta\in I_{\epsilon_0}$ and $\tau(\theta)$ is uniformly continuous for $\theta \in I_{\epsilon_0}$, in light of Proposition \ref{P:continuity}, the triangle inequality implies that as $\widetilde\theta\to\theta$
$$  
\big|\tau_h(\widetilde\theta)-\tau_h(\theta)\big| \lesssim h\,|\log h|  + o(1) \qquad \forall \theta \in I_{\epsilon_0}\,. 
$$
We conclude that any possible jumps of $\tau_h$ must be of order $O(h|\log h|)$.
Consequently, we must accept that the discrete balance equation \eqref{eq:h-cond} be satisfied up to an $O(h|\log h|)$-error. This leads to the following solution algorithm.
\end{remark}

\subsection*{Bisection algorithm for the nonlinear system}
\label{sub:bisection}

Under the assumption of Proposition \ref{rm:limit}, the function $\psi(\theta):=\kappa( \theta) - \tau(\theta)$  satisfies
$$
\lim_{\theta \rightarrow \pm(\tfrac\pi2-\epsilon_0)} \psi(\theta) = \pm \infty \; .
$$
Since the error estimate \eqref{eq:estimate} is uniform in $\theta \in I_{\epsilon_0}$, we deduce that the function $\psi_h(\theta):=\kappa( \theta) - \tau_h(\theta)$ changes sign in $I_{\epsilon_0}$. We thus apply the bisection algorithm is a slightly smaller interval $I_{\delta_0}$, with $\delta_0>\epsilon_0$, and generate a sequence $\{\theta_h^n\}_{n \geq 0}$. The sequence converges to a limit value $\theta_h^\star$, which is either the exact solution of \eqref{eq:h-cond} (if $\tau_h$ is continuous in $\theta_h^\star$), or satisfies
$$
\big|\kappa(\theta_h^\star) - \tau_h(\theta_h^\star)\big| = O(h |\log h|)  \qquad \text{as } \ h \to 0\,.
$$
Combining the above bound with \eqref{eq:estimate}, the triangle inequality yields
$$
\big|\kappa(\theta_h^\star) - \tau(\theta_h^\star)\big| = O(h |\log h|)  \qquad \text{as } \ h \to 0 \, ,
$$
that represents the asymptotic satisfaction of the equilibrium condition \eqref{eq:leaflet_primale}
.
\begin{remark}[error estimate for $\theta$]\label{rem:eet}
If $\psi(\theta^\star)=\kappa(\theta^\star) - \tau(\theta^\star)=0$ dictates the exact equilibrium angle $\theta^\star$ and the {\it non-degeneracy} condition $\psi'(\theta) \ge \lambda>0$ is valid for all $\theta$ in the vicinity of $\theta^*$, then
$$
| \theta^\star - \theta_h^\star| \le \frac{h |\log h|}{\lambda}  \qquad \text{as } \ h \to 0 \, .
$$
\end{remark} 
}

\section{Numerical Tests}
\label{sec:tests}

\subsection{Stabilization}
We briefly sketch the construction of the two choices of stabilizing bilinear forms  $S^E(\cdot, \cdot)$  in \eqref{eq:S^E} used in the numerical tests. We recall that condition \eqref{eq:S^E}  essentially requires that the stabilizing term $\mathcal{S}^E(\boldsymbol{v}_h, \boldsymbol{v}_h)$ scales as $a^E(\boldsymbol{v}_h, \boldsymbol{v}_h)$.
The first option for the stabilization is the so-called \texttt{dofi-dofi}.
Let us denote with $\vec{\boldsymbol{u}}_h$, $\vec{\boldsymbol{v}}_h \in \R^{N_{\text{DoFs}, E}}$
the vectors containing the values of the $N_{\text{DoFs}, E}$  local degrees of freedom associated to $\uu_h, \vv_h \in  \VV_h^E$. Then, we set
\begin{equation}
\label{eq:dofidofi}
\mathcal{S}^E_{\texttt{dofi}} (\uu_h, \vv_h) =  \nu \: \vec{\uu}_h \cdot  \vec{\vv}_h \,.
\end{equation}
The second stabilization adopted in the numerical tests is the \texttt{trace} stabilization introduced in \cite{wriggers}
\begin{equation}
\label{eq:trace}
\mathcal{S}^E_{\texttt{trace}} (\uu_h, \vv_h) =  
\nu h_E \,  \int_{\partial E} \partial_s \uu_h \cdot \partial_s \vv_h \, {\rm d}s  \,.
\end{equation}
Using standard scaling arguments  we notice that the above stabilizations yield the correct scaling for $\mathcal{S}^E (\cdot, \cdot)$ in accordance with \eqref{eq:S^E}, at least for mesh satisfying assumptions $\boldsymbol{(A1)}$ and $\boldsymbol{(A2)}$. An analysis under more general mesh assumptions can be found in ~\cite{BdV-Lovadina-Russo,Brenner-Sung}. Finally, note that we multiply both forms by $\nu$, which is a standard choice for this simple material law, in order to have a correct scaling also with respect to the material parameters.

\subsection{Problem setting and adopted meshes.}
In the proposed tests we consider the fluid-structure interaction problem \eqref{eq:fsi_primale} posed on the square domain  $\D = [0, 1]^2$ with vanishing external load $\boldsymbol{f} = \boldsymbol{0}$ and fluid viscosity $\nu = 1$. We refer again to Fig. \ref{F:leaflet-model} for a depiction of the general problem geometry. We take the following boundary conditions: free boundary conditions $\boldsymbol{h}=\boldsymbol{0}$ at the right outflow boundary edge $ \partial\D_N :=\{1\} \times [0, 1]$,
 Dirichlet boundary conditions at the left inflow edge $\{0\} \times [0, 1]$, given by
\[
\uu(0, y) = (\phi(y),0) \qquad \text{with} \ \  \phi(y)= 0.1 \, y(1 - y) \,.
\]
At the top and bottom wall $[0, 1] \times \{0, 1\}$ of the domain, no-slip boundary conditions are applied (i.e. $\uu = \boldsymbol{0}$).
{The hinged point is in position $O = (0.5, 0)$ and the leaflet has a length of 0.5 with the spring relaxed position $\theta_0 = 0$ being set as the vertical direction (that is, when the tip is in position (0.5,0.5)).}
In all subsequent tests, we assume a linear response $\kappa(\theta)=\kappa_s \theta $ of the spring, where the constant elastic modulus $\kappa_s$ will be specified in each test.

For what concerns  mesh generation, in the numerical tests we use a sequence of underlying square meshes $\{\Ph\}_h$ (where $h$ is the length of the edges of the squares) and we cut them with $\Gamma$. In order to avoid machine precision issues we collapse two vertexes of the associated cut mesh $\PhG$ if the distance between them is less then $\texttt{1e-14}$ with respect to the mesh size.
We investigate the results obtained with the VEM schemes of order one and two (denoted with $k=1$ and $k=2$, respectively); see  Remark \ref{rm:k2}.

{Notice that in the setting under investigation two possible situations can occur: if $1/h$ is even the hinged point $O$ is a vertex of the underlying mesh, while if $1/h$ is odd the point $O$ corresponds to the midpoint of an edge. Since this two situations may yield different mesh configurations, we analyse both cases in the following numerical tests.}
As a consequence of the mesh cutting procedure, anisotropic elements can be generated when the absolute angle value $|\theta|$ is very small or near $\pi/2$ (this latter configuration being outside our scope since it would require to include a contact condition among the vessel walls and the leaflet).
Furthermore, note that the mesh cutting procedure can generate very small elements and edges, that is elements with a diameter that is much smaller than $h$ and elements of diameter comparable to $h$ having edges that are much smaller. 

{
\subsection{Numerical experiments}
We now conduct three comprehensive numerical tests with the proposed VEM.\looseness=-1
}
\paragraph{Test 1: Study of the $\tau_h$ functional.}
\label{test1}
In the present test we assess the robustness of the VEM technology for fluid-structure interaction problems, in particular we evaluate the qualitative behaviour of the discrete torque functional $\tau_h$ in terms of continuity and monotonicity with respect to $\theta$, the condition number of the resulting linear system and the discrete inf-sup constant $\beta_h$ 
$$
\beta_h := \infsup{q_h \in Q_h / \{ \boldsymbol{0} \}  } {\boldsymbol{v}_h \in \VV_h/ \{ \boldsymbol{0} \} }
  \frac{b(\boldsymbol{v}_h, q_h)}{ \|\boldsymbol{v}_h\|_{\VV} \|q_h\|_Q }  .
$$
For the sake of simplicity, in the present test we consider the Stokes version of \eqref{eq:fsi_primale} (that is, without the convective term) with spring modulus $\kappa_s = \texttt{1}$ and the data described above.

In Figs \ref{fg:gtheta15-31} and \ref{fg:gtheta16-32} we plot the function $\tau_h$ obtained in the  ``odd case'' and in the ``even case'' for two levels of refinements ($1/h = 15$, $31$ and $1/h = 16$, $32$, respectively) for $k=1$, $2$ and the aforementioned choices of the stabilization forms. 
To validate the performances of the proposed numerical scheme we compute a ``reference torque'' $\tau$ manufactured in the following way: for any angle $\theta$  we build an ad-hoc (shape regular) very fine triangular mesh in the domain $\Omega=\D\setminus\Gamma(\theta)$. 
The output torque functional $\tau(\theta)$ is thus computed employing the well known {Crouzeix-Raviart Stokes finite element} $([\Pk_2 \oplus \mathcal{B}]^2, \text{ disc. }\Pk_1)$ on such a  mesh (with diameter $h=\texttt{0.02}$). 
Obviously this is a very expensive procedure in practice, but is acceptable to generate a reference solution.

The plots show the qualitative behaviour for the discrete functions $\tau_h$, in particular we can observe the following facts:
\begin{enumerate}[$\bullet$]
\item the graphs of the functions $\tau_h$ approach that of the reference function $\tau$ when $h$ decreases. As expected, the case $k=2$ yields better result than the case $k=1$.
Furthermore, we notice that, at least for the present data, the \texttt{dofi-dofi} stabilization produces discrete functions $\tau_h$ closer to the reference function $\tau$;
\item the graph of the function $\tau_h$ exhibits small jumps (or bumps) of amplitude decreasing with $h$, that appear when the leaflet tip (or its prolongation) crosses a vertex of the background mesh. This phenomenon is more evident for the \texttt{dofi-dofi} stabilization, whereas the \texttt{trace} stabilization  has better performances in terms of continuity of the associated discrete function $\tau_h$. An investigation of this aspect is presented below;
\item the function $\tau_h$ has, roughly speaking, a decreasing monotone trend with respect to $\theta$, that is only perturbed by the aforementioned small jumps or bumps.  This is in agreement with the physical intuition. Also note that the approximation gets better as $\theta$ increases.
\end{enumerate}

\begin{figure}[!t]
\center{
\includegraphics[scale=0.2]{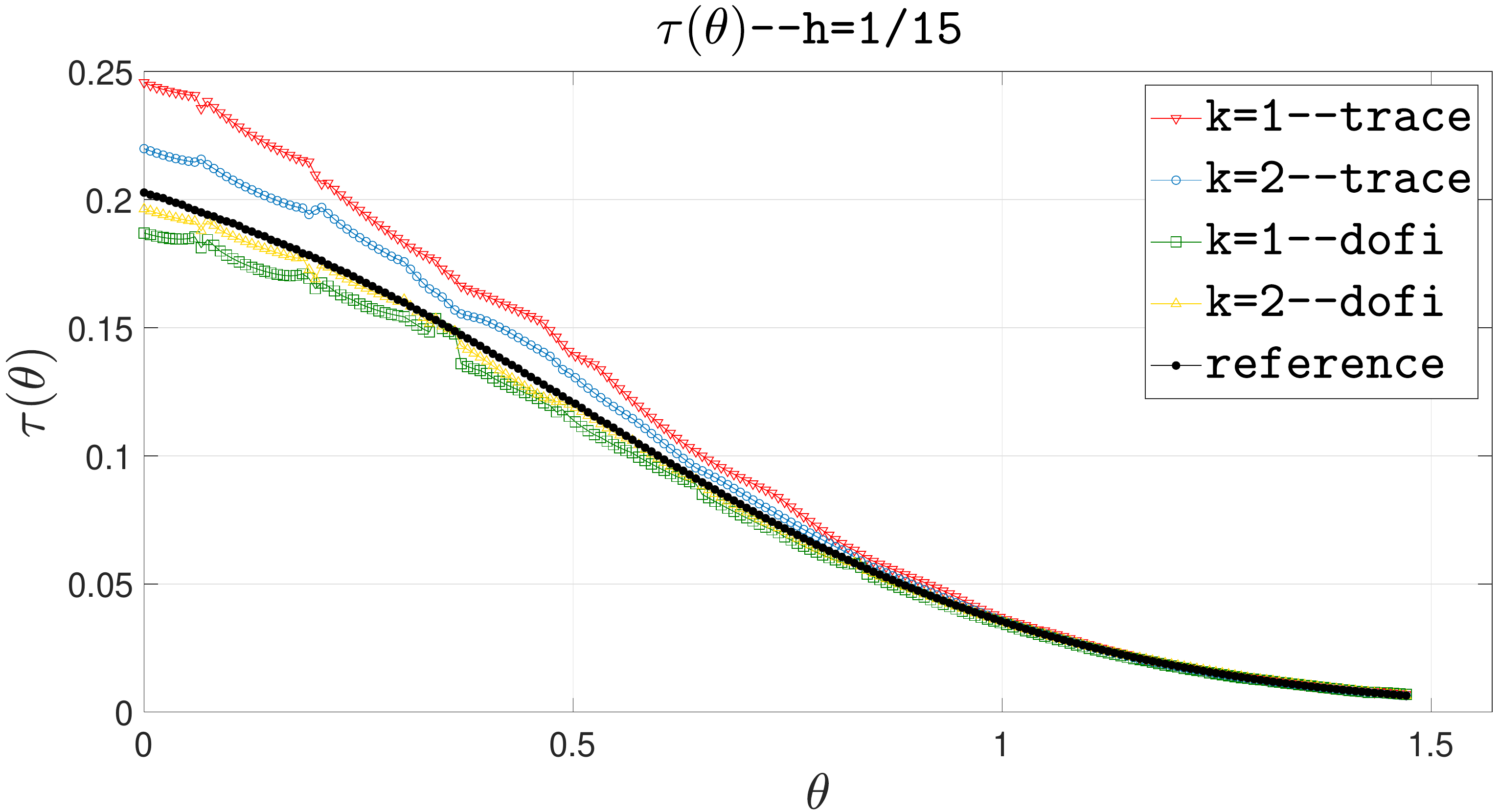}
\vspace{2ex}
\includegraphics[scale=0.2]{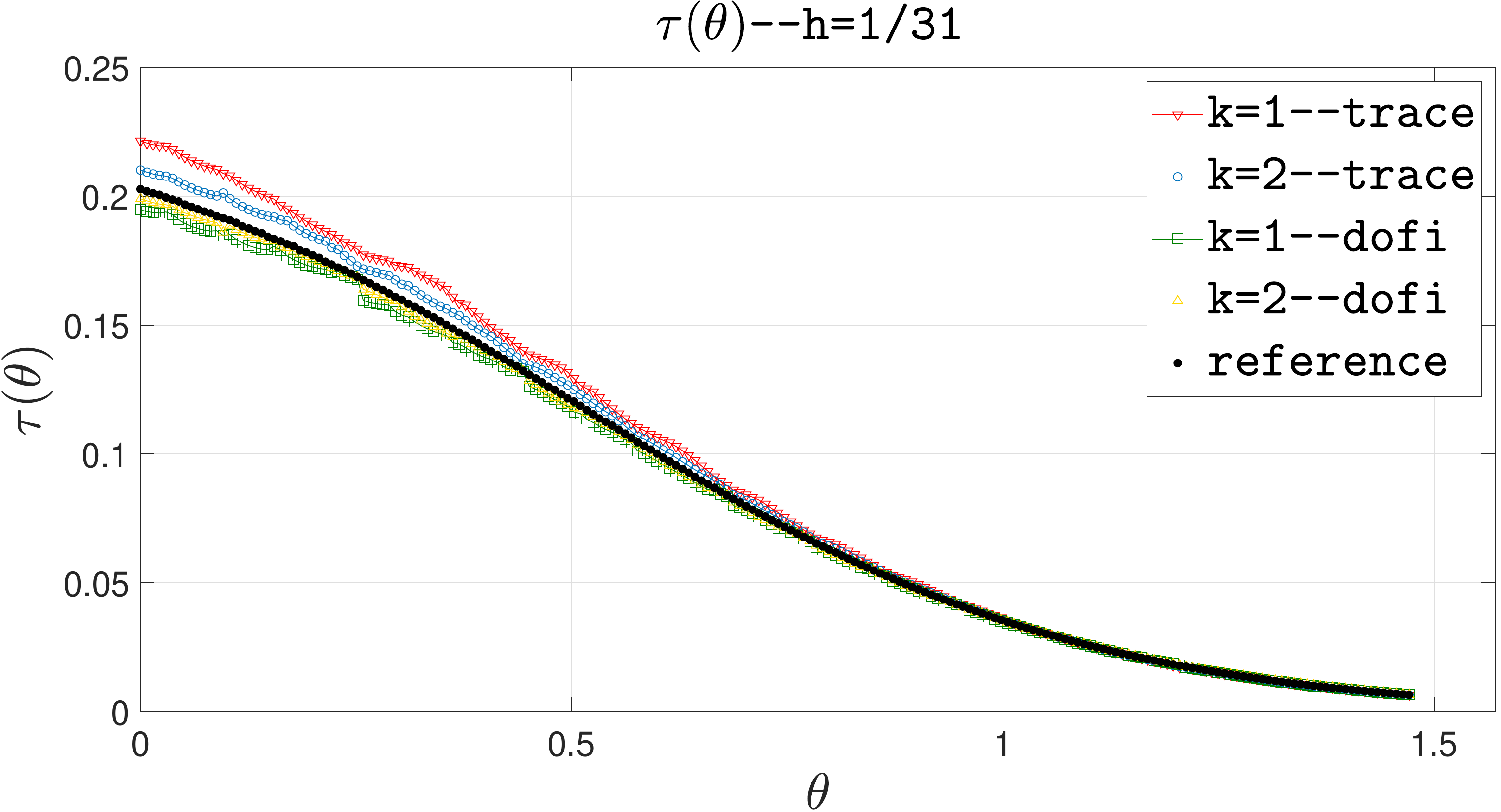}
}
\caption{{Test 1. $\tau_h$ obtained with $1/h= 15$ (left) and  $1/h= 31$ (right) for $k=1$, $2$, \texttt{trace} and \texttt{dofi-dofi} stabilizations.}}
\label{fg:gtheta15-31}
\end{figure}
\begin{figure}[!t]
\center{
\includegraphics[scale=0.2]{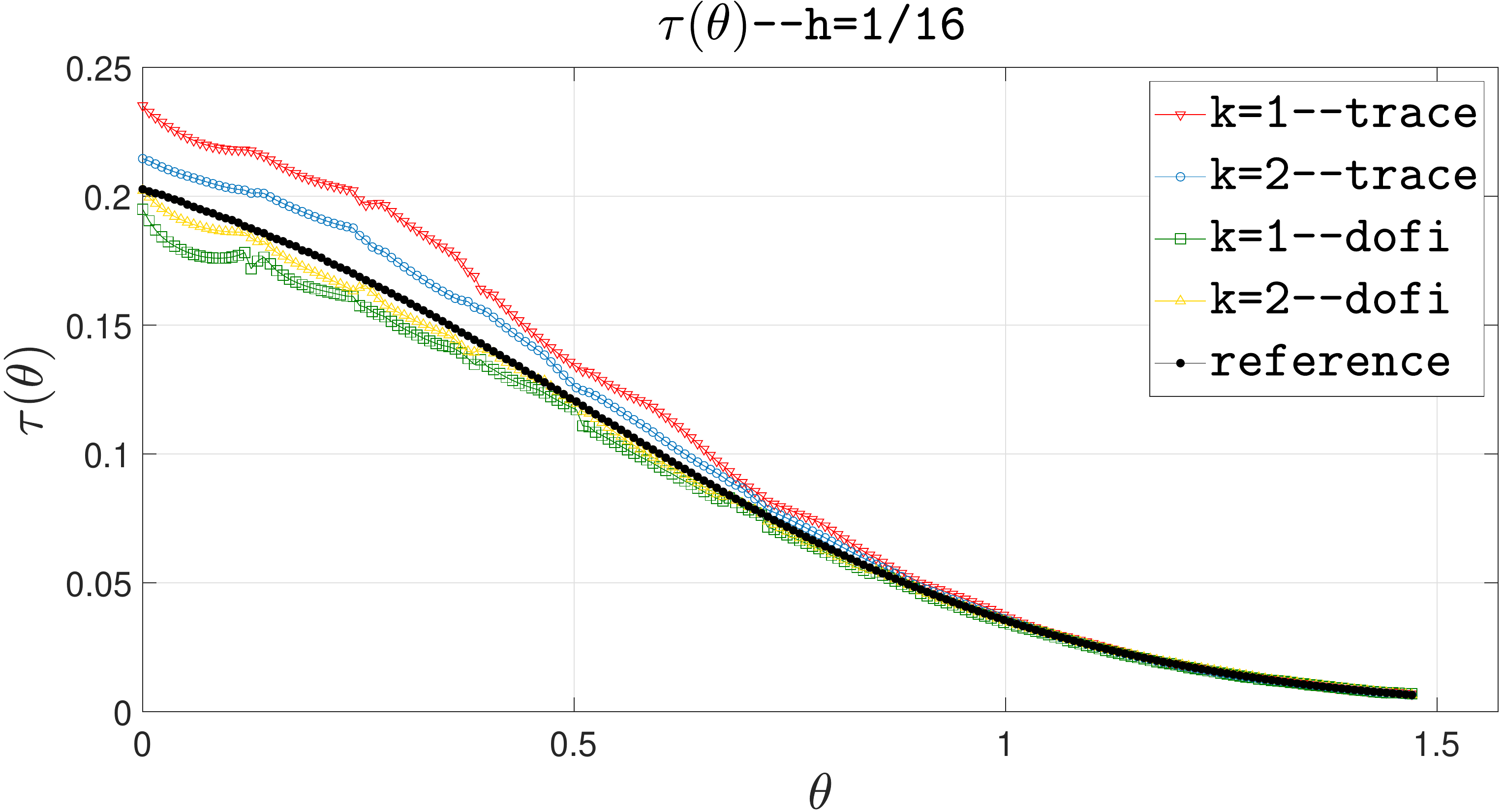}
\vspace{2ex}
\includegraphics[scale=0.2]{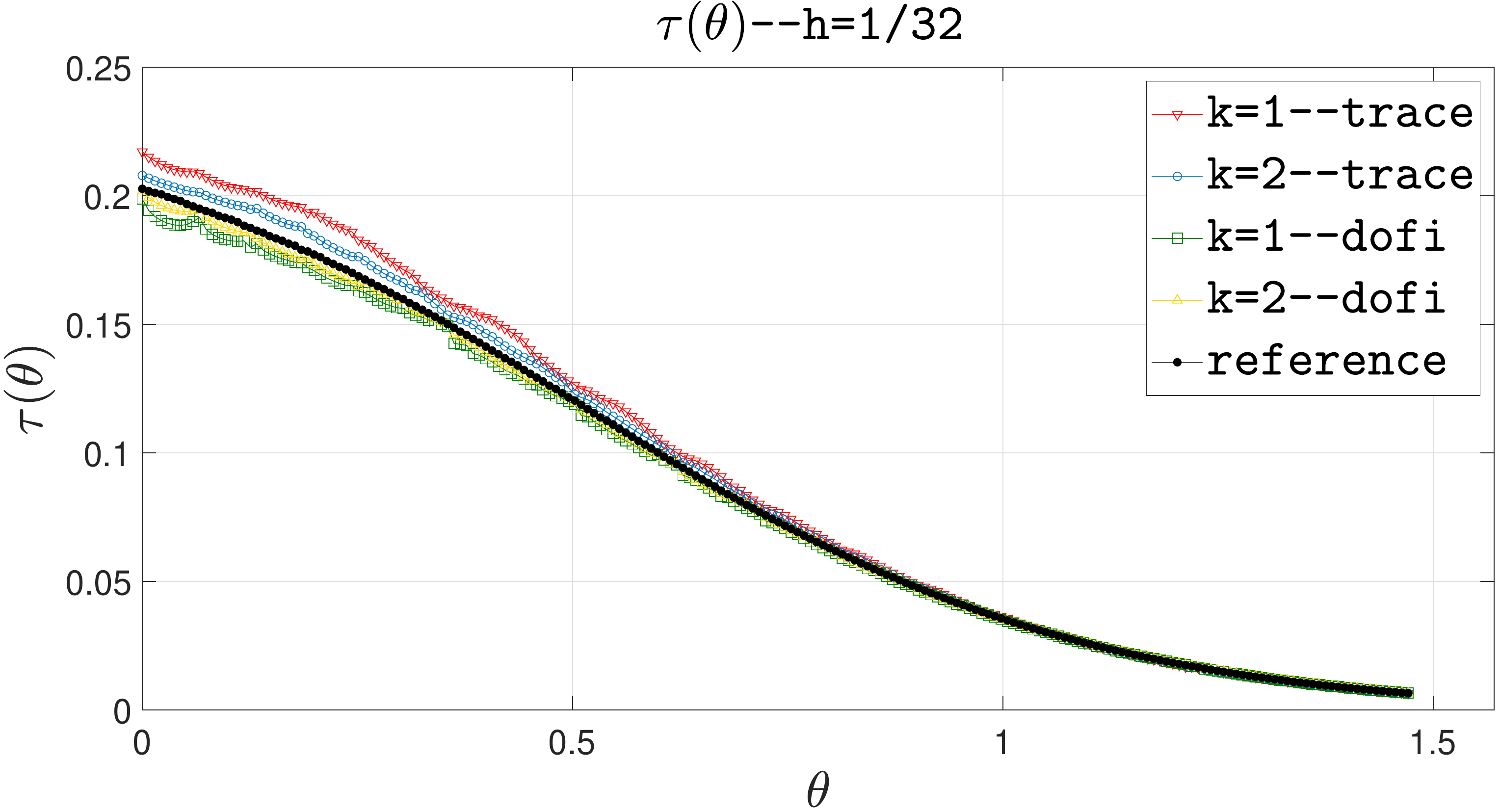}
}
\caption{{Test 1.  $\tau_h$ obtained with $1/h= 16$ (left) and  $1/h= 32$ (right) for $k=1$, $2$, \texttt{trace} and \texttt{dofi-dofi} stabilizations.}}
\label{fg:gtheta16-32}
\end{figure}

{In order to investigate the small jump/bump phenomenon detected above, in Fig. \ref{fg:tau15_meshes} and
Fig. \ref{fg:tau16_meshes}  we depict a zoom of the $\tau_h$ graph for some critical ranges of the angle $\theta$. 
We notice, as expected, that the jumps and bumps are related to a change in the topology of the mesh. In order to better appreciate this,  in both graphs we plot dashed vertical lines that mark the angle values associated to the mesh configurations shown in the lower part of the figure. 
For example, in Fig. \ref{fg:tau15_meshes} the jumps/bumps of cases \texttt{A} and \texttt{C} are generated by the leaflet (or its prolongation) crossing a mesh vertex; case \texttt{B} is instead generated by the leaflet tip crossing a mesh edge (which creates a big change in the local mesh configuration due to the leaflet prolongation procedure). Analogous observations can be made for Fig. \ref{fg:tau16_meshes}, cases \texttt{A}, \texttt{B}, \texttt{C}, \texttt{D}.
The difference between the two figures is that in Fig. \ref{fg:tau15_meshes}  the considered angles $\theta$ are very small (thus yielding anisotropic elements in addition to small edges/elements) while in Fig. \ref{fg:tau16_meshes} the considered angles are large (thus anisotropic elements are ruled out but small edges/elements can still be present). Some observations are in order.
\begin{enumerate}[$\bullet$]
\item By comparison of the Figures \ref{fg:tau15_meshes} and \ref{fg:tau16_meshes} one can immediately appreciate that the absence of anisotropy yields a much milder jump/bump phenomenon. 
\item At the critical angles, the \texttt{dofi-dofi} stabilization may generate jumps in the functional, with decreasing amplitude as $h$ tends to zero. The presence of such jumps appears to be related to the particular form of this  stabilization. Indeed, for the \texttt{dofi-dofi} stabilization a change in topology may modify the number of edges (and thus nodes) in an element thus leading to a smaller or larger sum in \eqref{eq:dofidofi}, which can justify the jumps in the $\tau_h$ graph. For instance, the small anisotropic triangle appearing in subfigure \texttt{B} of Fig. \ref{fg:tau15_meshes} has four edges before the leaflet tip touches the vertical line (due to the leaflet prolongation procedure), that become 3 edges after the tip has crossed such line.
\item Contrary to the \texttt{dofi-dofi} case, the \texttt{trace} stabilization generates bumps instead of jumps at the critical angles. Therefore only the function monotonicity, but not its continuity, is broken. This preferable behaviour of the \texttt{trace} stabilization may be partially associated to its known robustness in the presence of small edges ~\cite{wriggers,BdV-Lovadina-Russo,Brenner-Sung}. 
\item The changes in the mesh topology that happen at the critical angles yield abrupt modifications also for the corresponding pressure space. In order to check the influence of this pressure changes on the jumps/bumps previously mentioned, we ran  an analogous problem with a (vector) Laplace model problem (that is, without the divergence-free constraint and the corresponding pressure space).
Since the same jumps/bumps were found also in the new test problem, although with a smaller amplitude, we deduce that the incompressibility constraint  is not alone the cause of such phenomena. 
\end{enumerate}

}

\begin{figure}[!h]
\center{
\begin{overpic}[scale=0.25]{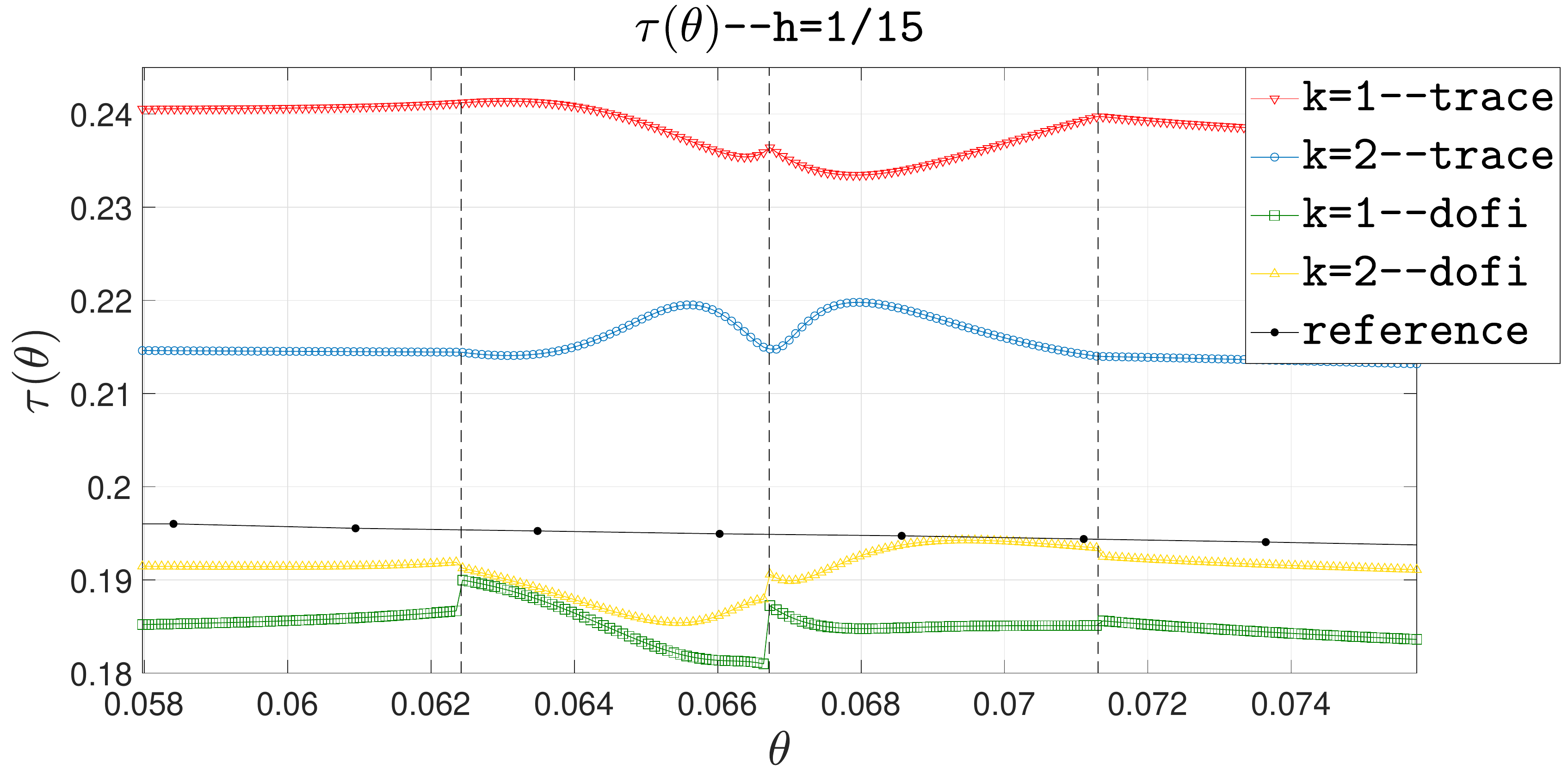} 
\put (25,20) {\texttt{case A}}
\put (45,20) {\texttt{case B}}
\put (66,20) {\texttt{case C}}
\end{overpic}
\\
\vspace{0.5cm}
\begin{overpic}[scale=0.12]{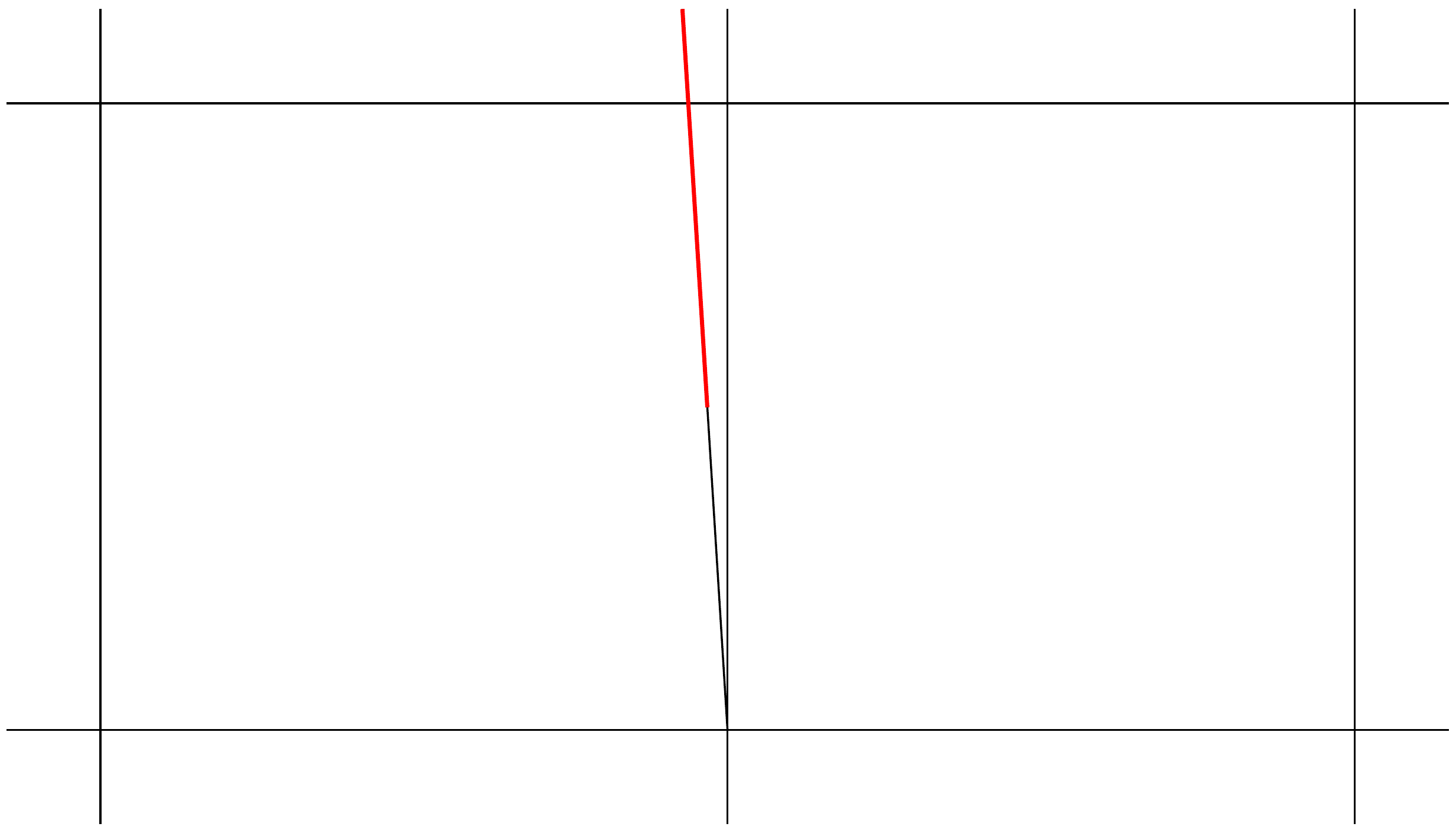}
\put (35,60) {\texttt{case A}}
\end{overpic}
\quad
\begin{overpic}[scale=0.12]{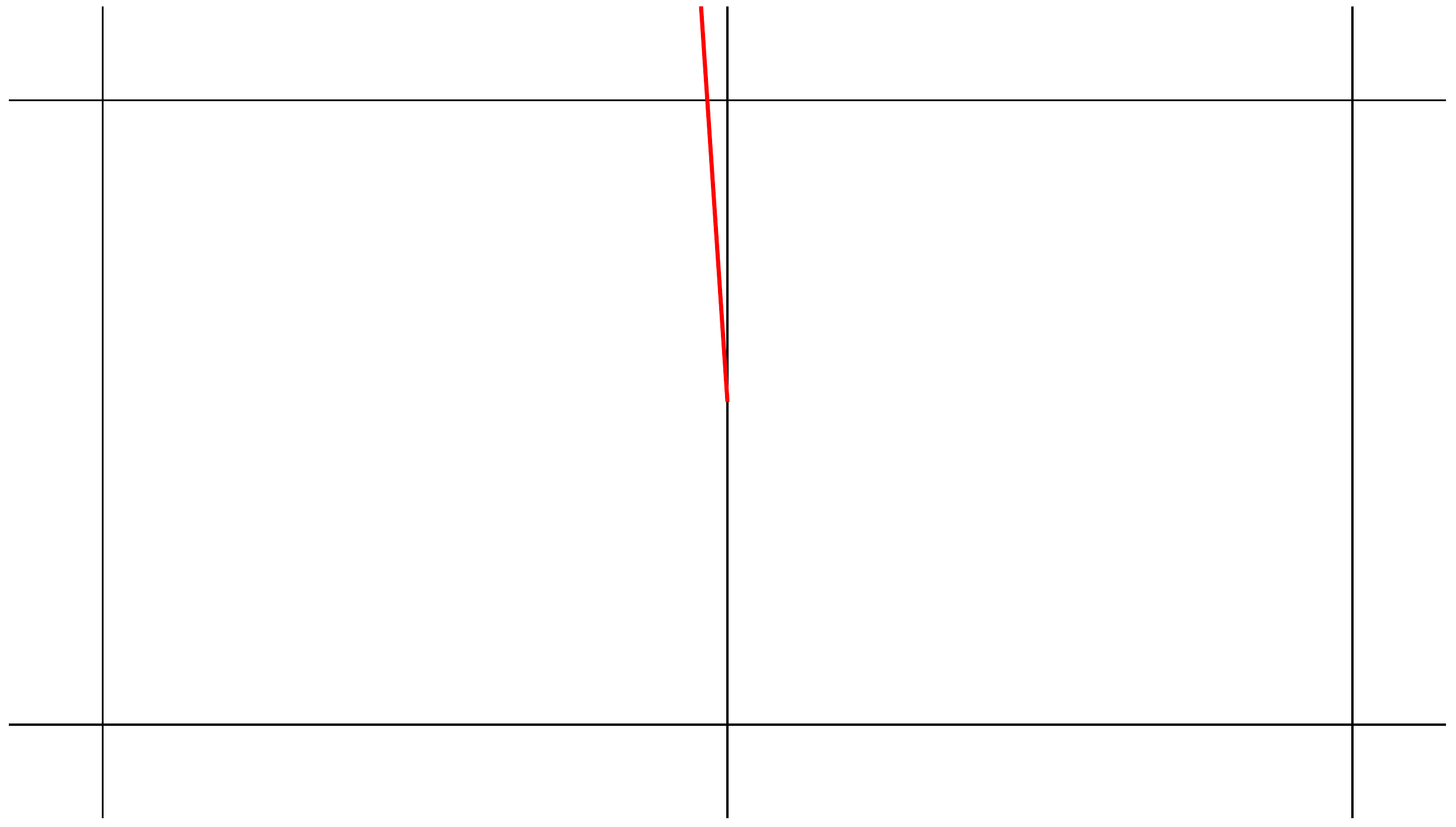}
\put (35,60) {\texttt{case B}}
\end{overpic}
\quad
\begin{overpic}[scale=0.12]{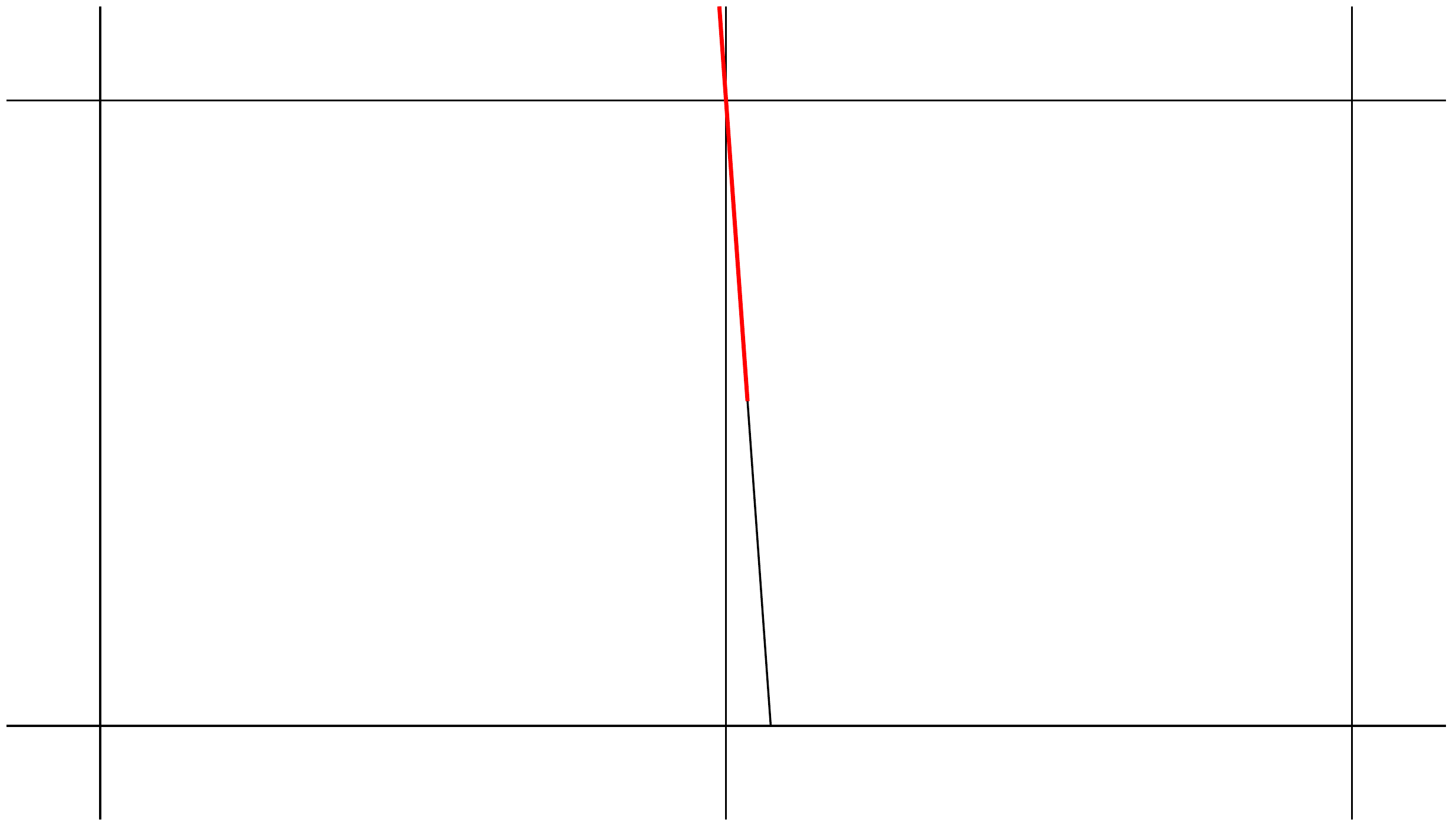}
\put (35,60) {\texttt{case C}}
\end{overpic}
}
\caption{{Test 1.  $\tau_h$ obtained with $1/h= 15$ for $k=1$, $2$, \texttt{trace} and \texttt{dofi-dofi} stabilizations near $\theta = 0$ (upper). Mesh configurations at the critical angles (lower). \texttt{case A}: leaflet prolongation crossing a mesh vertex;   \texttt{case B}: leaflet tip crossing a mesh edge; \texttt{case C}: leaflet  crossing a mesh vertex.
\texttt{case B} generates the more evident bump/jump of the function $\tau_h$.}}
\label{fg:tau15_meshes}
\end{figure}

\begin{figure}[!h]
\center{
\begin{overpic}[scale=0.25]{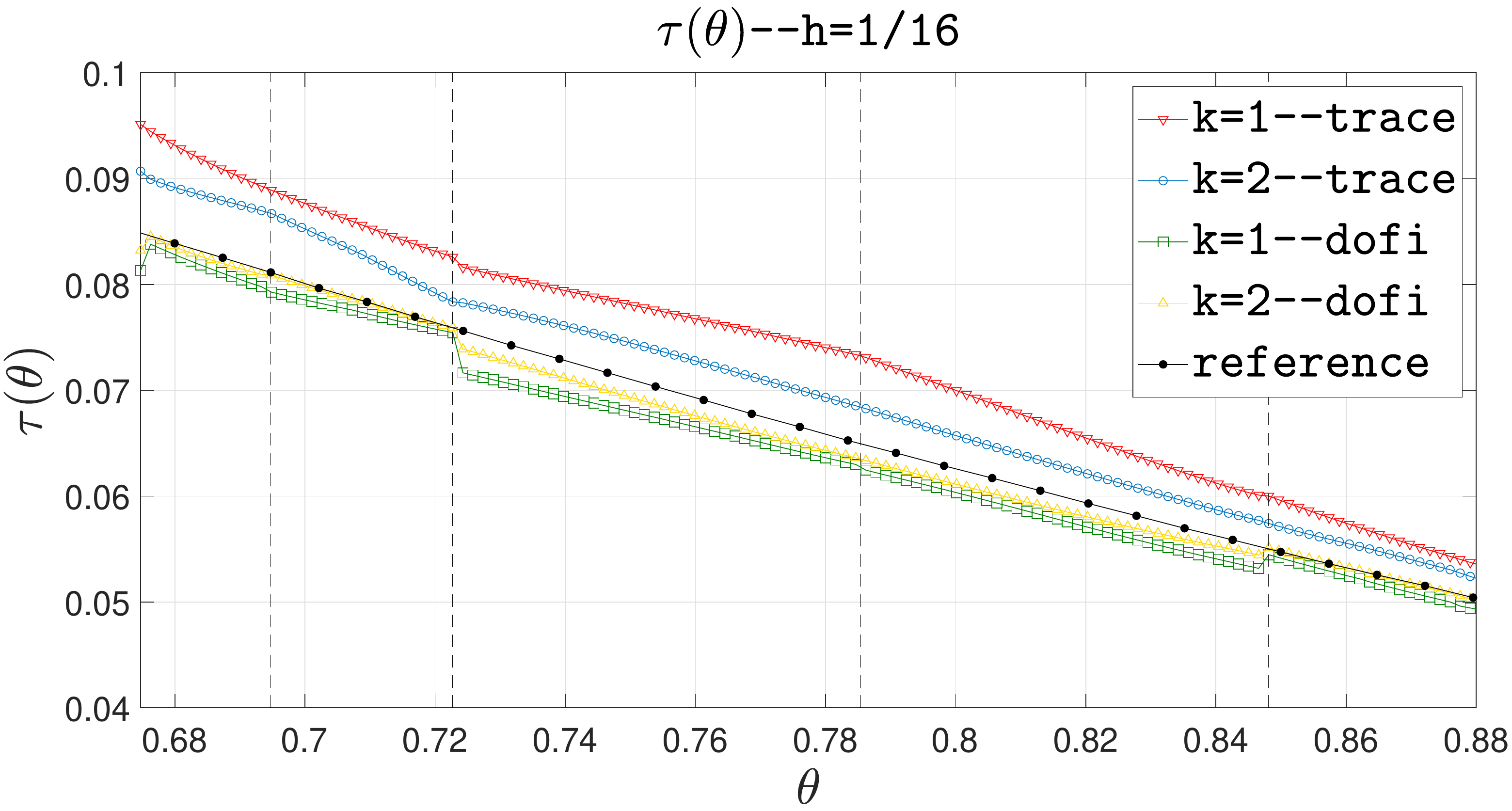} 
\put (13,12) {\texttt{case A}}
\put (25,12) {\texttt{case B}}
\put (52,12) {\texttt{case C}}
\put (79,12) {\texttt{case D}}
\end{overpic}
\\
\vspace{0.5cm}
\begin{overpic}[scale=0.15]{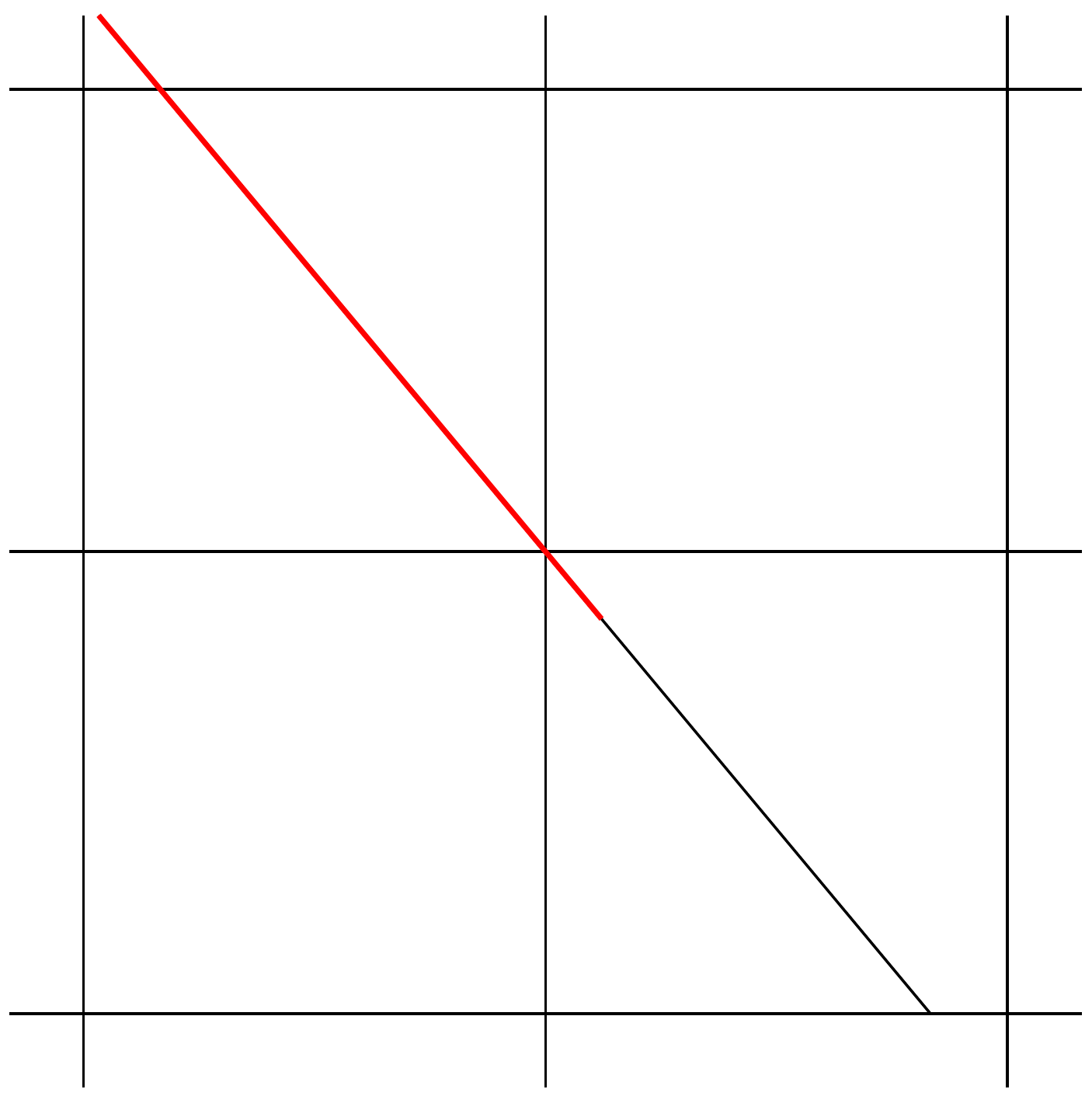}
\put (30,105) {\texttt{case A}}
\end{overpic}
\quad
\begin{overpic}[scale=0.15]{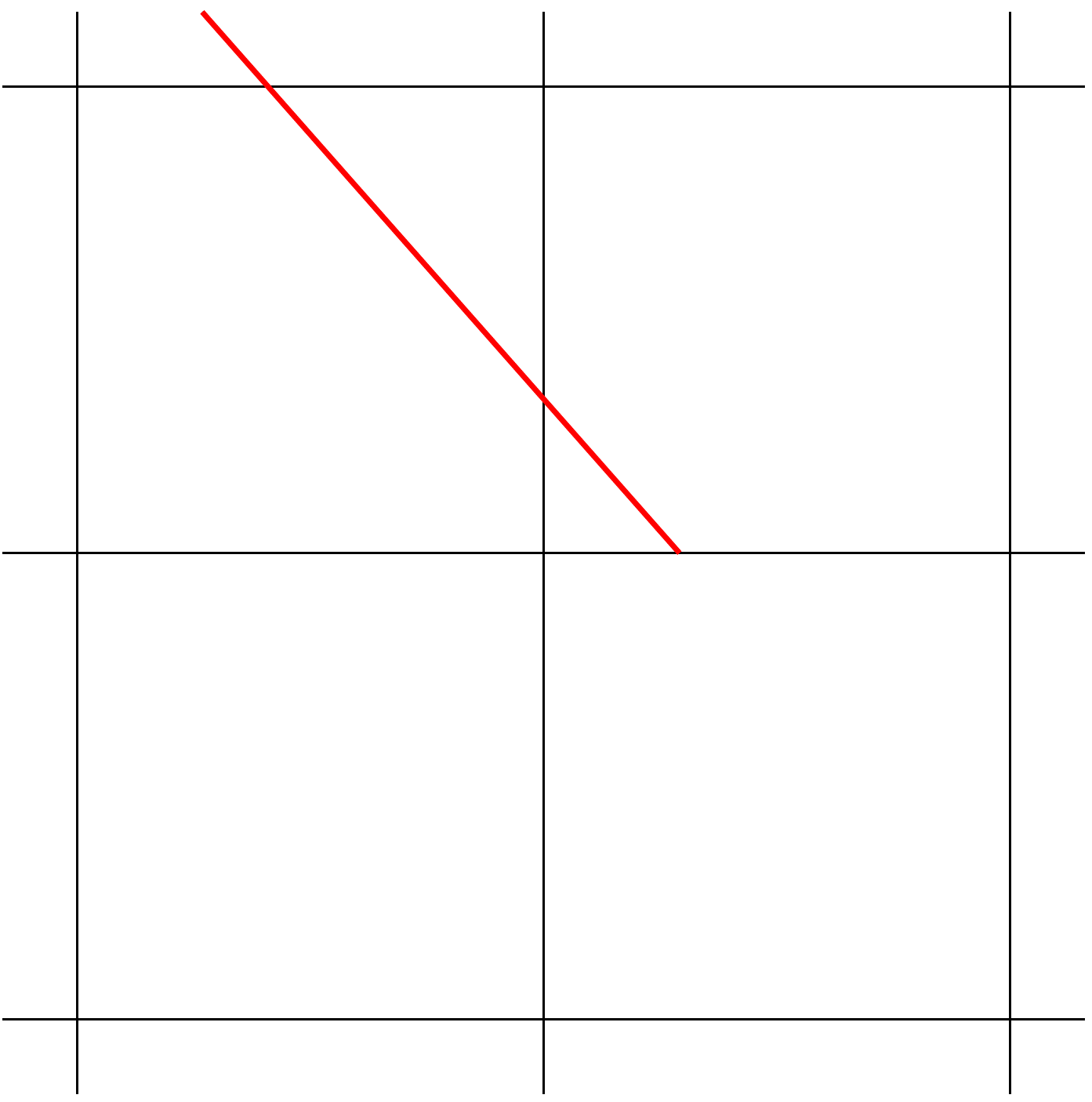}
\put (30,105) {\texttt{case B}}
\end{overpic}
\quad
\begin{overpic}[scale=0.15]{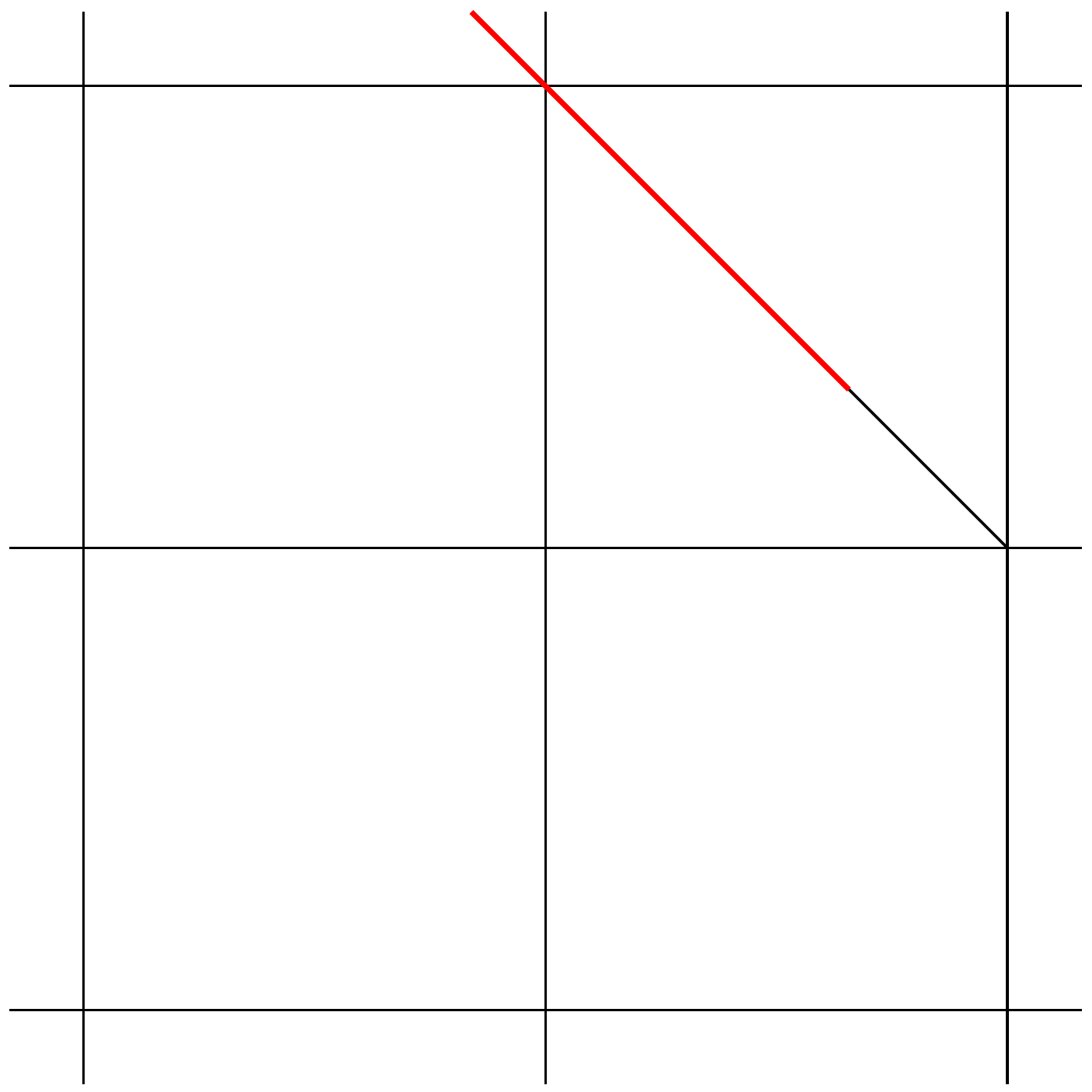}
\put (30,105) {\texttt{case C}}
\end{overpic}
\quad
\begin{overpic}[scale=0.15]{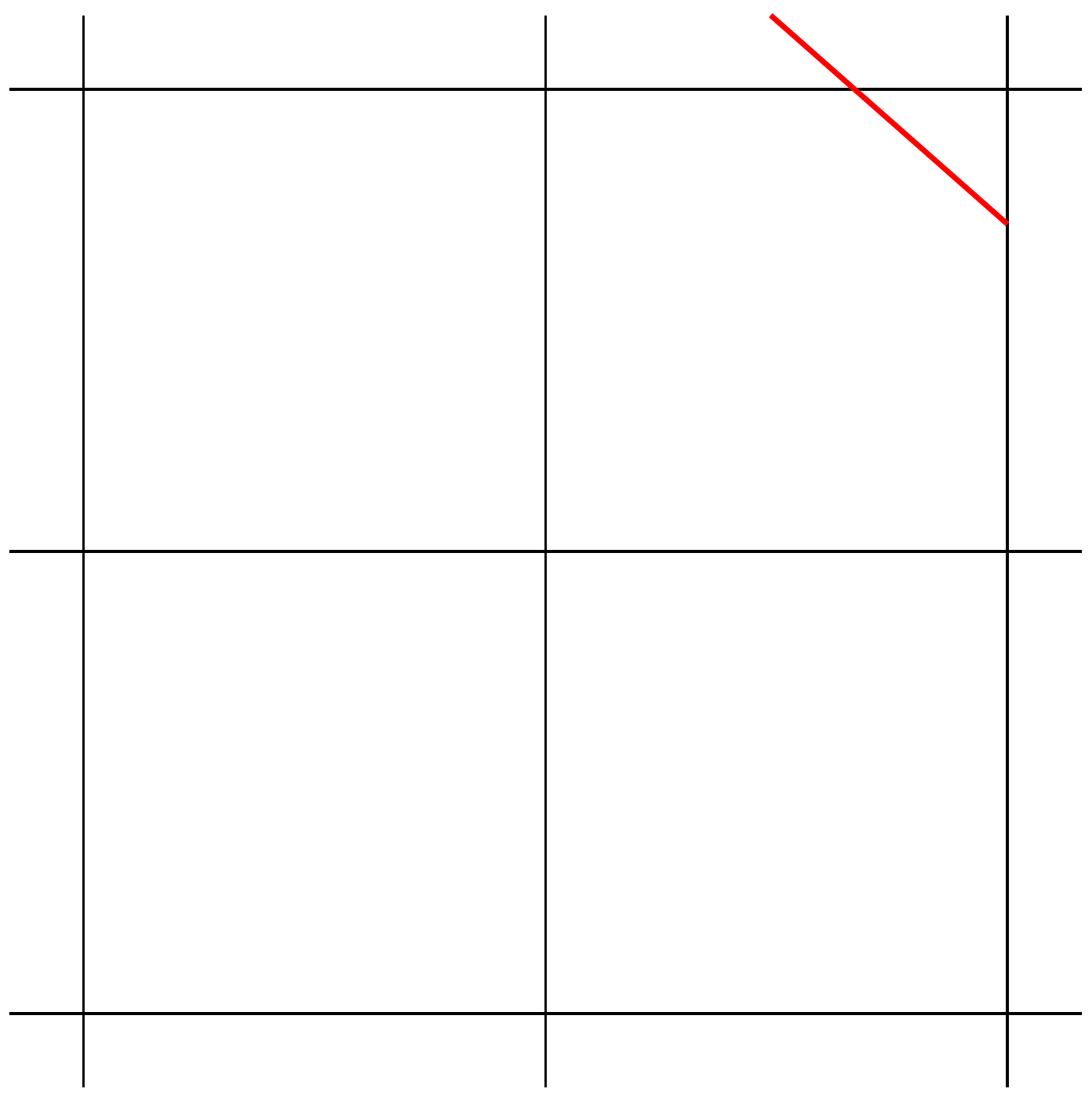}
\put (30,105) {\texttt{case D}}
\end{overpic}
}
\caption{{Test 1.  $\tau_h$ obtained with $1/h= 16$ for $k=1$, $2$, \texttt{trace} and \texttt{dofi-dofi} stabilizations near $\theta = \frac{\pi}{4}$ (upper). Mesh configurations at the critical angles (lower).
\texttt{case A}: leaflet crossing a mesh vertex;   \texttt{case B}: leaflet tip crossing a mesh edge; \texttt{case C}: leaflet and leaflet prolongation crossing a mesh vertex;
\texttt{case D}: leaflet tip crossing a mesh edge.
\texttt{case B} and \texttt{case D} generate the more evident bumps/jumps of the function $\tau_h$.}}
\label{fg:tau16_meshes}
\end{figure}
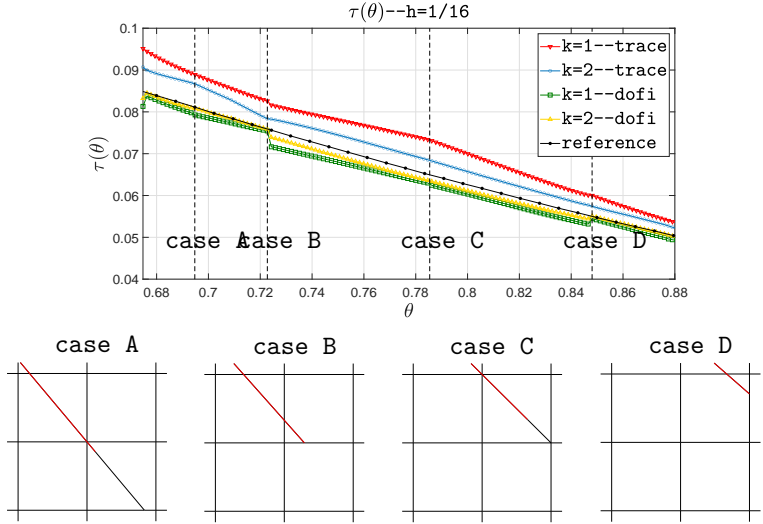

As a final remark, we must underline that all the above numerical perturbations of the $\tau_h$ functional get smaller as $h \rightarrow 0$ and, if one considers the strong local mesh topological changes in action, the scheme is still surprisingly robust. In a practical situation, one would not adopt coarse meshes such as those previously presented: in Fig. \ref{fg:gtheta-local} we plot the graph of $\tau_h$ for a fine mesh with $h = 1/128$ in the angle range $[0, 0.5]$. The smoothness can be clearly appreciated compared with the same angle range in Figs \ref{fg:gtheta15-31} and \ref{fg:gtheta16-32}.

\begin{figure}[!h]
\center{
\includegraphics[scale=0.225]{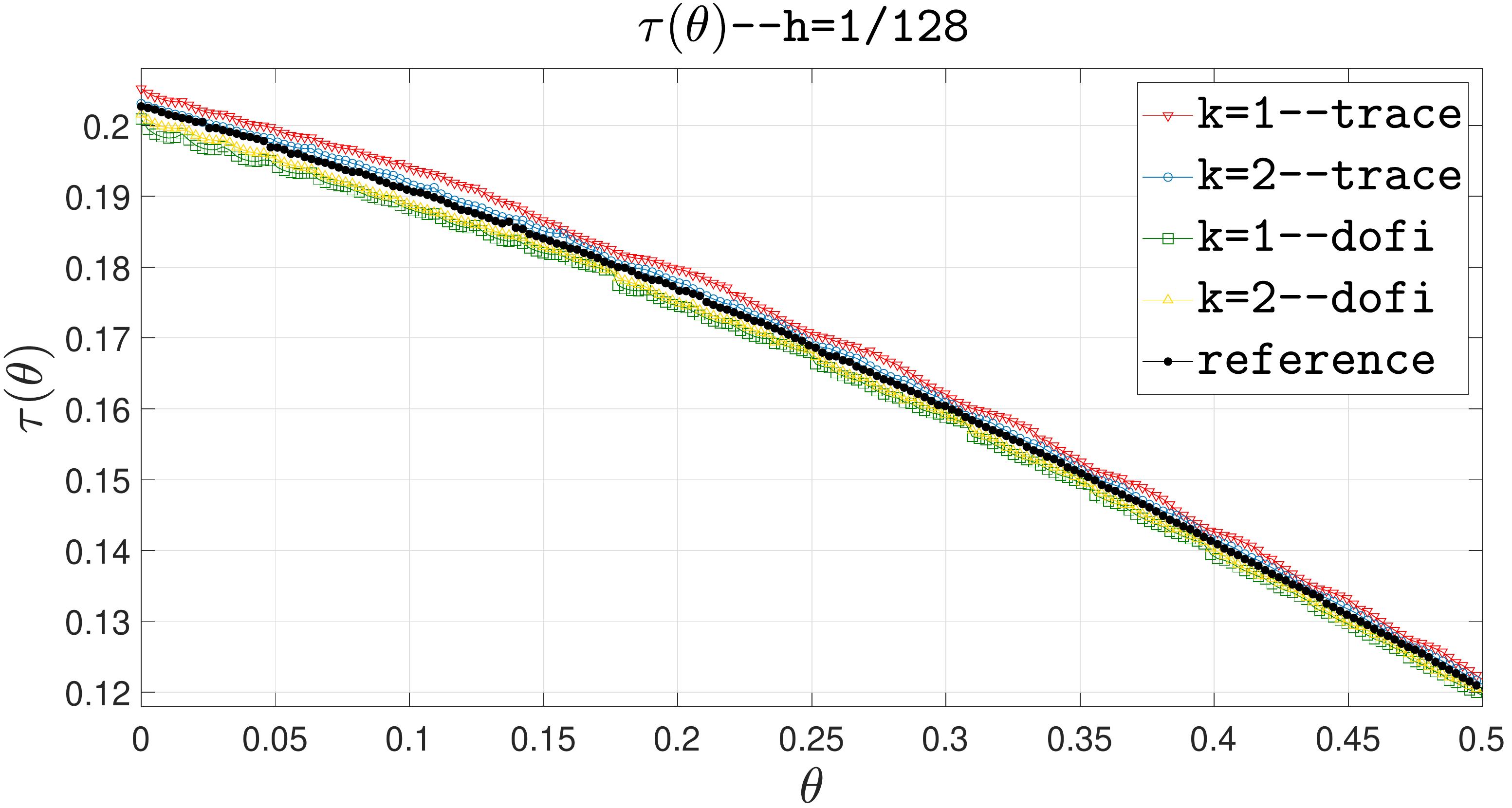}
}
\caption{{Test 1.  $\tau_h$ obtained with $1/h= 128$ for $k=1$, $2$, \texttt{trace} and \texttt{dofi-dofi} stabilizations. Focus on the interval $[0, 0.5]$.}}
\label{fg:gtheta-local}
\end{figure}

Fig. \ref{fig:test_cond} exhibits the condition number of the resulting linear system of the discrete scheme \eqref{eq:vem-stokes-fsi} as a function of the angular coordinate $\theta$.
We notice that the condition number suffers from the anisotropy of the elements (small angles $\theta$ and $1/h = 32$). Nevertheless as observed above (compare Figs \ref{fg:gtheta15-31} and \ref{fg:gtheta16-32}) the ill-conditioning of the problem seems not to affect the computation of $\tau_h$, at least for the direct solver adopted here.
Moreover we observe that for the \texttt{dofi-dofi} stabilization the condition number is more stable, in comparison with the \texttt{trace} stabilization, with respect to the presence of elements or edges with diameter/length that is orders of magnitude smaller than $h$.
{This can be roughly justified by considering a generic element of size $h_E$ with a ``small'' edge of length $h_e$ and vertexes $\nu,\nu'$. It is easy to check that the dual basis function $\phi$ associated to the vertex $\nu$ (or $\nu'$) satisfies $\mathcal{S}_{\texttt{trace}}^E(\cdot,\cdot) \sim h_E/h_e$. Therefore in the presence of large ratios $h_E/h_e >\!\!> 1$ this has a clear detrimental effect on the condition number of the ensuing stiffness matrix.}

\begin{figure}[!b]
\center{
\includegraphics[scale=0.19]{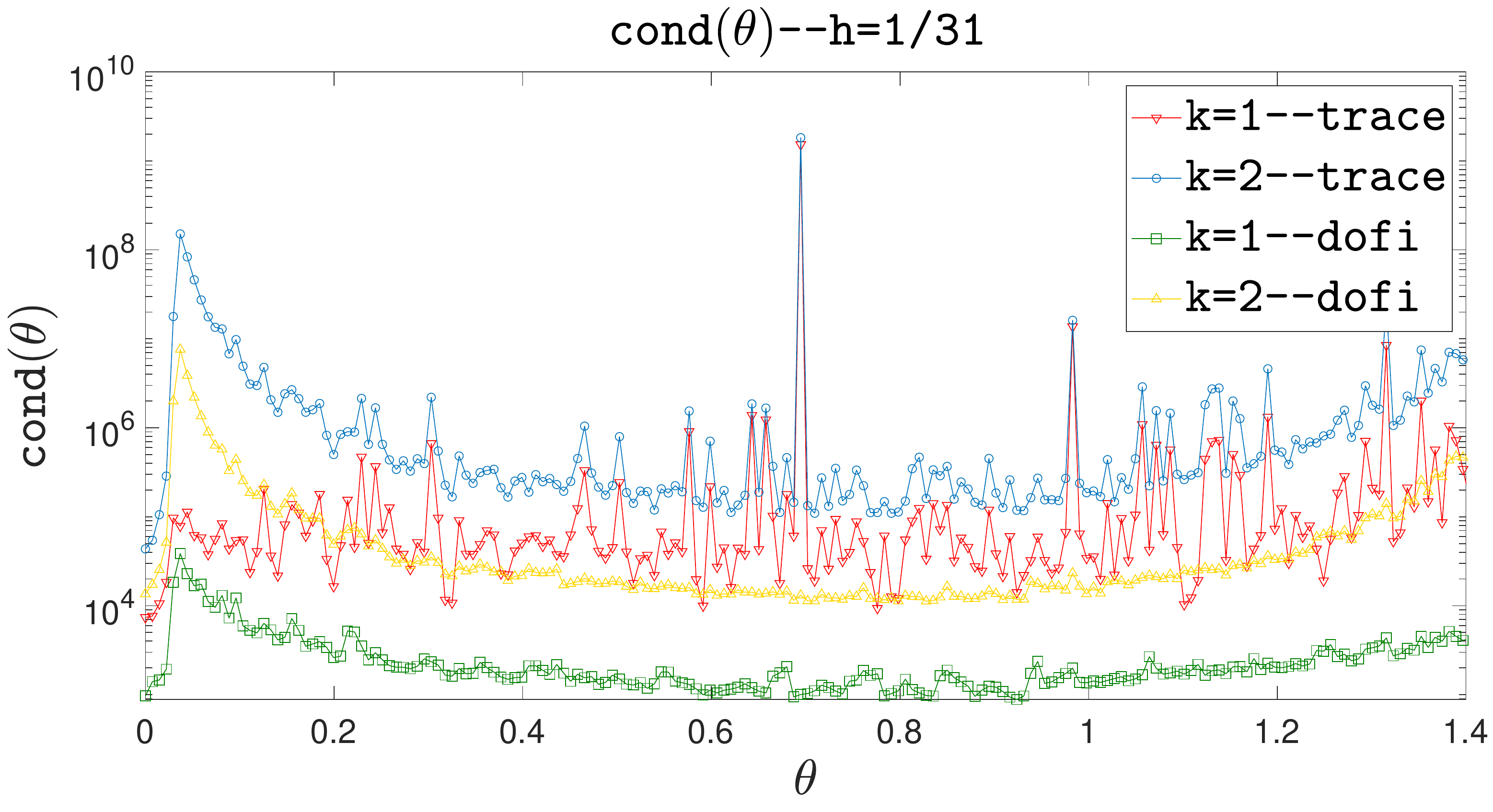}
\vspace{2ex}
\includegraphics[scale=0.19]{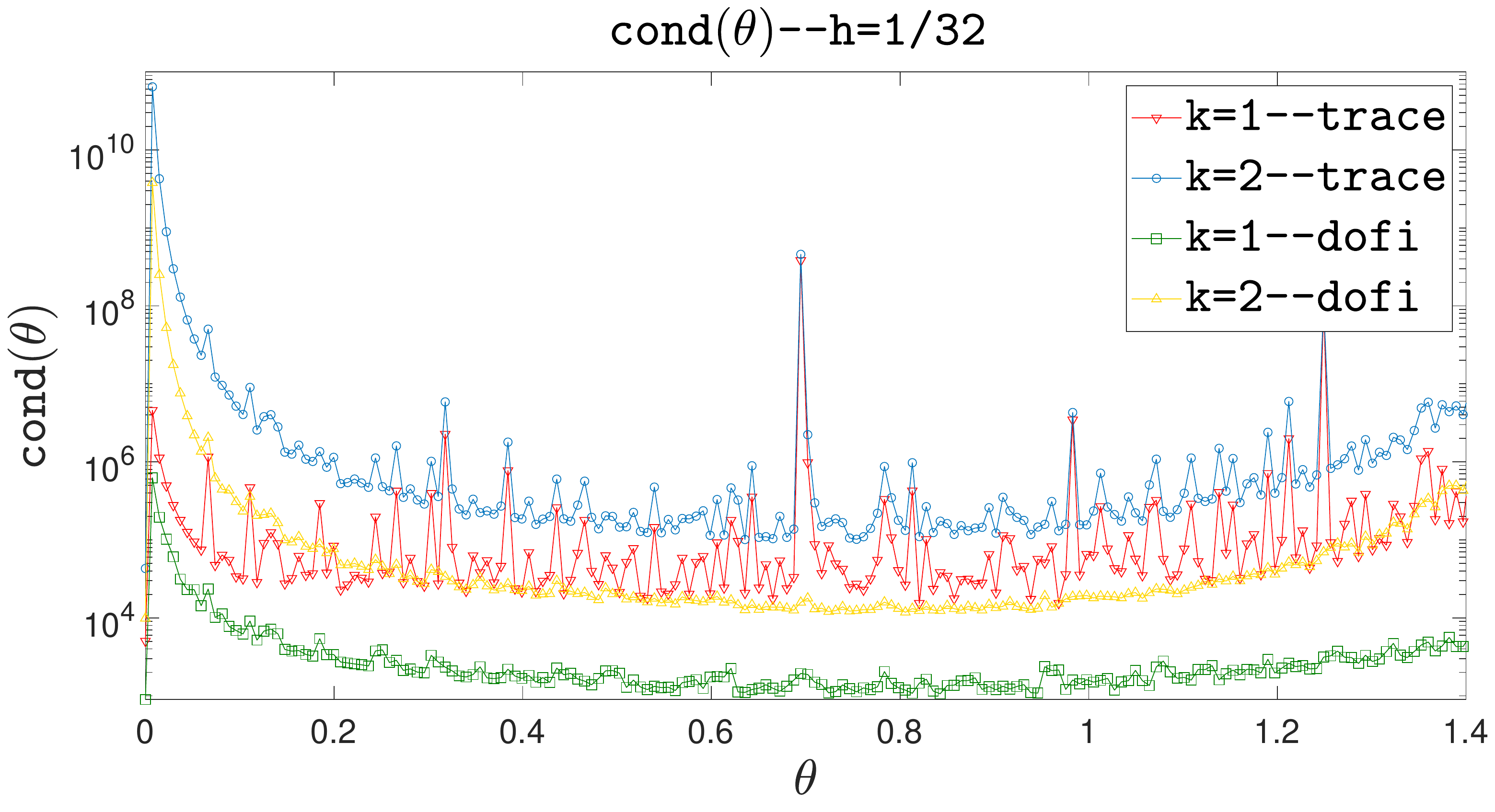}
\caption{{Test 1. 
Condition number with $1/h = 31$ (left) and $1/h = 32$ (right) for $k=1$, $2$.
The condition number suffers from the anisotropy of the elements (small angles $\theta$ and $1/h = 32$). The condition number for the \texttt{dofi-dofi} stabilization is more stable in comparison with the \texttt{trace} stabilization.}}
\label{fig:test_cond}
}
\end{figure}

At last, Table \ref{tab:test_infsup} shows the stability of the discrete inf-sup constant $\beta_h$ with respect to the anisotropy and different sizes of the elements, for both $k=1$ and $k=2$, and both the adopted stabilizations. 
For the computation of $\beta_h$ we use the algebraic argument in \cite{inf-sup-test}. 
We pick small angles $\theta$ in order to assess the performance of the scheme with respect to anisotropy.
Table \ref{tab:test_infsup} clearly indicates that the inf-sup constant $\beta_h$ is robust with respect to the anisotropy of the mesh elements.

\begin{table}[!h]
\centering
\small{
\begin{tabular}{lccccc}
\toprule
&  & \multicolumn{2}{c}{\texttt{trace}}
& \multicolumn{2}{c}{\texttt{dofi-dofi}}
\\
\midrule  
&  $\theta$  
&  $\texttt{k = 1}$   
&  $\texttt{k = 2}$  
&  $\texttt{k = 1}$   
&  $\texttt{k = 2}$       
\\
\midrule                            
\multirow{4}*{\texttt{h=1/15}}
&$\texttt{1e-8}$       
&$\texttt{1.75014e-01}$         & $\texttt{1.77600e-01}$
&$\texttt{2.21828e-01}$         & $\texttt{1.84007e-01}$             
\\
&$\texttt{1e-6}$          
&$\texttt{1.75014e-01}$         & $\texttt{1.77600e-01}$
&$\texttt{2.21828e-01}$         & $\texttt{1.84007e-01}$       
\\
&$\texttt{1e-4}$       
&$\texttt{1.75011e-01}$         & $\texttt{1.77601e-01}$
&$\texttt{2.21825e-01}$         & $\texttt{1.84008e-01}$ 
\\
&$\texttt{1e-2}$       
&$\texttt{1.74655e-01}$        & $\texttt{1.77662e-01}$
&$\texttt{2.21448e-01}$        & $\texttt{1.84085e-01}$       
\\
\midrule
\multirow{4}*{\texttt{h=1/16}}
&$\texttt{1e-8}$       
&$\texttt{1.86731e-01}$         & $\texttt{6.24530e-02}$
&$\texttt{2.24419e-01}$         & $\texttt{1.81699e-01}$       
\\
&$\texttt{1e-6}$          
&$\texttt{1.86731e-01}$         & $\texttt{5.63922e-02}$
&$\texttt{2.24419e-01}$         & $\texttt{1.82237e-01}$       
\\
&$\texttt{1e-4}$       
&$\texttt{1.86735e-01}$         & $\texttt{5.62713e-02}$
&$\texttt{2.24415e-01}$         & $\texttt{1.82234e-01}$ 
\\
&$\texttt{1e-2}$       
&$\texttt{1.86933e-01}$        & $\texttt{6.49363e-02}$
&$\texttt{2.24960e-01}$        & $\texttt{1.83015e-01}$       
\\
\bottomrule
\end{tabular}
}
\caption{Test 1.  Inf-sup constant $\beta_h$ for $k=1$, $2$. Even and odd case, \texttt{trace} amd \texttt{dofi-dofi} stabilizations.  The inf-sup constant is robust with respect to the anisotropy of the mesh elements.}
\label{tab:test_infsup}
\end{table}

\paragraph{Test 2: Validation of the nonlinear scheme for the Stokes equations.}
\label{test2bis}\textbf{}
In the present test we numerically explore the convergence of the VEM scheme for  the  ``benchmark problem'' described in the following, with the aim of validating  the proposed discretization scheme and the associated nonlinear algorithm.
We consider again the linear (Stokes) version of \eqref{eq:fsi_primale} with the data  described above. 
Since no exact solution is explicitly available, we build a reliable reference numerical solution as follows. The idea is to fix an angle $\theta^*$ of the leaflet, compute the corresponding torque $\tau(\theta^*)$ by a highly accurate numerical scheme, and find the value $\kappa_s^*$ of the spring elastic modulus by imposing that $\theta^*$ is the equilibrium position of the leaflet. To be precise, we choose the angle $\theta^* := \pi/6 - \texttt{0.01}$, since this angle yields a complex mesh configuration with small elements and edges (see Fig. \ref{fg:benchmark_meshes}), thereby representing a severe test for the robustness of the VEM technology. The torque $\tau(\theta^*)$ is expensively, yet accurately computed as in Test 1 by the Crouzeix-Raviart method on a fine triangular mesh of diameter $h=\texttt{0.01}$ in the domain $\Omega=\D \setminus \Gamma(\theta^*)$. Finally, $\kappa_s^* := \frac{\tau(\theta^*)}{\theta^*}$ is obtained by enforcing the balance condition $\kappa_s^*\theta^*=\tau(\theta^*)$.

The value $\theta^*$ is approximated by the VEM scheme \eqref{eq:fsi_vem} (without the convective term $c_h$), solving the resulting nonlinear equation $\kappa_s^*\theta=\tau_h(\theta)$ by the bisection algorithm discussed in Subsection~\ref{sub:bisection}; let $\theta_h^*$ denote the output of this procedure.

\begin{figure}[!b]
\center{
\begin{overpic}[scale=0.17]{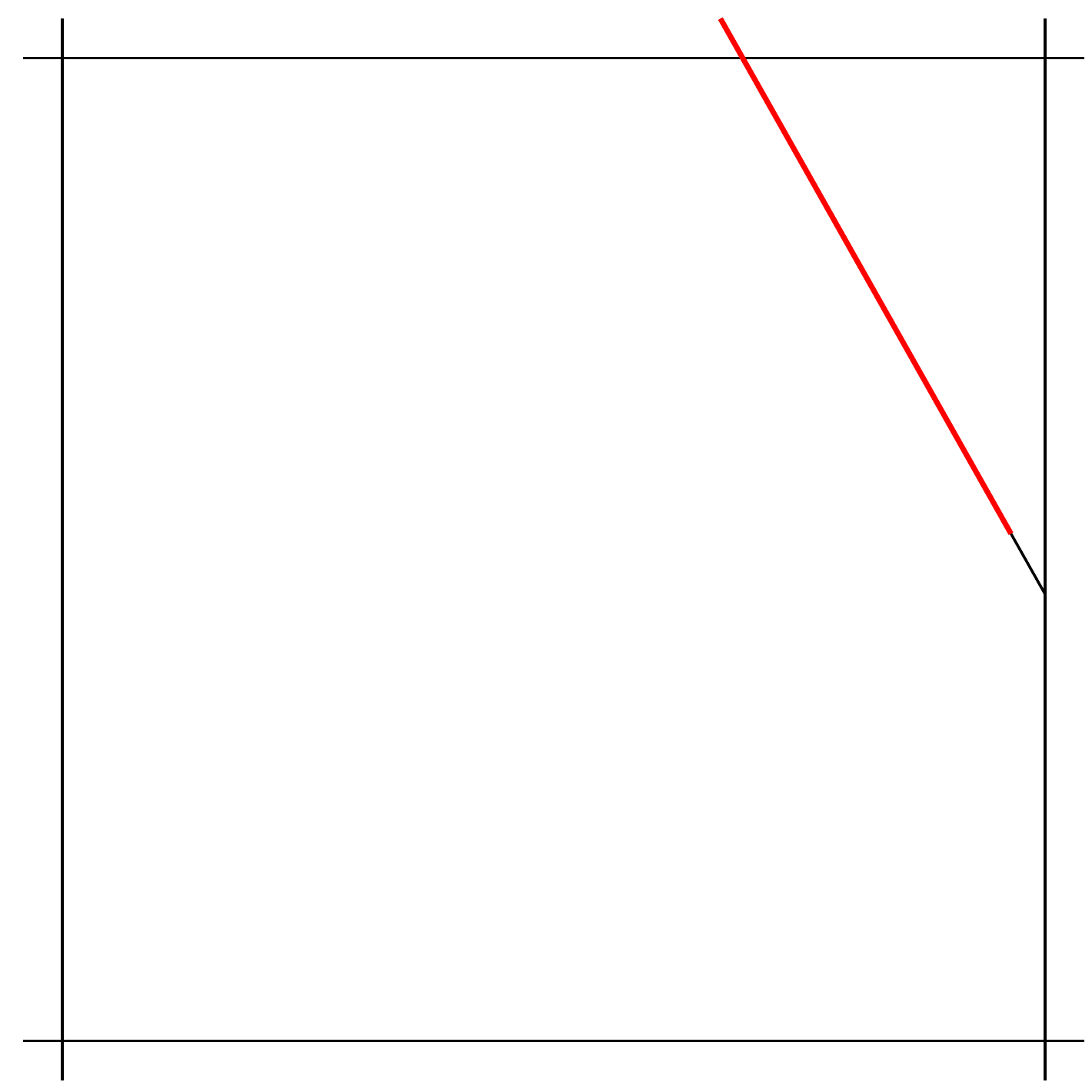}
\put (30, -10) {\texttt{case A}}
\end{overpic}
\qquad \qquad
\begin{overpic}[scale=0.17]{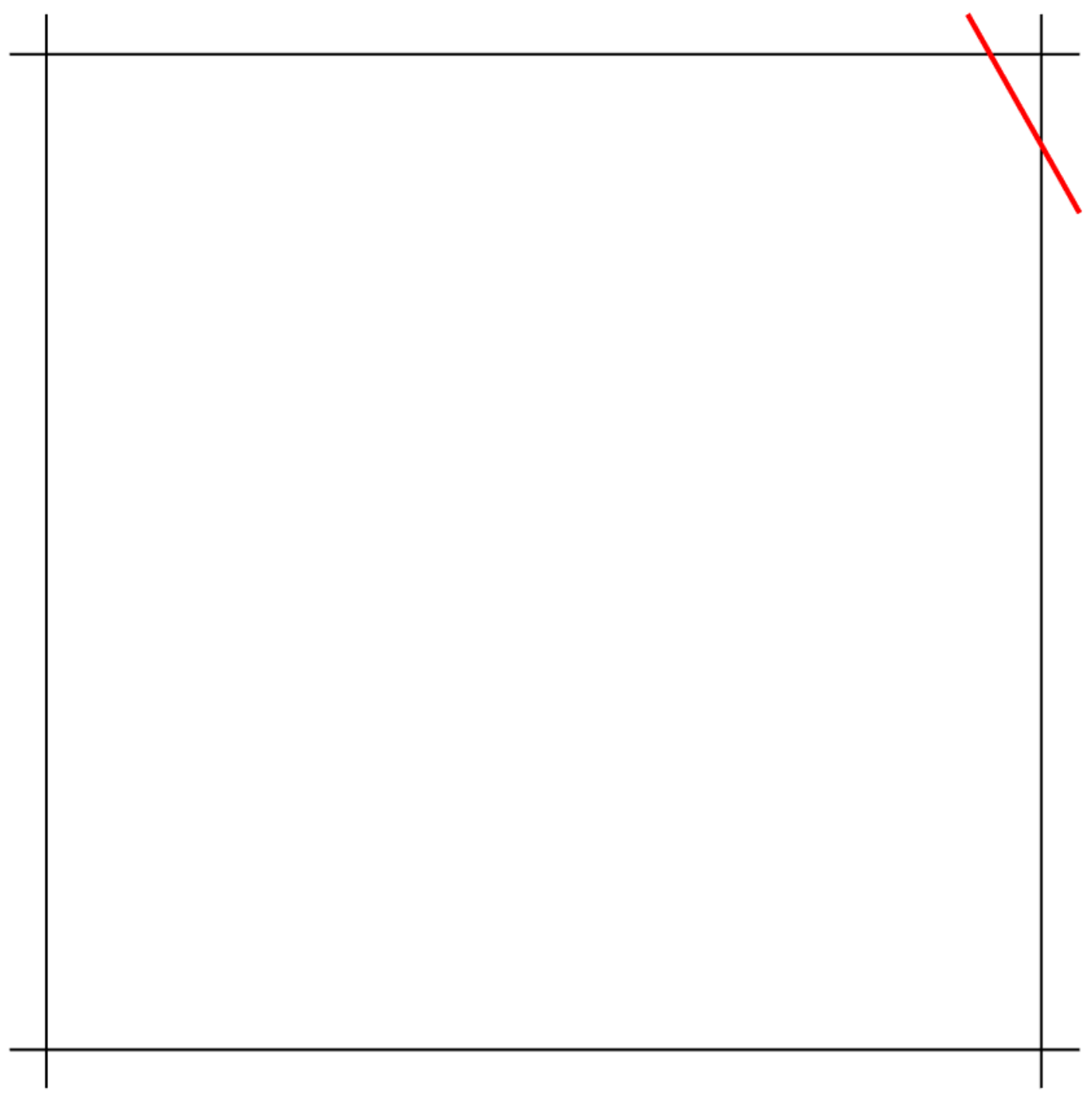}
\put (30, -10) {\texttt{case B}}
\end{overpic}
\qquad \qquad
\begin{overpic}[scale=0.17]{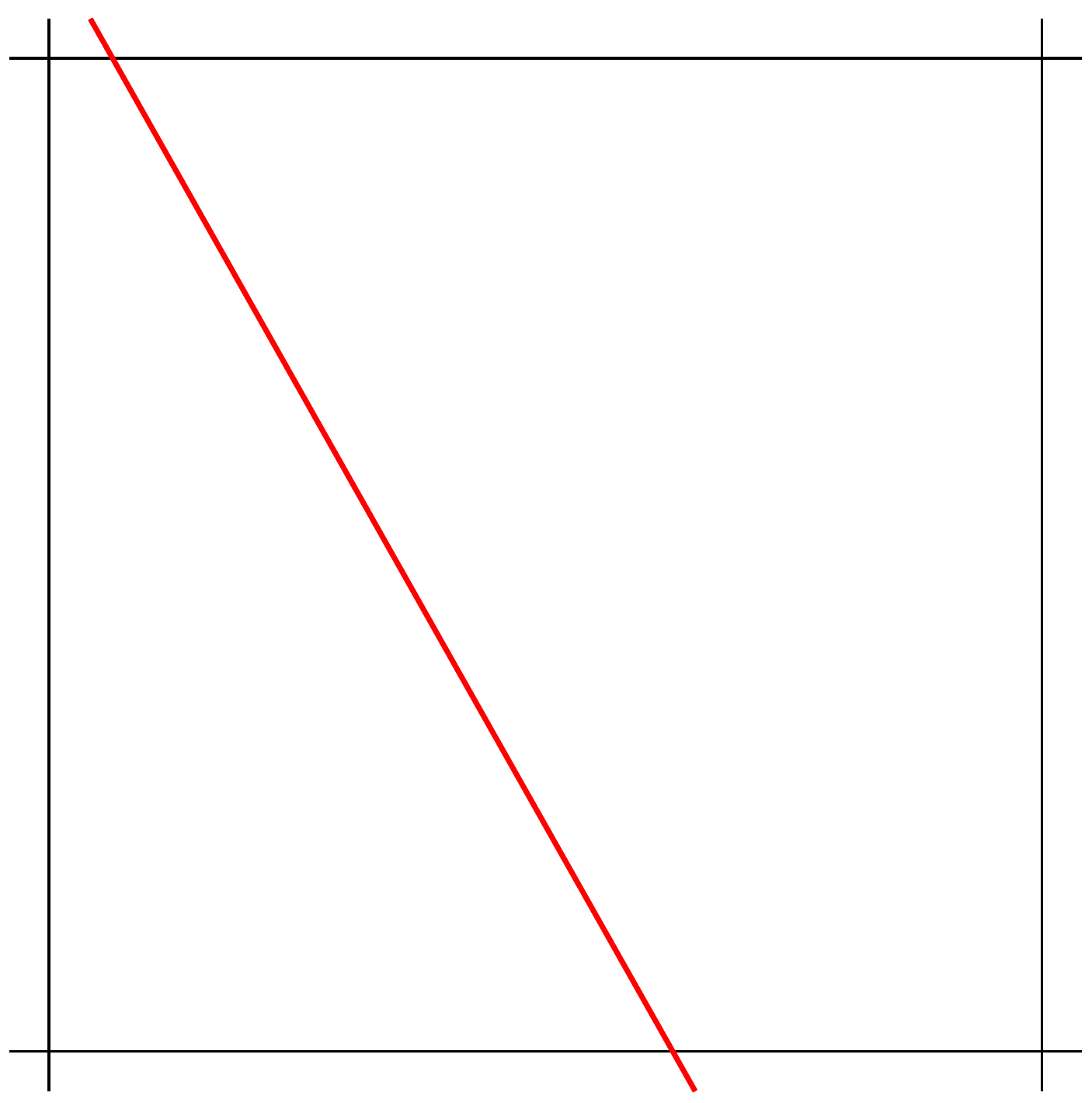}
\put (30, -10) {\texttt{case C}}
\end{overpic}
}
\caption{{Test 2.  Examples of mesh elements for $\theta=\pi/6-\texttt{1e-2}$. 
\texttt{case A}: small prolongation of the leaflet $\Gamma$ (even case).
\texttt{case B}: elements with different size.
\texttt{case C}: small edges.
}}
\label{fg:benchmark_meshes}
\end{figure}


\begin{table}[!h]
\centering
\small{
\begin{tabular}{lcccc}
\toprule
& 
\multicolumn{2}{c}{\texttt{trace}}
& 
\multicolumn{2}{c}{\texttt{dofi-dofi}}
\\
\midrule
   $\texttt{1/h}$   
&  $\texttt{k = 1}$   
&  $\texttt{k = 2}$      
&  $\texttt{k = 1}$   
&  $\texttt{k = 2}$     
\\
\midrule      
$\texttt{5}$       
&$\texttt{9.619392e-02}$         & $\texttt{5.263332e-02}$     
&$\texttt{4.155839e-02}$         & $\texttt{2.687218e-03}$ 
\\
$\texttt{9}$       
&$\texttt{6.693311e-02}$         & $\texttt{3.913659e-02}$          
&$\texttt{3.640182e-02}$         & $\texttt{1.051568e-02}$ 
\\
$\texttt{17}$       
&$\texttt{3.371117e-02}$         & $\texttt{1.796874e-02}$     
&$\texttt{1.337427e-02}$         & $\texttt{1.278502e-03}$     
\\
$\texttt{33}$       
&$\texttt{1.762575e-02}$         & $\texttt{9.453033e-03}$     
&$\texttt{6.364691e-03}$         & $\texttt{5.301562e-04}$   
\\
$\texttt{65}$       
&$\texttt{8.237521e-03}$         & $\texttt{5.364761e-03}$     
&$\texttt{2.894178e-03}$         & $\texttt{1.170867e-04}$        
\\
\midrule
$\texttt{4}$       
&$\texttt{1.506788e-01}$         & $\texttt{7.799135e-02}$     
&$\texttt{5.773616e-02}$         & $\texttt{1.766750e-02}$ 
\\
$\texttt{8}$       
&$\texttt{6.917897e-02}$         & $\texttt{4.419630e-02}$     
&$\texttt{2.872391e-02}$         & $\texttt{1.710762e-02}$ 
\\
$\texttt{16}$       
&$\texttt{3.555107e-02}$         & $\texttt{1.816131e-02}$     
&$\texttt{8.238265e-03}$         & $\texttt{3.096882e-03}$     
\\
$\texttt{32}$       
&$\texttt{1.826239e-02}$         & $\texttt{1.015554e-02}$     
&$\texttt{8.238265e-03}$         & $\texttt{1.808767e-03}$   
\\
$\texttt{64}$       
&$\texttt{1.000000e-02}$         & $\texttt{6.112618e-03}$     
&$\texttt{3.313301e-03}$         & $\texttt{1.449026e-03}$        
\\
\bottomrule
\end{tabular}
}
\caption{Test 2.  Error $|\theta^*-\theta_h^*|$ obtained $k=1$, $2$, with different values of the mesh size $h$ for the proposed benchmark problem; \texttt{trace} stabilization and \texttt{dofi-dofi} stabilization.}
\label{tab:btotal}
\end{table}

Table \ref{tab:btotal} reports the errors $|\theta^*-\theta_h^*|$ for different choices of the discretization parameter and for the two considered stabilizations. We observe that in both the even and odd cases the method converges to the exact solution. The convergence trend is more evident for the \texttt{trace} stabilization. However the \texttt{dofi-dofi} stabilization yields, at least for this test, better results.
We notice that, as expected, the method obtained with $k=2$ produces better performances in comparison with the $k=1$ scheme, nevertheless both schemes exhibit a linear rate of convergence. This is consistent with Theorem \ref{thm:finale} (see also Remark \ref{rem:eet}) and the low Sobolev regularity of the exact solution.

\medskip
\paragraph{Test 3: Performance w.r.t. $\theta$ and $h$ for the Navier-Stokes equations.}
\label{test2}
The aim of this test is to check the actual performance of the virtual element method for the full Navier-Stokes equations, using again the data given above and assuming as in Test 2 a linear law $\kappa(\theta)=\kappa_s\theta$  for the spring angular momentum. We vary $\kappa_s$ by several orders of magnitude, and we consider different refinements of the computational mesh.

Tables \ref{tabodd} and \ref{tabeven}  display the angular coordinates $\theta_h$ obtained by the virtual element discretization \eqref{eq:fsi_vem} and the bisection algorithm described in Subsection \ref{sub:bisection}, for odd and even values of $1/h$,  using both the \texttt{trace} and the \texttt{dofi-dofi} stabilizations.
We observe, as expected, that bigger rotation angles $\theta$ correspond to smaller values of $\kappa$, and larger values of the spring elastic modulus generate less pronounced displacements of the leaflet.
We do not have a reference solution for the present test, but we can appreciate that, for each choice of $\kappa_s$, the values of $\theta_h$ for different $h$ and $k=1,2$ are in mutual agreement  and seem to converge to a common  value.

\begin{table}[!h]

\centering
\small{
\begin{tabular}{ll*{6}{c}}
\toprule
& &  \multicolumn{3}{c}{\texttt{trace}} & \multicolumn{3}{c}{\texttt{dofi-dofi}}
\\
\midrule             
&   $\texttt{1/h}$  
&  $\kappa_s=\texttt{0.01}$   
&  $\kappa_s=\texttt{1.00}$       
&  $\kappa_s=\texttt{100}$   
&  $\kappa_s=\texttt{0.01}$    
&  $\kappa_s=\texttt{1.00}$   
&  $\kappa_s=\texttt{100}$      \\
\midrule
\multirow{5}*{$\texttt{k=1}$}                             
&$\texttt{5}$       
&$\texttt{1.38151}$         &$\texttt{0.27635}$         & $\texttt{0.00363}$
&$\texttt{1.31711}$         &$\texttt{0.15212}$         & $\texttt{0.00156}$          
\\
&$\texttt{9}$     
&$\texttt{1.31213}$         &$\texttt{0.22939}$         & $\texttt{0.00280}$
&$\texttt{1.28624}$         &$\texttt{0.16258}$         & $\texttt{0.00175}$          
\\
&$\texttt{17}$          
&$\texttt{1.29514}$         &$\texttt{0.20351}$         & $\texttt{0.00238}$
&$\texttt{1.29035}$         &$\texttt{0.17198}$         & $\texttt{0.00187}$         
\\
&$\texttt{33}$       
&$\texttt{1.29036}$         &$\texttt{0.18982}$         & $\texttt{0.00219}$
&$\texttt{1.29184}$         &$\texttt{0.17562}$         & $\texttt{0.00194}$        
\\
&$\texttt{65}$       
&$\texttt{1.28991}$         &$\texttt{0.18404}$         & $\texttt{0.00209}$
&$\texttt{1.29423}$         &$\texttt{0.17673}$         & $\texttt{0.00197}$
\\
\midrule
\multirow{5}*{$\texttt{k=2}$}                             
&$\texttt{5}$       
&$\texttt{1.28454}$         &$\texttt{0.23737}$         & $\texttt{0.00266}$          
&$\texttt{1.29268}$         &$\texttt{0.16902}$         & $\texttt{0.00185}$
\\
&$\texttt{9}$     
&$\texttt{1.28992}$         &$\texttt{0.20421}$         & $\texttt{0.00232}$
&$\texttt{1.29927}$         &$\texttt{0.17469}$         & $\texttt{0.00192}$          
\\
&$\texttt{17}$          
&$\texttt{1.28888}$         &$\texttt{0.19438}$         & $\texttt{0.00216}$
&$\texttt{1.29863}$         &$\texttt{0.17336}$         & $\texttt{0.00196}$          
\\
&$\texttt{33}$       
&$\texttt{1.28992}$         &$\texttt{0.18592}$         & $\texttt{0.00208}$          
&$\texttt{1.29593}$         &$\texttt{0.17845}$         & $\texttt{0.00198}$
\\
&$\texttt{65}$       
&$\texttt{1.29066}$         &$\texttt{0.18147}$         & $\texttt{0.00204}$          
&$\texttt{1.29497}$         &$\texttt{0.17804}$         & $\texttt{0.00199}$
\\
\bottomrule
\end{tabular}
}
\caption{Test 3.  Angular coordinates $\theta_h$ of the leaflet for degree of approximation $k=1$, $2$ with different values of the mesh size $h$ and different values of the spring torsional elastic modulus $\kappa_s$. \texttt{trace} and \texttt{dofi-dofi} stabilizations. Odd case.}
\label{tabodd}
\end{table}

\begin{table}[!h]
\centering
\small{
\begin{tabular}{ll*{6}{c}}
\toprule
& &  \multicolumn{3}{c}{\texttt{trace}} & \multicolumn{3}{c}{\texttt{dofi-dofi}}
\\
\midrule             
&   $\texttt{1/h}$  
&  $\kappa_s = \texttt{0.01}$   
&  $\kappa_s = \texttt{1.00}$       
&  $\kappa_s = \texttt{100}$   
&  $\kappa_s = \texttt{0.01}$    
&  $\kappa_s = \texttt{1.00}$   
&  $\kappa_s = \texttt{100}$      \\
\midrule
\multirow{5}*{$\texttt{k=1}$}                             
&$\texttt{4}$       
&$\texttt{1.40724}$         &$\texttt{0.28031}$         & $\texttt{0.00398}$
&$\texttt{1.33317}$         &$\texttt{0.13095}$         & $\texttt{0.00167}$          
\\
&$\texttt{8}$     
&$\texttt{1.34711}$         &$\texttt{0.23448}$         & $\texttt{0.00277}$
&$\texttt{1.28593}$         &$\texttt{0.15314}$         & $\texttt{0.00185}$          
\\
&$\texttt{16}$          
&$\texttt{1.29678}$         &$\texttt{0.20533}$         & $\texttt{0.00233}$
&$\texttt{1.29009}$         &$\texttt{0.16813}$         & $\texttt{0.00193}$         
\\
&$\texttt{32}$       
&$\texttt{1.29056}$         &$\texttt{0.19373}$         & $\texttt{0.00215}$
&$\texttt{1.29179}$         &$\texttt{0.17434}$         & $\texttt{0.00197}$        
\\
&$\texttt{64}$       
&$\texttt{1.28945}$         &$\texttt{0.18642}$         & $\texttt{0.00208}$
&$\texttt{1.29534}$         &$\texttt{0.17652}$         & $\texttt{0.00198}$
\\
\midrule
\multirow{5}*{$\texttt{k=2}$}                             
&$\texttt{4}$       
&$\texttt{1.28581}$         &$\texttt{0.23224}$         & $\texttt{0.00273}$          
&$\texttt{1.29015}$         &$\texttt{0.16291}$         & $\texttt{0.00205}$
\\
&$\texttt{8}$     
&$\texttt{1.29103}$         &$\texttt{0.20582}$         & $\texttt{0.00228}$
&$\texttt{1.29583}$         &$\texttt{0.17134}$         & $\texttt{0.00203}$          
\\
&$\texttt{16}$          
&$\texttt{1.28910}$         &$\texttt{0.19288}$         & $\texttt{0.00213}$
&$\texttt{1.29706}$         &$\texttt{0.17447}$         & $\texttt{0.00201}$          
\\
&$\texttt{32}$       
&$\texttt{1.28998}$         &$\texttt{0.18740}$         & $\texttt{0.00207}$          
&$\texttt{1.29590}$         &$\texttt{0.17619}$         & $\texttt{0.00200}$
\\
&$\texttt{64}$       
&$\texttt{1.29029}$         &$\texttt{0.18222}$         & $\texttt{0.00203}$          
&$\texttt{1.29492}$         &$\texttt{0.17800}$         & $\texttt{0.00200}$
\\
\bottomrule
\end{tabular}
}
\caption{Test 3. Angular coordinates $\theta_h$ of the leaflet for degree of approximation $k=1$, $2$ with different values of the mesh size $h$ and different values of the spring torsional elastic modulus $\kappa_s$. \texttt{trace} and \texttt{dofi-dofi} stabilizations. Even case.}
\label{tabeven}
\end{table}

We notice that the cutting procedure previously described may generate strongly anisotropic elements, particularly in the even case. For instance in the last case with $\kappa =  \texttt{100}$ and $k=\texttt{2}$ we get a solution $\theta_h = \texttt{0.002000}$, for the last refinement. Nevertheless we notice that the results in the given Tables demonstrate the robustness of Virtual Element technology in this respect.
We also observe that, at least for the proposed test, the \texttt{trace} stabilization yields a monotone trend of convergence to the solution. 
 
Finally, in Figs \ref{fig:test1_a}, \ref{fig:test1_b} and \ref{fig:test1_c} 
we show the plots of the  numerical velocity field and pressure field for
$\kappa_s = \texttt{0.01}, \texttt{0.1}, \texttt{1}$ obtained for the even case $1/h = 17$ with the first-order VEM scheme with \texttt{dofi-dofi} stabilization.

\begin{figure}[!h]
\center{
\includegraphics[scale=0.15]{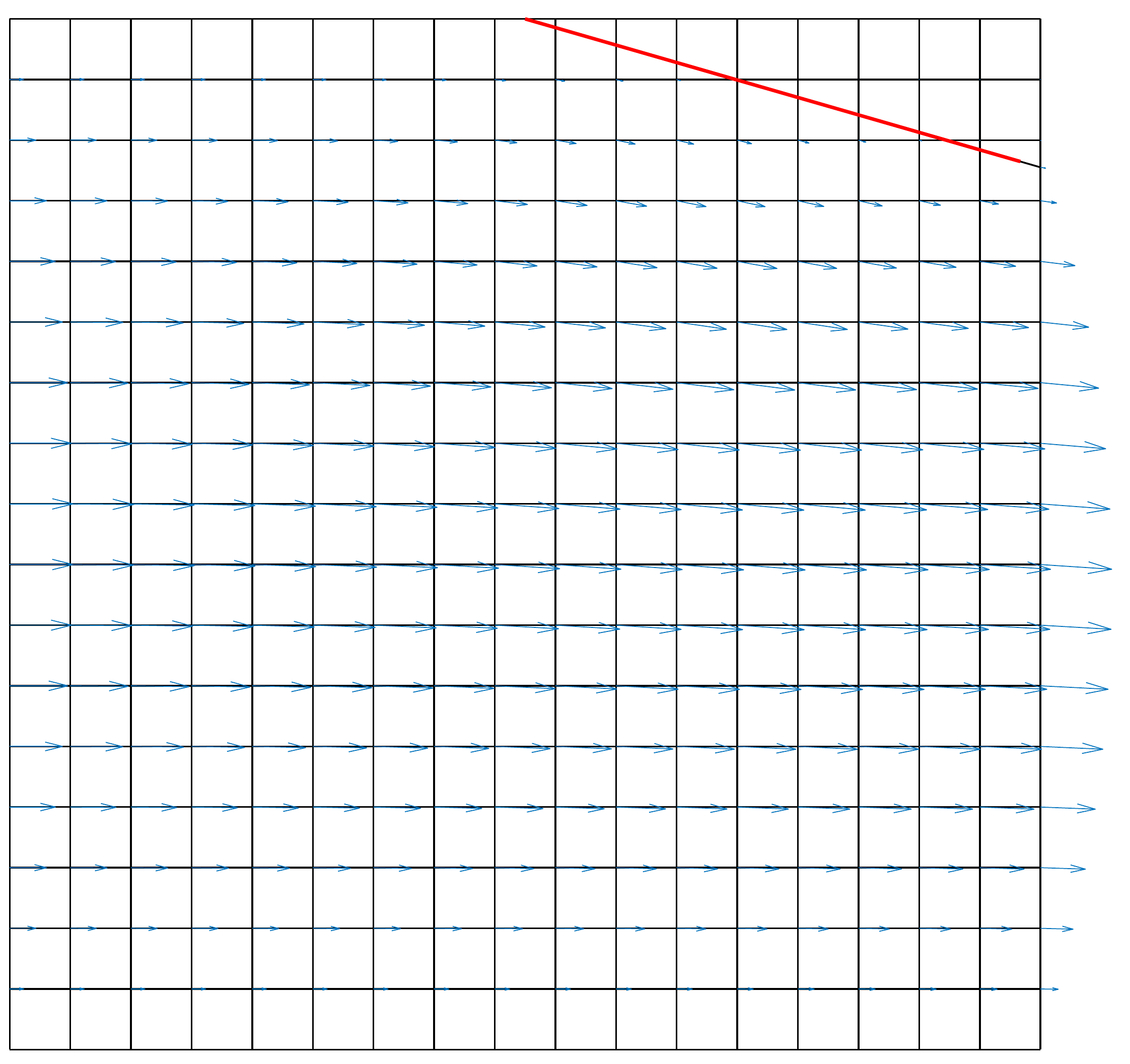}
\quad
\includegraphics[scale=0.12]{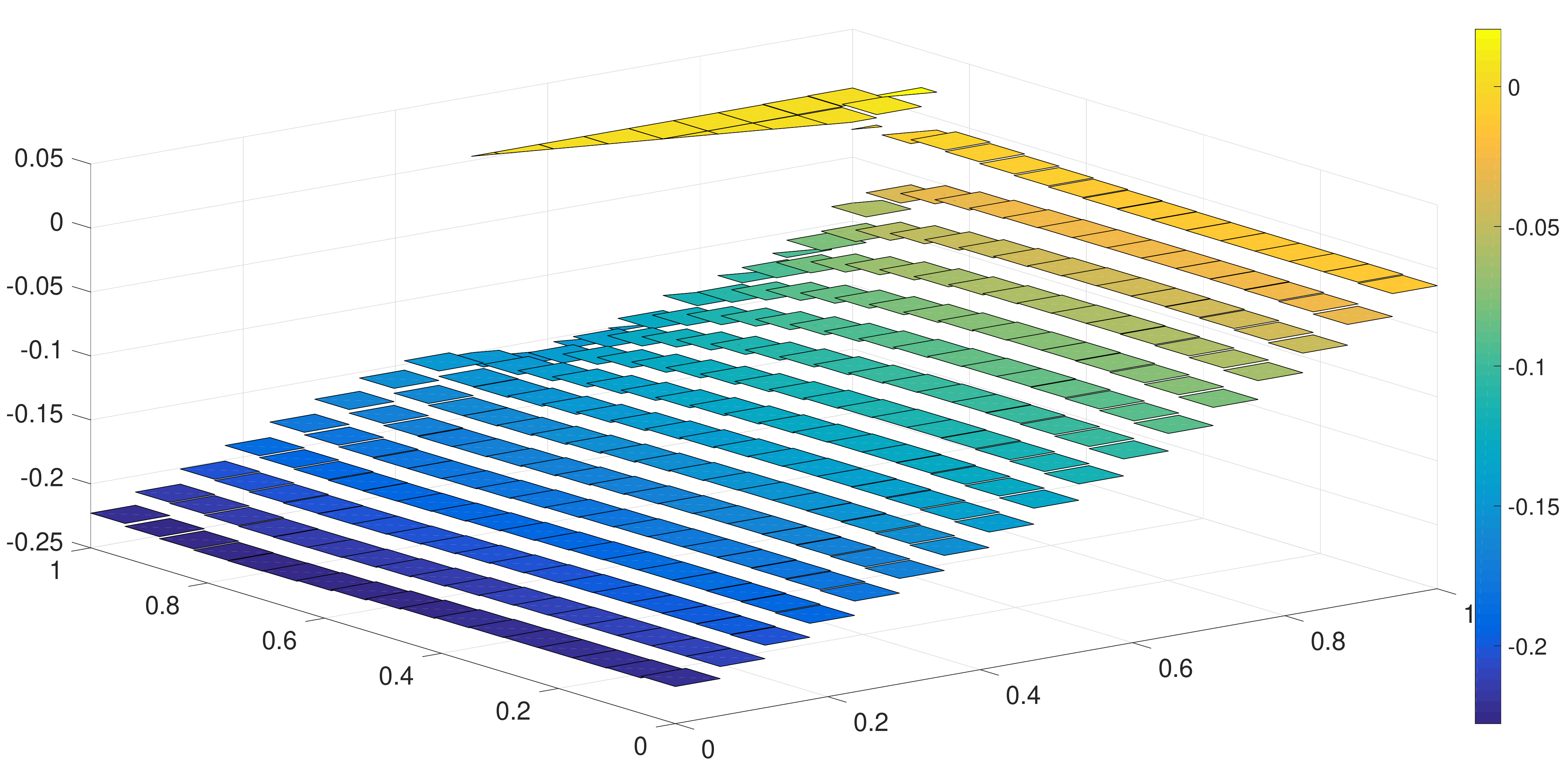}
\caption{{Test 3.  Velocity and pressure for the mesh  size $1/h=17$ and $\kappa_s =  \texttt{0.01}$ with $k=1$.}}
\label{fig:test1_a}
}
\end{figure}
\begin{figure}[!h]
\center{
\includegraphics[scale=0.15]{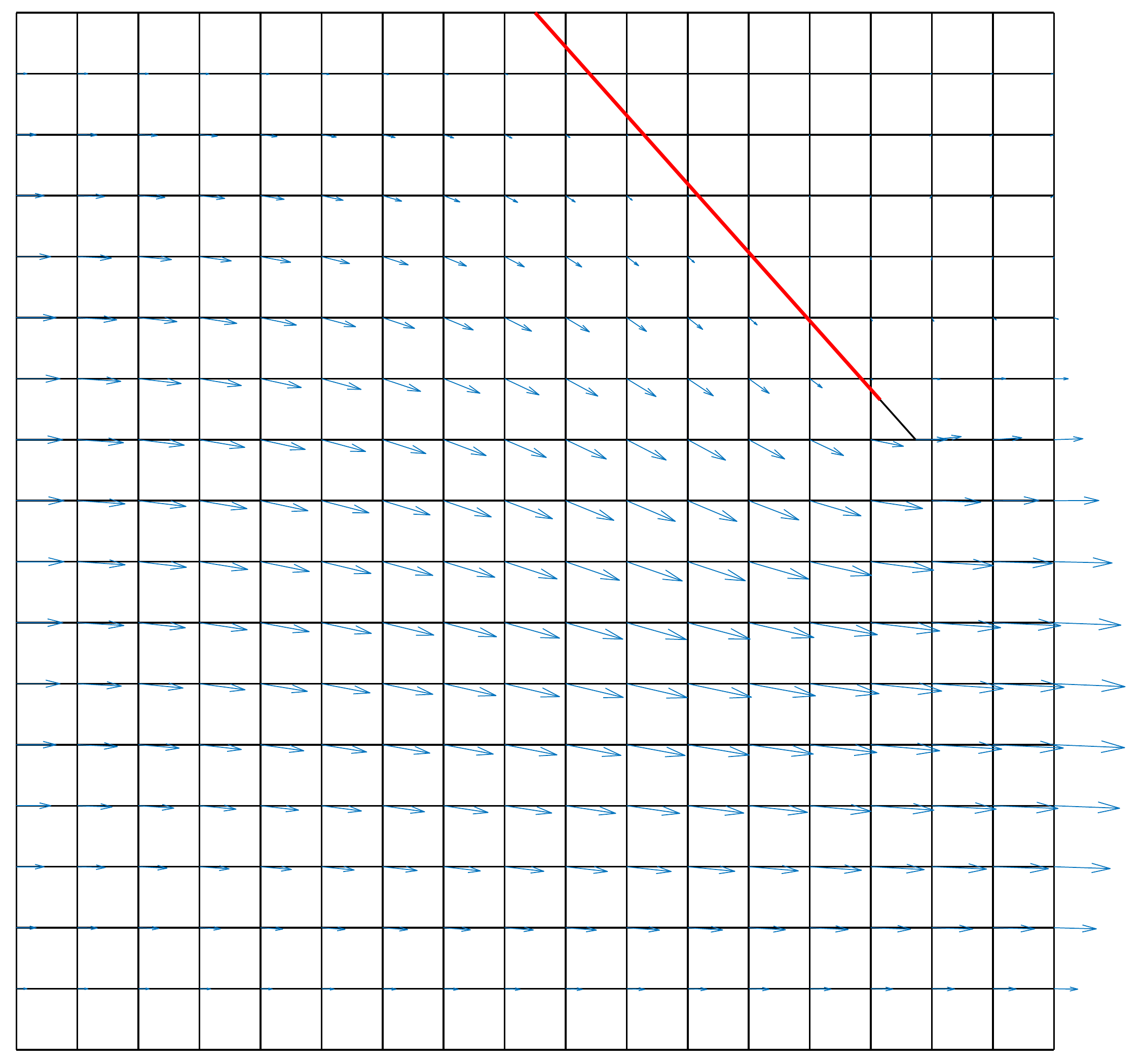}
\quad
\includegraphics[scale=0.12]{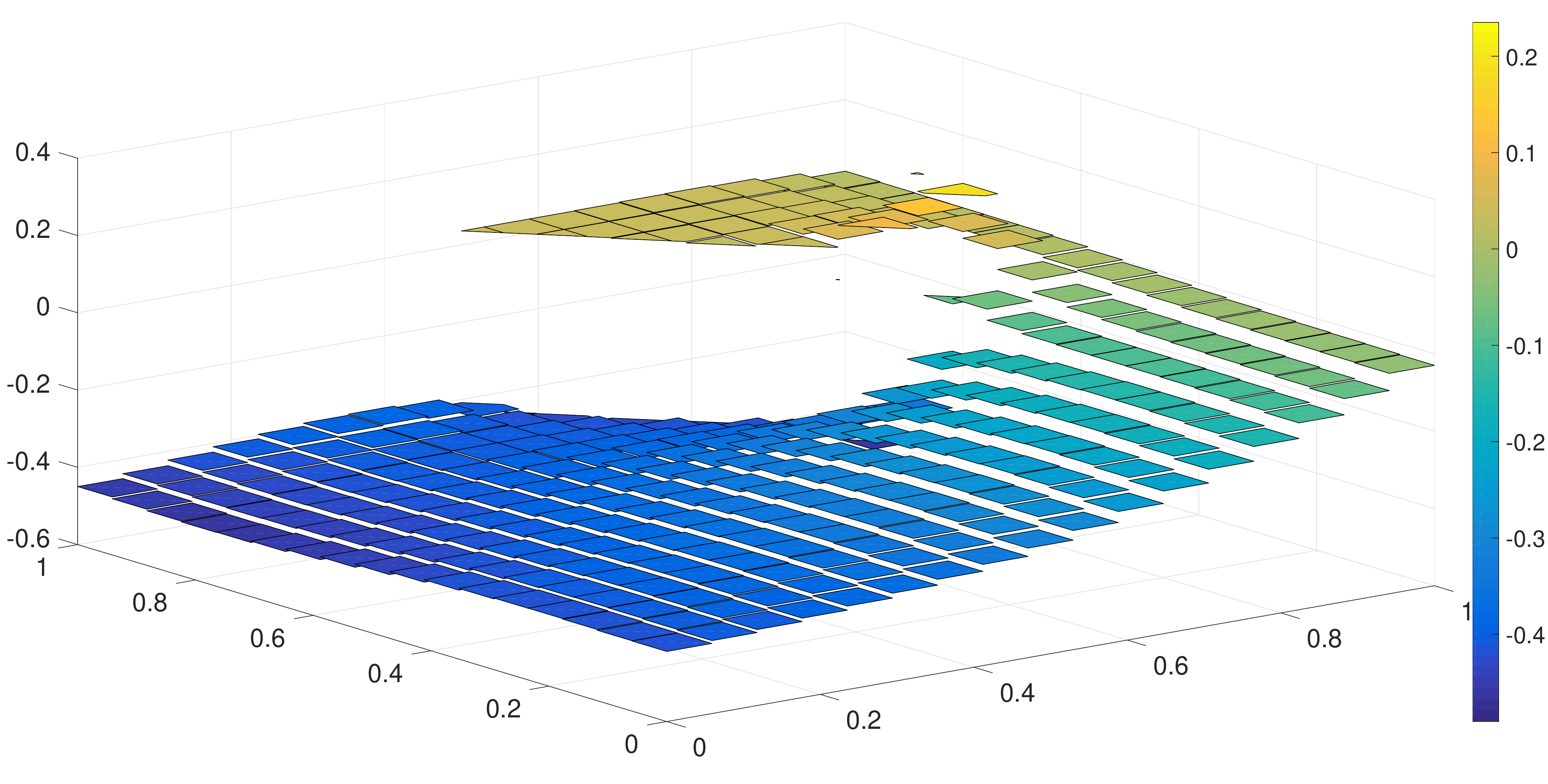}
\caption{{Test 3.  Velocity and pressure for the mesh  size $1/h=17$ and $\kappa_s =  \texttt{0.1}$ with $k=1$.}}
\label{fig:test1_b}
}
\end{figure}
\begin{figure}[!h]
\center{
\includegraphics[scale=0.15]{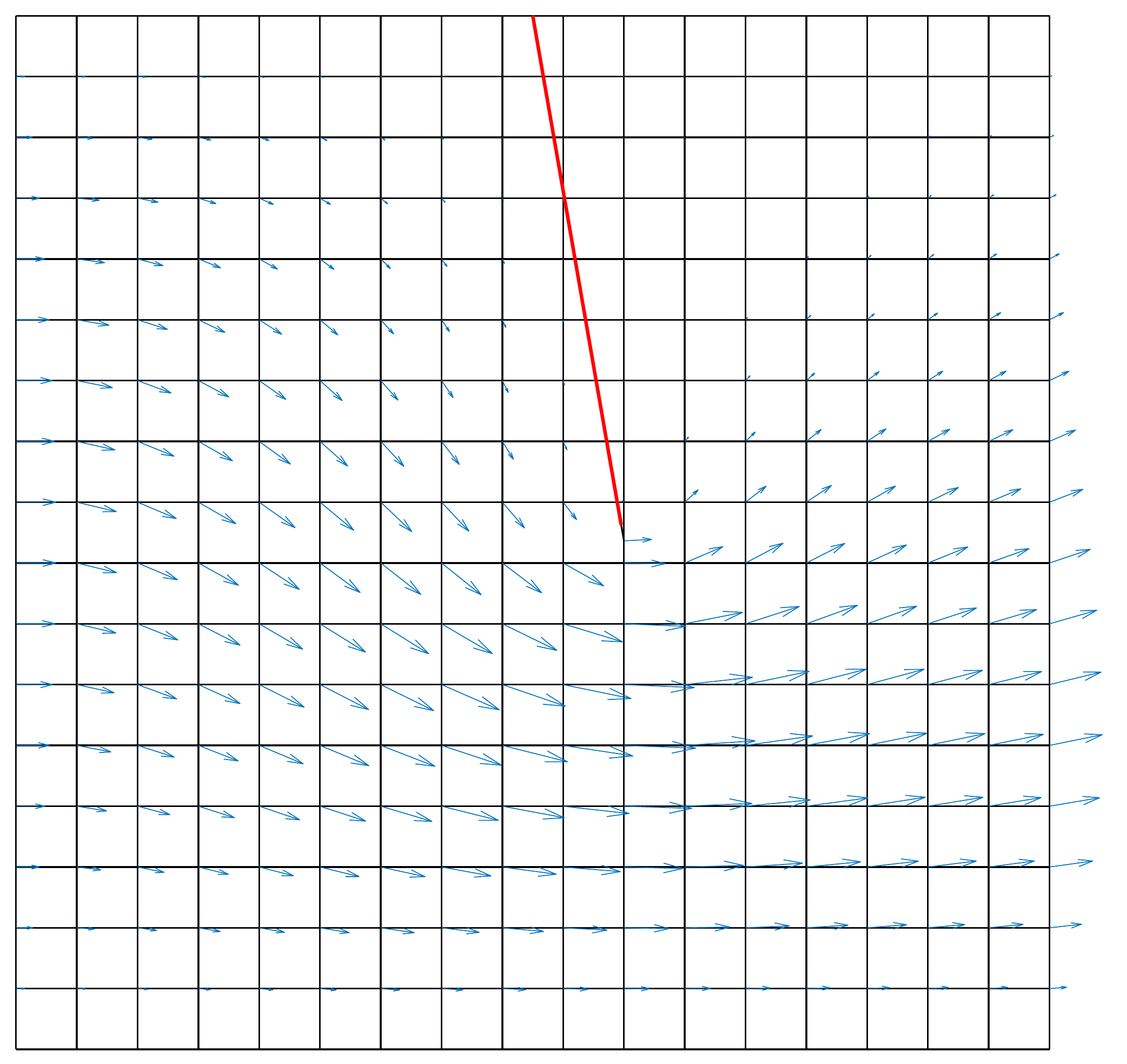}
\quad
\includegraphics[scale=0.12]{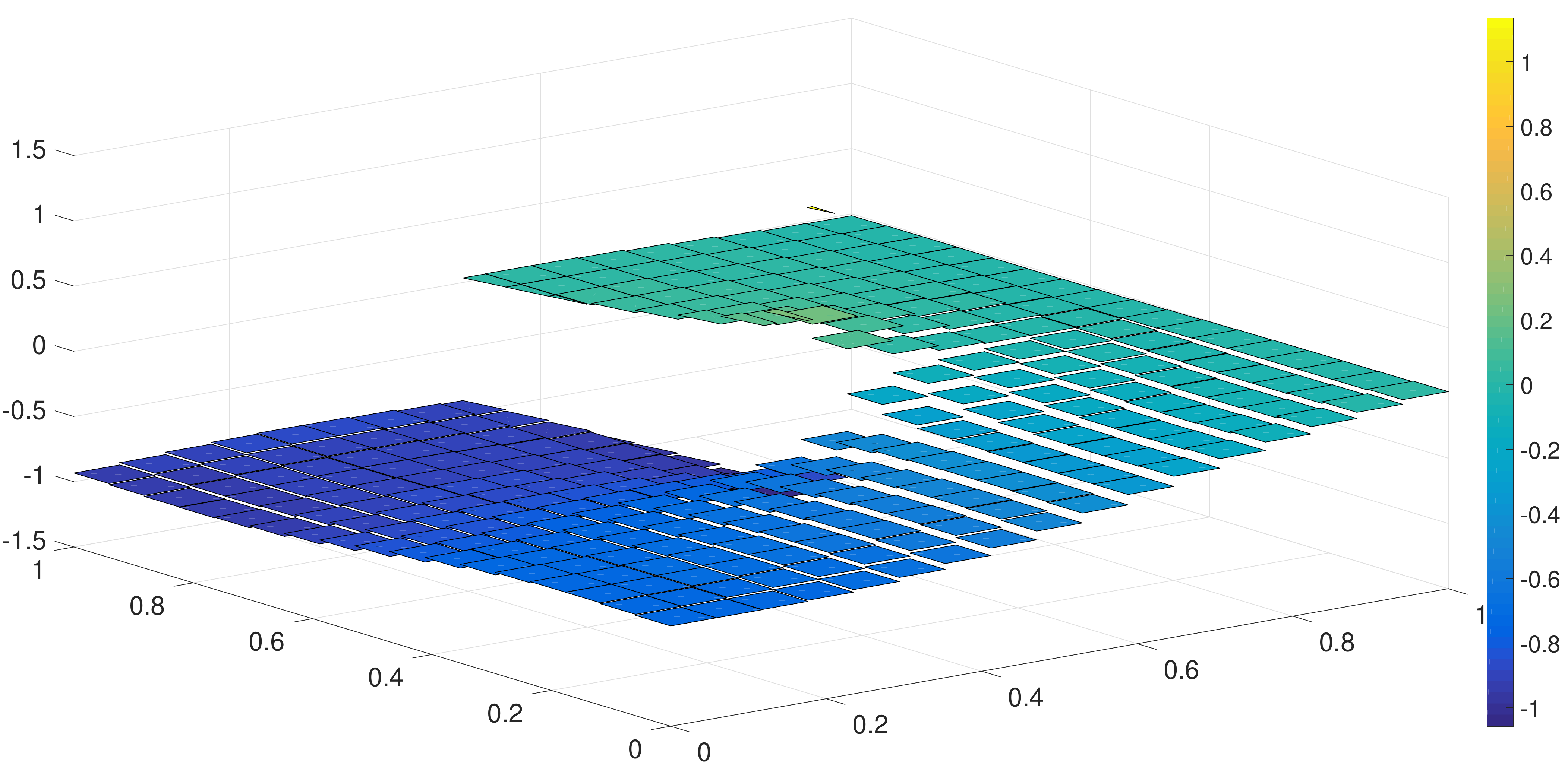}
\caption{{Test 3.  Velocity and pressure for the mesh  size $1/h=17$ and $\kappa_s =  \texttt{1}$ with $k=1$.}}
\label{fig:test1_c}
}
\end{figure}

%


\section{Conclusions}
\label{sec:conclusions}

We have investigated the equilibrium of a hinged rigid leaflet with an attached rotational spring, immersed in a stationary incompressible fluid within a rigid channel; we have assumed invariance in the transversal direction, leading to a two-dimensional geometry. Any equilibrium position corresponds to a balance between the spring angular momentum and the torque exerted by the fluid on the leaflet.
Our problem essentially depends upon two parameters, namely the angle $\theta$ of rotation of the leaflet around the hinge, and the thickness $\epsilon$ of the leaflet, which is allowed to take the value 0, thus reducing the leaflet to a segment.

The results in this paper concern the mathematical properties of the model on the one hand, and its numerical treatment on the other hand. Both theory and numerics rely on a variational formulation of the equilibrium problem, that we have derived first. Along the same lines, we have expressed torque in terms of bulk integrals involving an adjoint problem. 

Next, having in mind to assess the existence of equilibria by topological arguments, we have proven that the torque functional is continuous with respect to the angle $\theta$, in the whole interval of definition; this key result holds for both $\epsilon >0$ and $\epsilon=0$. In the former case (the `fat' leaflet), we have even established the differentiability of torque, by explicitly computing the shape derivative of the functional with respect to a rigid rotation, and showing its boundedness; this is a non-standard task, as we admit the presence of corners in the leaflet. These arguments do not extend to the case $\epsilon=0$ (the `thin' leaflet), although we are inclined to conjecture that differentiability of torque holds as well, and might be proven by different techniques. Nonetheless, we have established the continuity of the torque functional for the thin leaflet by relying on the uniform convergence of the continuous torque functionals for fat leaflets, as their thickness tends to 0.
With these results at hand, we have identified sufficient conditions on the spring angular momentum for the existence (and uniqueness) of equilibrium positions. 

On the numerical side, we have proposed a family of Galerkin {discretizations based on the Virtual Element Method (VEM) for the Stokes equations; the schemes differ in} the choice of the polynomial degree and the definition of the stabilization terms.  Our idea has been to exploit the capability of the VEM to handle arbitrary polygonal elements seamlessly, since elements of this type are precisely created when a thin leaflet cuts a background uniform grid of quadrilaterals. This feature is quite relevant for the design of an efficient computational method, as the search for an equilibrium requires to evaluate the torque for many different positions of the leaflet. 

We have derived quasi-optimal error estimates for the discrete torque functional,  in which the rate of decay is twice the one of the approximation error for the solution, and we have proposed a bisection algorithm for solving the discrete nonlinear equation. Grounded on these results, we have performed an extensive and detailed testing of our numerical methods. First of all, we have studied the discrete torque functional $\tau_h$ as a function of the angle $\theta$ for different discretization parameters (mesh size, polynomial degree, stabilization choice), in order to investigate the robustness of the scheme to the abrupt topological mesh changes that may happen at certain critical angles {due to the leaflet prolongation.} It turns out that the effect of such topological changes is more pronounced for the so-called  \texttt{dofi-dofi} stabilization form than for the  \texttt{trace} stabilization form (the former is, however, generally more accurate than the latter). Increasing the degree and/or refining the mesh, the jumps and bumps that appear in the graph of $\tau_h(\theta)$ at the critical angles are significantly reduced; for a fine mesh, as one would expect to use in applications, such features almost disappear.
We studied also numericaly the convergence of discrete equilibrium point to the exact one, for an ad-hoc problem with known equilibrium position. The {experimental rates are consistent with theory given} the irregular nature of the solution.
Finally, we illustrated the effect of degenerate elements in the inf-sup constant and conditioning of the system for a wide range of angles $\theta$. Geometric degeneracy is usually associated with manageable spikes in both quantities, with a better behaviour of the  \texttt{dofi-dofi} stabilization in terms of condition number. 

From the practical perspective, we conclude that, although there is some influence of the mesh quality on the results, the scheme is sufficiently robust and reliable. Considering the simplicity, and thus the efficiency, of the mesh cutting procedure when compared with other techniques, we believe our approach is viable. The extension to more complex problems will be the topic of future research.
 

\section*{Acknowledgments}
The authors are indebted to R.G. Dur\'an for bringing up [\cite{Duran:12}] to their attention.
LBdV and GV  were partially supported by the European Research Council through the H2020 Consolidator Grant (grant no. 681162) CAVE - Challenges and Advancements in Virtual Elements. 
LBdV was partially supported by the italian PRIN 2017 grant ``Virtual Element Methods: Analysis and Applications''.
CC carried out this work within the MIUR ``Progetto di Eccellenza 2018-2022'' 
(CUP: E11G18000350001). 
LBdV, CC and GV are members of the INdAM research group GNCS.
RHN was partially supported by NSF grants DMS-1411808 and DMS-1908267.
These supports are gratefully acknowledged.

\bibliographystyle{plain}
\bibliography{biblio}

\end{document}